\documentclass[a4paper, reqno]{amsart}

\usepackage{amssymb}

\usepackage{amsthm}
\usepackage{amsmath}
\usepackage{amscd}
\usepackage{amsfonts}
\usepackage{enumerate}
\usepackage{comment}
\usepackage{enumitem}
\usepackage{pdflscape}
\usepackage[all]{xy}

\newtheorem{thm}{Theorem}[section]
\newtheorem{lem}[thm]{Lemma}
\newtheorem{eg}[thm]{Example}
\newtheorem{prop}[thm]{Proposition}
\newtheorem{cor}[thm]{Corollary}
\newtheorem{rem}[thm]{Remark}
\newtheorem{defn}[thm]{Definition}

\numberwithin{equation}{section}

\newcommand{\BC}{\mathbb{C}}
\newcommand{\BN}{\mathbb{N}}
\newcommand{\BR}{\mathbb{R}}
\newcommand{\RP}{\mathbb{R}_+}

\newcommand{\CL}{\mathcal{B}}

\newcommand{\KI}{\mathfrak{I}}
\newcommand{\KJ}{\mathfrak{J}}

\newcommand{\CB}{\mathrm{CB}}

\newcommand{\CP}{\mathrm{CP}}
\newcommand{\os}{\mathrm{os}}
\newcommand{\rd}{\mathrm{d}}
\newcommand{\rn}{\mathrm{r}}
\newcommand{\Tr}{\mathrm{Tr}}
\newcommand{\tp}{\mathrm{t}}
\newcommand{\cb}{\mathrm{cb}}
\newcommand{\R}{\mathrm{Re}\ \!}

\newcommand{\Morc}{\mathrm{CCP}}
\newcommand{\sa}{\mathrm{sa}}
\newcommand{\id}{\mathrm{id}}
\newcommand{\MS}{\mathrm{S}}

\newcommand{\smnoind}{\noindent}

\begin{document}

%\begin{frontmatter}

\title{Duality for operator systems with generating cones}

\author{Yu-Shu Jia and Chi-Keung Ng}

\address[Yu-Shu Jia]{Chern Institute of Mathematics, Nankai University, Tianjin 300071, China.}
\email{1410035@mail.nankai.edu.cn}

\address[Chi-Keung Ng]{Chern Institute of Mathematics and LPMC, Nankai University, Tianjin 300071, China.}
\email{ckng@nankai.edu.cn}

\date{\today}


\begin{abstract}
Let $S$ be a complete operator system with a generating cone; i.e. $S_\sa = S_+ - S_+$. 
We show that there is a matrix norm on the dual space $S^*$, under which, and the usual dual matrix cone, $S^*$ becomes a dual operator system with a generating cone, denoted by $S^\rd$.
The canonical complete order isomorphism $\iota_{S^*}: S^* \to S^\rd$ is a dual Banach space isomorphism. 
Furthermore, we construct a canonical completely contractive weak$^*$-homeomorphism  $\beta_S: (S^\rd)^\rd\to S^{**}$, and verify that it is a complete order isomorphism. 

For a complete operator system $T$ with a generating cone and a completely positive complete contraction $\varphi:S\to T$, there is a weak$^*$-continuous completely positive complete contraction $\varphi^\rd:T^\rd \to S^\rd$ with $\iota_{S^*}\circ \varphi^* = \varphi^\rd \circ \iota_{T^*}$. 
This produces a faithful functor from the category of complete operator systems with generating cones (where morphisms are completely positive complete contractions) to the category of dual operator systems with generating cones (where morphisms are weak$^*$-continuous completely positive complete contractions). 

We define the notion of approximately unital operator systems, and verify that operator systems considered in \cite{CvS} and \cite{CvS2} are approximately unital. 
If $S$ is approximately unital, then $\iota_{S^*}:S^* \to S^\rd$ is an operator space isomorphism and $\beta_S: (S^\rd)^\rd\to S^{**}$ is a complete isometry. 
We will also establish that the restriction of the faithful functor $(S,T,\varphi)\mapsto (T^\rd, S^\rd, \varphi^\rd)$ to the category of approximately unital complete operator systems is both full and injective on objects. 

On our way to the above, we find that for a complete $^*$-operator space $X$ equipped with a norm closed matrix cone, the bidual $X^{**}$ is a unital operator system (which show that $X$ is an operator system) if and only if $X$ approximately unital. 
This gives an abstract characterization of complete operator systems whose biduals being unital.
\end{abstract}

\maketitle

\section{Introduction and Notations}\label{sec:Introd}

%\medskipI

A duality framework for (not necessarily unital) operator systems was introduced by the second named author in \cite{Ng-dual-OS} (see also related results in \cite{Han}). 
Most of the statements in \cite{Ng-dual-OS} apply to a class of operator systems called ``dualizable operator systems''. 
However, one draw back for this framework is that it is not easy to determine when an operator system is ``dualizable'', except for unital operator systems and $C^*$-algebras. 
Therefore, a natural question is whether there exists a property for operator systems which is easier to verify, and under which the duality is good enough. 
It is the first aim of this article. 
The second aim of this article is to apply the duality results to a class of operator systems that has recently been introduced by Connes and van Suijlekom in \cite{CvS} and \cite{CvS2} (see also \cite{Faren, torlerance, van, NC}). 

%\medskipI

For the first aim, we will show that the duality is nice for (norm) complete operator systems having generating cones. 
More precisely, let $S$ be an operator system (i.e. $S$ is a self-adjoint subspace of some $\CL(H)$ that may not contain the identity, equipped with the induced matrix norm and the induced matrix cone).
We define a matrix semi-norm on its dual space $S^*$ by 
$$\|f\|^\rn:= \sup \{\|f^{(m)}(u)\|: u\in M_m(S)_+; \|u\|\leq 1;  m\in \BN\},$$
for every $f\in M_n(S^*) = \CB(S,M_n)$. It will be shown in Theorem \ref{thm:reg-dual}(b) that the quotient of $S^*$ by the null space of this matrix semi-norm, equipped with the  matrix norm induced by $\|\cdot\|^\rn$ as well as the quotient of the dual matrix cone on $S^*$,  is a weak$^*$-dense operator subsystem of a dual operator system $S^\rd$ (a \emph{dual operator system} is a weak$^*$-closed subspace of some $\CL(H)$, equipped with the induced matrix norm, the induced matrix cone and the induced weak$^*$-topology; see Definition \ref{defn:dual-MOS}). 
The canonical map 
$$\iota_{S^*}: S^* \to S^\rd$$ 
is a completely positive complete contraction, and this map is also weak$^*$-continuous when
\begin{quote}
	we equip $S^* = \tilde S^*$ with the weak$^*$-topology $\sigma(S^*,\tilde S)$, where $\tilde S$ is the completion of $S$.
\end{quote}
Note that we will always consider this weak$^*$-topology $\sigma(S^*,\tilde S)$ on $S^*$. 
As will be seen in Theorem \ref{thm:prop-os-dual}(a), the dual operator system $S^\rd$ is \emph{universal}, in the sense that if $V$ is a dual operator system and $\varphi:S^*\to V$ is a weak$^*$-continuous completely positive complete contraction, then one can find a unique weak$^*$-continuous completely positive complete contraction $\overline{\varphi}: S^\rd \to V$ satisfying 
$$\varphi = \overline{\varphi}\circ \iota_{S^*}.$$ 
Since $S^{**}$ is always a dual operator system, the above universal property implies that the dual map $\iota_{S^*}^*:(S^\rd)^*\to S^{**}$ induces a weak$^*$-continuous completely positive complete contraction 
$$\overline{\iota_{S^*}^*}: (S^\rd)^\rd\to S^{**}.$$ 
On the other hand, any completely positive complete contraction $\varphi$ from $S$ to another operator system $T$ will induce a weak$^*$-continuous completely positive complete contraction $\varphi^\rd: T^\rd \to S^\rd$ (see Relation \eqref{eqt:phi-d}). 
Note, however, that there exist non-zero operator systems $S$ for which $S^\rd = \{0\}$ (see Example \ref{eg:not-unique}).

%\medskipI

Let us recall that (see also Definition \ref{defn:dual})
\begin{quote}
a complete operator system $S$ is said to be \emph{dualizable} if one can find a (not necessarily isometric) complete embedding from $S^*$ to some $\CL(H)$ such that the matrix cone on $S^*$ is the one induced from this embedding. 
\end{quote}
It was shown in \cite{Ng-dual-OS}  (see also Lemma \ref{lem:dualizable}) that if $S$ is dualizable, then $\iota_{S^*}$ is a Banach space isomorphism (in fact, an operator space isomorphism). 
Moreover, $(S,T,\varphi) \mapsto (T^\rd,S^\rd,\varphi^\rd)$ is a faithful contravariant functor from the category of dualizable operator systems to dual operator systems. 
It was also shown in \cite{Ng-dual-OS} that if $S$ is unital, then $S$ is dualizable and $\overline{\iota_{S^*}^*}:(S^\rd)^\rd\to S^{**}$ is a completely isometric complete order isomorphism.

%\medskipI

In this article, we will extend these results. 
Notice that a dualizable operator system is by definition complete and will always have a generating cone; i.e. its cone will linearly span the whole space (see Lemma \ref{lem:dualizable}).  

%\medskipI

In the following, we assume that $S$ is an operator system such that its completion $\tilde S$ has a generating cone (here, the matrix cone of $\tilde S$ is the closure of the matrix cone of $S$). 
We will show in Proposition \ref{prop:iota-bdd-below} that
\begin{itemize}
	\item $\iota_{S^*}$ is a complete order isomorphism, a Banach space isomorphism and a weak$^*$-homeomorphism;
	
	\item $S^\rd$ also has a generating cone.
\end{itemize}
We will also obtain in Theorem \ref{thm:comp-isom-tau} that 
\begin{itemize}
	
	\item $\overline{\iota_{S^*}^*}$ is always a complete order isomorphism and a weak$^*$-homeomoprhism;

	\item $\overline{\iota_{S^*}^*}$ is completely isometric if and only if there is a complete contraction from $S$ to $(S^\rd)^\rd$ compatible with both $\overline{\iota_{S^*}^*}$ and the canonical map from $S$ to $S^{**}$.
\end{itemize}
Furthermore, the following will be established in Corollary \ref{cor:dual-funct}:
\begin{itemize}
	\item the functor $(S,T,\varphi) \mapsto (T^\rd,S^\rd,\varphi^\rd)$ from the category of complete operator systems with generating cone to dual operator systems with generating cone is faithful. 
\end{itemize}

%\medskipI

Concerning our second aim, Connes and van Suijlekom associated a non-complete operator system with every tolerance relation on an arbitrary metric space (see \cite{CvS2}  and also \cite{CvS}). 
One of the tools they employed in \cite{CvS2} to study this kind of operator systems is  the notion of approximate order units that are  ``matrix norm defining''. 
In Definition \ref{defn:weakly-matrix}, 
we will introduce a weaker notion  (see Theorem \ref{thm:unital-bidual}(c)) called ``weakly matrix norm defining'' increasing net, and we will say that an operator system is \emph{approximately unital} if it admits a ``weakly matrix norm defining'' increasing net. 
The following will be proved:
\begin{itemize}
	\item If an operator system is approximately unital, then so is its completion (see Theorem \ref{thm:unital-bidual}(a)). 
	
	\item Operator systems associated with tolerance relations are approximately unital (see Corollary \ref{cor:toler-rel}(a)). 
	
	\item A complete operator system $S$ is approximately unital if and only if $S^{**}$ is unital (see Theorem \ref{thm:unital-bidual}(a)). 
	
	\item Every approximately unital complete operator system is dualizable (see Proposition \ref{prop:approx-unit}(a)). 
	
	\item If $S$ is approximately unital, then $\overline{\iota_{S^*}^*}$ is completely isometric (see Proposition \ref{prop:approx-unit}(c)). 
	
	\item The functor $(S,T,\varphi) \mapsto (T^\rd,S^\rd,\varphi^\rd)$ from the category of approximately unital complete operator systems to the category of dual operator systems with generating cones is faithful, full and injective on objects  (see Theorem \ref{thm:dual-funct-approx-unit}). 
	
	\item If $S$ is an operator system and $T\subseteq S$ is a subsystem which admits an increasing net that weakly defines the matrix norm on $S$, then every completely positive complete contraction from $T$ to $M_n$ extends to a completely positive complete contraction from $S$ to $M_n$ (see Proposition \ref{prop:approx-unit-subsys}(b)). 
\end{itemize}

%\medskipI

On our way to the results stated above, we also obtain the following facts concerning a $^*$-operator space $X$ equipped with a closed matrix cone (in particular, when $X = S$ or $X = S^*$ for an operator system $S$):  

\begin{itemize}
	\item For every $n\in \BN$, the assignment $\phi\mapsto \Upsilon_\phi$ is an isometric order isomorphism from $M_n(X^*)_\sa = \CB(X,M_n)_\sa$ onto $(M_n(X)^*_\sa, \gamma_X^*)$, where 
	\begin{align*}
\gamma_X(x):= \inf\big\{\Tr_n(\beta^2)\|w\|: x=\beta w&\beta; w\in M_n(X);\\ 
& \beta\in (M_n)_+ \text{ is invertible}\big\},
	\end{align*}
	and $\Upsilon_\phi(x) :={\sum}_{k,l=1}^n \phi(x_{k,l})_{k,l}$ when $\phi\in M_n(X^*)$ and $x\in M_n(X)$ (see Theorem \ref{thm:bdd-pos-CCP}). %The form of $\gamma_X$ here seems different from the form in Theorem \ref{thm:bdd-pos-CCP}, yet they are actually  the same (see the aregument of \cite[Lemma 6.8]{effros1} for more details).
	
	\item For each $n\in \BN$, there exists an isometry order isomorphism from $M_n(X)^{**}$ onto $M_n(X^{**})$ sending the image of $M_n(X)$ in $M_n(X)^{**}$ onto the image of $M_n(X)$ in  $M_n(X^{**})$ (see Theorem \ref{thm:bidual-MOS}). 
	
	\item If $X$ is complete, then $X^{**}$ is a unital operator system (which implies that $X$ is an operator system) if and only if $X$ admits a weakly matrix norm defining increasing net (see Theorem \ref{thm:unital-bidual}(a)). 
\end{itemize}

%\medskipI

The third point above is an abstract characterization for $^*$-vector spaces equipped with complete matrix norms and norm closed matrix cones whose biduals are unital operator systems. 

%\medskipI

In the remainder of this section, we will set some basic notation. 

%\medskipI

For (real or complex) vector spaces $X$ and $Y$, we denote by $L(X,Y)$ the set of (respectively, real or complex) linear maps from $X$ to $Y$.
When $X$ and $Y$ are normed spaces, we will use $\CL(X,Y)\subseteq L(X,Y)$ to denote the set of bounded linear maps, and put $\CL(X):=\CL(X,X)$. 
Furthermore, we denote by 
$X^*$ the dual space of $X$ (i.e., $X^*=\CL(X,\BC)$ in the complex case and  $X^*=\CL(X,\BR)$ in the real case).
Note that we will always consider complex linearity for maps between complex vector spaces.  
Moreover, 
\begin{equation}\label{eqt:notat-B-X}
	B_X:=\left\lbrace x\in X:\left\|x\right\|\leq 1\right\rbrace,  \text{ and }\tilde X \text{ is always the completion of }X.
\end{equation}

%%\medskipI

\begin{defn}\label{defn:order}
	(a) A complex vector space $X$ is a \emph{$^*$-vector space} if it is equipped with an involution $^*:X\to X$. 
	Moreover, a complex linear map $\varphi:X\to Y$ between $^*$-vector spaces is said to be \emph{self-adjoint} if $\varphi^* = \varphi$, where
	\begin{equation}\label{eqt:def-adj}
		\varphi^*(x):= \varphi(x^*)^* \qquad (x\in X).
	\end{equation}
	
	\smnoind
	(b) A $^*$-vector space $X$ is called an \emph{ordered vector space} if 
	$$X_\sa:=\{x\in X:x^* = x \}$$ 
	is a real ordered vector space, equipped with a cone $X_+$ (see \ref{sec:ord-vs}).

	\smnoind
	(c) Let  $X$ and $Y$ be complex ordered vector spaces.
	A self-adjoint complex linear map $\varphi:X\to Y$ is called 
a \emph{positive map} (respectively, an \emph{order monomorphism} or an \emph{order isomorphism}) if $\varphi|_{X_\sa}:X_\sa \to Y_\sa$ is positive (respectively, an order monomorphism  or an order isomorphism), in the sense of Definition \ref{defn:pos}. 
\end{defn}

%%\medskipI

We write $M_{m,n}$ for the complex vector space of complex $m\times n$-matrices, and put $M_n:= M_{n,n}$. 
For a complex vector space $X$, we denote 
$$M_{m,n}(X) := M_{m,n}\otimes X,$$ 
where $\otimes$ is the tensor product over $\BC$. 
One may think of an element $x$ in $M_{m,n}(X)$ as a $m\times n$-matrix $[x_{k,l}]_{k,l}$ with $x_{k,l}\in X$. 
By the universal property of tensor products, a bilinear map from $X\times Y$ to $Z$ will induce a bilinear map from $M_{k,m}(X)\times M_{m,n}(Y)$ to $M_{k,n}(Z)$. 
We set 
$$x\oplus y:= 
\begin{pmatrix}
	x&0\\
	0&y
\end{pmatrix}\in M_{m+n}(X) \qquad (x\in M_m(X); y\in M_n(X)).$$
We will identify $M_n(X)\subseteq M_{n+1}(X)$ and $M_{m,n} \subseteq M_{\max\{m,n\}}$, by putting elements of $M_n(X)$ and $M_{m,n}$ into the upper left corners of $M_{n+1}(X)$ and $M_{\max\{m,n\}}$, respectively. 
Denote
\[M_\infty(X):={\bigcup}_{n\in \mathbb{N}}M_n(X).\]
Notice that $M_\infty := \bigcup_{n\in \BN} M_n$ is an algebra and $M_\infty(X) = M_\infty \otimes X$ is a $M_\infty$-bimodule. 
If $X$ and $Y$ are complex vector spaces and $\varphi:X\to Y$ is a complex linear map, we define 
\begin{equation*}\label{eqt:def-varphi-infty}
	\varphi^{(\infty)}:= \id_{M_\infty}\otimes \varphi: M_\infty(X)\to M_\infty(Y)
\end{equation*}
and put $\varphi^{(n)}:=\varphi^{(\infty)}|_{M_n(X)}$.

\begin{comment}
%%\medskipI

For a complex algebra $A$, an  $A$-bimodule $E$ is said to be \emph{essential} if $E = A\cdot E \cdot A$. 
Obviously, $M_\infty(X)$ is an essential $M_\infty$-bimodule. 
Conversely, it can be shown that if $E$ is an essential $M_\infty$-bimodule, then there is a vector space $X$ such that $E = M_\infty(X)$. 
Furthermore, for every linear map $\varphi:X\to Y$, the map $\varphi^{(\infty)}$ is a $M_\infty$-bimodule map, and every $M_\infty$-bimodule map from $M_\infty(X)$ to $M_\infty(Y)$ is of this form. 
\end{comment}

%%\medskipI

For a complex normed space $X$ and $f=[f_{k,l}]_{k,l}\in M_\infty(X^*)$, we will see $f$ as an element in $\CL(X,M_\infty)$ with 
\[f(x):=[f_{k,l}(x)]_{k,l}\qquad (x\in X).\]
Through this, we will identify implicitly
$$M_n(X^*) = \CL(X,M_n) \qquad (n\in\BN\cup \{\infty\}).$$
In a similar way, we will identify $M_n(X)$ with the subspace of $\CL(X^*,M_n)$ consisting of $\sigma(X^*,X)$-continuous maps. 

%%\medskipI

Now, let us recall the following abstract description of operator spaces (see \cite[Proposition 2.3.6]{OPS} and \cite[Theorem 2.1]{Ng-reg-mod}).

%%\medskipI

\begin{defn}\label{defn:oper-sp}
	(a) A complex vector space $X$ is called an \emph{operator space} if there is a norm $\|\cdot\|$ on $M_\infty(X)$ (called the \emph{matrix norm} on $X$) satisfying
	$$\|\alpha x\beta + \gamma y \delta\| \leq \|\alpha^*\alpha+\gamma^*\gamma\|^{1/2}\|\beta^*\beta+\delta^*\delta\|^{1/2} \max\{\|x\|, \|y\|\},$$
	whenever $x,y\in M_\infty(X)$ and $\alpha,\beta, \gamma, \delta\in M_\infty$.
	
	\smnoind
	(b) We say that $X$ is a \emph{$^*$-operator space} if it is both a $^*$-vector space and an operator space such that the induced involution on $M_\infty(X)$ is isometric. 
\end{defn}

%%\medskipI

Let $X$ and $Y$ be two operator spaces. 
A complex linear map $\varphi\in L(X,Y)$ is said to be \emph{completely bounded} (respectively, \emph{completely contractive} or \emph{completely isometric}) if $\varphi^{(\infty)}$ is bounded (respectively, contractive or isometric). 
Moreover, $\varphi$ is called a
\emph{complete embedding} if $\varphi^{(\infty)}$ is both bounded and bounded below (i.e. ${\inf}_{\left\|x\right\|=1}\left\|\varphi^{(\infty)}(x)\right\|>0$). 
The set $\CB(X,Y)$ of all completely bounded maps from $X$ to $Y$ is a normed space with norm $\|\varphi\|_\cb := \|\varphi^{(\infty)}\|$. 
In fact, $\CB(X,Y)$ is an operator space under the identification $M_\infty(\CB(X,Y)) = \CB(X,M_\infty(Y))$. 

%%\medskipI

For an operator space $X$, the dual space $X^*$ is again an operator space under the dual matrix norm given by 
\begin{equation}\label{eqt:def-dual-mat-norm}
	\left\|f\right\| :=\sup \{\|f^{(\infty)}(x)\|:x\in B_{M_\infty(X)}\} \qquad (f\in M_\infty(X^*)).
\end{equation}
If, in addition, $X$ and $Y$ are $^*$-operator spaces, then both $X^*$ and $\CB(X,Y)$ are $^*$-operator spaces under the involution defined in a similar way as \eqref{eqt:def-adj}. 

%%\medskipI

For basic theory on matrix norms and operator spaces, the reader may consult, e.g., \cite{OPS}.

%%\medskipI

\section{Preliminaries: Two technical results}\label{section2}

%%\medskipI

Let $X$ and $Y$ be complex vector spaces. 
A bilinear map $( \cdot,\cdot) :X\times Y\to \BC$ is called a \emph{pairing} between $X$ and $Y$ if it satisfies:
$$\{x\in X: {\sup}_{y\in Y}|(x,y) | = 0\} = \{0\}\ \text{and} \ \{y\in Y: {\sup}_{x\in X}|(x,y)| = 0\} = \{0\}.$$ 
In the case when $X$ and $Y$ are $^*$-vector spaces, we say that the pairing $(\cdot, \cdot)$ \emph{respects the involutions} if 
$(x^*,y)= \overline{(x,y^*)}$ ($x\in X$ and $y\in Y$).

%%\medskipI

As in \cite[p.161]{choi}, for a pairing $(\cdot, \cdot)$ between $X$ and $Y$, we define a pairing between $M_\infty(X)$ and $M_\infty(Y)$ by 
\begin{equation}\label{eqt:def-inf-pair}
	(x,y) :={\sum}^\infty_{k,l =1} (x_{k,l},y_{k,l}) \qquad (x\in M_\infty(X); y\in M_\infty(Y)).
\end{equation}
The weakest vector topology on $M_\infty(Y)$ under which $y\mapsto (x,y)$ is continuous for every $x\in M_\infty(X)$ will be denoted by $\sigma\big(M_\infty(Y),M_\infty(X)\big)$. 
It is clear that a net $\{y^{i}\}_{i\in \KI}$ in $M_\infty(Y)$ is $\sigma\big(M_\infty(Y),M_\infty(X)\big)$-converging to $y\in M_\infty(Y)$ if and only if $\{y^{i}_{k,l}\}_{i\in \KI}$ is $\sigma(Y,X)$-converging to $y_{k,l}$, for every $k,l\in \BN$. 
From this, we see that the multiplication (no matter on the left or on the right) by a fixed element in $M_\infty$ is $\sigma\big(M_\infty(Y),M_\infty(X)\big)$-continuous.

%%\medskipI

The topology on $M_n(Y)$ that is induced from $\sigma\big(M_\infty(Y),M_\infty(X)\big)$ will be denoted by $\sigma\big(M_n(Y),M_n(X)\big)$.
It coincides with the weakest topology such that $y\mapsto (x,y)$ is continuous for each $x\in M_n(X)$. 
Observe also that if a subset $S\subseteq M_\infty(Y)$ is $\sigma\big(M_\infty(Y),M_\infty(X)\big)$-closed, then 
$$S\cap M_n(Y) = \big\{y\in S: y_{kl} =0, \text{ for all }k,l\geq n \big\}$$ 
is $\sigma\big(M_n(Y),M_n(X)\big)$-closed in $M_n(Y)$ for every $n\in \BN$. 
The converse is true when $S$ is closed under the projection from $M_\infty(Y)$ to $M_n(Y)$ for all $n\in \BN$.

%%\medskipI

\begin{defn}\label{defn:weak-top}
	Let $X$ be a normed space. 
	The topology $\sigma\big(M_\infty(X^*),M_\infty(X)\big)$ is called the \emph{weak$^*$-topology} on $M_\infty(X^*)$.
\end{defn}

%%\medskipI

\begin{defn}\label{defn:mat-conv}
	Let $X$ be a complex vector space.
	A subset $K\subseteq M_\infty(X)$ is called \emph{matrix convex} if   
	$$\alpha^*u \alpha + \beta^* v \beta \in K,$$
	whenever $u,v\in K$ and $\alpha, \beta\in M_\infty$ with $\alpha^*\alpha + \beta^*\beta$ being a non-zero projection.
	In this case, we set $K_n:=K\cap M_n(X)$ ($n\in \BN$). 
\end{defn}

%%\medskipI

In the following, we always denote by $I_n$ the identity of $M_n$, which will be regarded as a projection in $M_\infty$.

%%\medskipI

\begin{rem}\label{rem:matrix-conv}
	(a) If $K\subseteq M_\infty(X)$ is matrix convex, then  $0\in K$. 
	Indeed, Pick any $u\in K_n$. 
	Then $u\oplus 0\in K_{2n}$, and it is easy to find $\alpha\in M_{2n}$ with $\alpha^*\alpha$ being a non-zero projection and $0=\alpha^*(u\oplus0)\alpha \in K$.

	\smnoind
	(b) If a subset $K\subseteq M_\infty(X)$ is matrix convex, then $0\in K_1$ and $\{K_n\}_{n\in \BN}$ is a ``matrix convex set in $X$'' in the sense of \cite{Witt} (see also \cite[\S 3]{effros1}). 
	In fact, part (a) above implies that $0\in K_1$. 
	Moreover, the following two conditions are clearly satisfied
	for any $m,n\in \mathbb{N}$, 
	\begin{itemize}
		\item $\alpha^*u\alpha\in K_n$ whenever $u\in K_m$ and $\alpha\in M_{m,n}$ with $\alpha^*\alpha=I_n$;
		\item $u\oplus v \in K_{m+n}$ whenever $u\in K_m$ and $v\in K_n$.
	\end{itemize}
	
	\smnoind
	(c) Suppose that $\{K_n\}_{n\in \BN}$ is a matrix convex set in $X$ in the sense of \cite{Witt} such that $0\in K_1$. 
	Then the subset $K:= \bigcup_{n\in \BN}K_n$ of $M_\infty(X)$ is matrix convex in the sense of Definition \ref{defn:mat-conv}. 
	Indeed, by the definition in \cite{Witt}, the two conditions as in part (b) above hold for $\{K_n\}_{n\in \BN}$.
	Let $u,v\in K$ and $\alpha,\beta\in M_\infty$ with $p:= \alpha^*\alpha + \beta^*\beta$ being a non-zero projection. 
	Since $0\in K_n$ ($n\in \BN$), we may find a large enough $m$, so that $u,v\in K_m$ (recall that we identify $M_n(X)\subseteq M_{n+k}(X)$) as well as  $\alpha, \beta\in M_m$. 
	Then there exists a unitary matrix $\nu\in M_m$ such that $\nu^*p\nu= I_k$ for some $k\leq m$. 
	Set $\delta_1 := \nu^* \alpha^* \nu$ and $\delta_2 := \nu^* \beta^* \nu$. 
	Then $(\delta_1,\delta_2)(\delta_1,\delta_2)^*
	=I_k$ and 
\begin{eqnarray*}
	\alpha^* u \alpha + \beta^* v \beta 
	&=& (\alpha^*, \beta^*)
	(u\oplus v)(\alpha^*, \beta^*)^*\\
	&=& \nu (\delta_1,\delta_2)
	(\nu^*\oplus \nu^*)(u\oplus v)(\nu\oplus \nu)(\delta_1,\delta_2)^*\nu^*\in K_m,
\end{eqnarray*}
	because $\nu \nu^* = I_m$ and $(\nu^*\oplus \nu^*)(\nu\oplus  \nu) = I_{2m}$.	
	
\smnoind
(d) Let $\{K_n\}_{n\in \BN}$ be a matrix convex set in $X$, in the sense of \cite{Witt} (but $K_1$ is  not assumed to contain $0$).
If $K:= \bigcup_{n\in \BN}K_n\subseteq M_\infty(X)$, then $K\cap M_n(X) = \{u\oplus 0: u\in K_m; m\leq n\}$ ($n\in \BN$).
Notice that $\{K\cap M_n(X)\}_{n\in \BN}$ is a matrix convex set in $X$ (in the sense of \cite{Witt}) if and only if $0\in K_1$. 
\end{rem}

%%\medskipI

\begin{lem}\label{lem:matrix-conv-with-0}
	Let $X$ be a complex vector space and $K\subseteq M_\infty(X)$. 
	Then $K$ is matrix convex if and only if $\alpha^*u \alpha + \beta^* v \beta \in K$ for any $u,v\in K$ and $\alpha, \beta\in M_\infty$ with $\|\alpha^*\alpha + \beta^*\beta\|\leq 1$. 
\end{lem}
\begin{proof}
	Obviously, if the said condition holds, then $K$ is matrix convex. 
	Conversely, suppose that $K$ is matrix convex. 
	Let $u,v\in K$ and $\alpha,\beta\in M_\infty$ with $\|\alpha^*\alpha + \beta^*\beta\|\leq 1$. 
	Pick a large enough $m\in \BN$ such that  $\alpha, \beta\in M_m$ and $u,v\in K_m$.
	Define $\delta:= (I_m-\alpha^*\alpha - \beta^*\beta)^{1/2} \in M_m$. 
	Since $0\in K$ (see Remark \ref{rem:matrix-conv}(a)), we may regard $0\in K_m$. 
	Then one has $u\oplus v\oplus 0\in K$ and 
	$$\alpha^* u \alpha + \beta^* v \beta 
	= (\alpha^*, \beta^*, \delta)(u\oplus v\oplus 0)(\alpha^*, \beta^*, \delta)^*\in K,$$
	because $(\alpha^*,\beta^*,\delta)(\alpha^*,\beta^*,\delta)^* = I_m$. 
\end{proof}

%%\medskipI
\begin{eg}\label{eg:mat-conv}
	(a) The intersection of two matrix convex subsets of $M_\infty(X)$ is a matrix convex subset. 
	
	\smnoind
	(b) If $X$ is a normed space and $K\subseteq M_\infty(X)$ is matrix convex, then the weak$^*$-closure of $K$ in $M_\infty(X^{**})$ is a matrix convex subset of $M_\infty(X^{**})$ (because of the weak$^*$-continuity of multiplications by elements in $M_\infty$).

	\smnoind
	(c) If $X$ is a $^*$-operator space, then Lemma \ref{lem:matrix-conv-with-0} tells us that both $B_{M_\infty(X)}$ and $B_{M_\infty(X)_{sa}}$ are matrix convex subsets of $M_\infty(X)$. 	
\end{eg}

%%\medskipI

The proposition below is a matrix version of the (strong) separation theorem, which is an adaptation of \cite[Theorem 5.4]{effros1}. 
Note that $I_\infty$ in the statement is the identity of the $C^*$-algebra $\CL(\ell^2)\otimes \CL(\ell^2)$, which contains $M_\infty(M_\infty) = \bigcup_{n\in \BN} M_\infty(M_n)$ as a subalgebra.

%%\medskipI

\begin{prop}[Effros-Winkler]\label{prop:vs-polar}
	Let $Y$ and $Z$ be complex vector spaces equipped with a pairing $(\cdot, \cdot)$.
	Let $K\subseteq M_\infty(Z)$ be a $\sigma\big(M_\infty(Z),M_\infty(Y)\big)$-closed matrix convex subset (containing $0$).
	Let $n\in \BN$ and $v_0\in M_n(Z)\setminus K_n$. 
	
	\smnoind
	(a) There exists a $\sigma(Z,Y)$-continuous linear map $\psi:Z\to M_n$ satisfying  
	$$\R \psi^{(\infty)}(w)\leq I_\infty\ (w\in K)\quad\text{but}\quad \R \psi^{(n)}(v_0)\not\leq I_{M_n(M_n)}.$$

	\smnoind
	(b) Suppose that $Y$ and $Z$ are $^*$-vector spaces such that the pairing $(\cdot, \cdot)$ respects the involutions.
	If $K\subseteq M_\infty(Z)_\sa$ and $v_0\in M_n(Z)_\sa$, then the map $\psi$ in part (a) can be chosen to be self-adjoint.
\end{prop}
\begin{proof}
	(a) Since $K_n$ is a $\sigma\big(M_n(Z),M_n(Y)\big)$-closed convex subset of $M_n(Z)$ containing $0$, the usual separation theorem will produce a $\sigma\big(M_n(Z),M_n(Y)\big)$-continuous complex linear functional $F:M_n(Z)\to \mathbb{C}$ such that $\R F(v)\leq 1 < \R F(v_0)$ for every $v\in K_n$. 
	By \cite[Lemma 5.3]{effros1} and Remark \ref{rem:matrix-conv}(b), there is a state $\omega$ on $M_n$ with
	\[ \R F(\alpha^* v\alpha)\leq \omega(\alpha^*\alpha)\qquad (v\in K_m, \alpha\in M_{m,n}, m\in \BN).\]
	Let $\mathrm{tr}$ be the normalized trace on $M_n$. 
	By replacing $F$ and $\omega$ with $(1-\epsilon)F$ and $(1-\varepsilon)\omega+\varepsilon\ \! \mathrm{tr}$, respectively, for small $\epsilon>0$, we may assume that $\omega$ is faithful. 
	Denote by $(H, \pi, \xi_0)$ the GNS representation of $\omega$.
	Set	
	\[\tilde {\alpha}:= (\alpha^*, 0, \dots, 0)^*\in M_n \quad (\alpha\in M_{1,n}),\]
	and $H_0:=\{\pi(\tilde  \alpha)\xi_0: \alpha\in  M_{1,n}\}$. 
	A similar argument as that of \cite[Theorem 5.4]{effros1} will produce a linear map $\varphi_F:Z\to \CL(H_0)$ satisfying
	\[F(\beta^* z\alpha)=\big\langle\varphi_F(z)\pi(\tilde {\alpha})\xi_0|\pi(\tilde {\beta})\xi_0\big\rangle  \qquad (z\in Z; \alpha,\beta\in M_{1,n}),\]
	$\R \varphi_F^{(\infty)}(w)\leq I_\infty$ ($w\in K$) and  $\R \varphi_F^{(n)}(v_0)\not\leq I_{M_n(M_n)}$.
	Moreover, as $F$ is $\sigma\big(M_n(Z),M_n(Y)\big)$-continuous and $H_0$ is finite dimensional, the above displayed equality tells us  that $\varphi_F$ is $\sigma(Z,Y)$-continuous. 
	
	\smnoind
	(b) Since $K\subseteq M_\infty(Z)_\sa$ and $v_0\in M_n(Z)_\sa$, we may assume the function $F$ in the proof of part (a) is self-adjoint, by replacing $F$ with $(F+F^*)/2$ if necessary. 
	For $\zeta\in H_0$, there exists $\alpha\in M_{1,n}$ with $\zeta = \pi(\tilde \alpha)\xi_0$, and so, 
	\[\left\langle\varphi_F^*(x)\zeta|\zeta\right\rangle=\left\langle\zeta|\varphi_F(x^*)\zeta\right\rangle=\overline{F(\alpha^* x^*\alpha)}=F(\alpha^* x\alpha)=\left\langle\varphi_F(x)\zeta|\zeta\right\rangle \qquad (x\in Z).\]
	Hence,  $\varphi_F$ is self-adjoint.
\end{proof}

%%\medskipI

Suppose that there is a normed space $X$ with an isometric involution such that the $^*$-vector spaces $Y$ and $Z$  in Proposition \ref{prop:vs-polar}(b) are 
dual space $X^*$ and bidual space $X^{**}$, respectively, and the pairing $(\cdot, \cdot)$ is the canonical one. 
Then the weak$^*$-continuous map $\psi: X^{**}\to M_n$ given by the statement of Proposition \ref{prop:vs-polar} will be of the form $\phi^{**}$ for a bounded linear map $\phi:X\to M_n$. 
This gives the following. 

%%\medskipI

\begin{cor}\label{cor:polar}
	Let $X$ be a normed space with an isometric involution and $K\subseteq M_\infty(X^{**})_\sa$ be a weak$^*$-closed matrix convex subset  (containing $0$). 
	For $n\in \BN$ and $v_0\in M_n(X^{**})_\sa\setminus K_n$, one can find a bounded self-adjoint linear map $\phi:X\to M_n$ satisfying $(\phi^{**})^{(\infty)}(w)\leq I_\infty$ ($w\in K$) but 
	$(\phi^{**})^{(n)}(v_0)\not\leq I_{M_n(M_n)}$.
\end{cor}

%%\medskipI

The following fact concerning a $^*$-operator space $X$ is obvious and will be used, sometimes implicitly, throughout this article: 
\begin{equation}\label{eqt:norm-sa}
	\|x\| = \Big\|\Big(\begin{matrix}
		0&x\\
		x^*&0
	\end{matrix}\Big)\Big\|
	\qquad (x\in M_n(X); n\in \BN). 
\end{equation}

%%\medskipI

Our next result is a known fact, but we give an argument for it using our previous discussion.  
Note that we regard the map $\kappa_X:X\to X^{**}$ as the inclusion map here. 

%%\medskipI

\begin{cor}\label{cor:dense-ball}
	For a $^*$-operator space $X$, the canonical image of the unit ball $B_{M_\infty(X)}$ in $M_\infty(X^{**})$ is $\sigma\big(M_\infty(X^{**}), M_\infty(X^*)\big)$-dense in $B_{M_\infty(X^{**})}$. 
\end{cor}
\begin{proof}
	Let $K$ be the $\sigma\big(M_\infty(X^{**}), M_\infty(X^*)\big)$-closure of $B_{M_\infty(X)_\sa}$. 
	Since $$\Big\|\Big(\begin{matrix}
		a&b\\
		c&d
	\end{matrix}\Big)\Big\|\leq 1\text{ will imply } \|b\|\leq 1\quad (a,b,c,d\in M_n(X)),$$ 
	it suffices to show that $K = B_{M_\infty(X^{**})_\sa}$ (see Relation \eqref{eqt:norm-sa}). 
	
	Suppose on the contrary that there exists $v_0\in B_{M_n(X^{**})_\sa}\setminus K_n$ for some $n\in \BN$.
	As $K$ is matrix convex (see Example \ref{eg:mat-conv}), Corollary \ref{cor:polar} produces a bounded self-adjoint linear map $\phi\in \CL(X, M_n) = M_n(X^*)$ with 
	\begin{equation*}\label{eqt:Re-psi-leq}
		\phi^{(\infty)}(w)\leq I_\infty\quad (w\in B_{M_\infty(X)_\sa}) \quad \text{but}\quad (\phi^{**})^{(n)}(v_0)\not\leq I_{M_n(M_n)}. 
	\end{equation*}
	The first set of inequality in the above and Relation \eqref{eqt:norm-sa} tell us that the self-adjoint map $\phi$ is a complete contraction; i.e. $\|\phi\|_{M_n(X^*)}\leq 1$.  
	However, this will imply that $\|(\phi^{**})^{(n)}(v_0)\| = \|v_0^{(n)}(\phi)\|\leq 1$, which contradicts the second inequality above; because $(\phi^{**})^{(n)}(v_0)\in M_n(M_n)_\sa$.
\end{proof}

\begin{comment}
%%\medskipI

The argument above does not apply to the case of general operator spaces when Corollary \ref{cor:polar} is replaced by Proposition \ref{prop:vs-polar}(a). 
The reason is that, unlike the case when $n=1$, for a linear map $\phi: X \to M_n$, there seems to have no relation between $\|\phi^{(\infty)}\|$ and $\|\R \phi^{(\infty)}\|$.
Note that the polar decomposition argument does not work for $n>1$, because for a unitary in $u\in M_m(M_n)$, it does not make sense to talk about $u\cdot \phi^{(m)}(x) = \phi^{(m)}(u\cdot x)$ for $x\in M_m(X)$.

\end{comment}

%%\medskipI

\begin{defn}\label{defn:semi-mat-ord-sp}
	(a) A $^*$-vector space $X$ is called a \emph{semi-matrix ordered vector space} if there exists a subset
	$M_\infty(X)_+\subseteq M_\infty(X)_{sa}$ (called the \emph{matrix cone} of $X$)
	satisfying
	\[\alpha^* x\alpha + \beta^* y \beta\in M_\infty(X)_+ \qquad (\alpha,\beta\in M_\infty; x,y\in M_\infty(X)_+).\]
	In this case, we put $M_n(X)_\sa := M_n(X)\cap M_\infty(X)_\sa$ as well as $M_n(X)_+:=M_n(X)\cap M_\infty(X)_+$ ($n\in \BN$).
	
	\smnoind
	(b) A self-adjoint complex linear map $\varphi:X\to Y$ between two semi-matrix ordered vector spaces is called
	\emph{completely positive}, \emph{completely order monomorphic} or \emph{a complete order isomorphism} if $\varphi^{(\infty)}$ is, respectively, a  positive map, an order monomorphism or an order isomorphism.
	We denote by $\CP(X,Y)$ the set of all completely positive maps from $X$ to $Y$. 
\end{defn}

%%\medskipI

Note that we do not assume the matrix cone in the above to be proper; in other words, $M_\infty(X)_+ \cap - M_\infty(X)_+$ could be non-zero. 
That is why we added the prefix ``semi-'' to distinguish our spaces with the notion of matrix ordered vector spaces considered in \cite{Wern-subsp}. 

%%\medskipI

In the following, $e_k\in M_{1,n}$
is  the element $(0,...,0,1,0,...,0)$ with the $k$-entry being $1$, but all other entries being $0$. 
Moreover, we set 
$\beta_e:= (e_1, \dots, e_n)\in M_{1,n}(M_{1,n})$, as well as 
\begin{equation}\label{eqt:def-tau-kl}
	e^{k,l}:= e_k^*e_l\in M_n.
\end{equation}

%%\medskipI

Let $X$ and $Y$ be semi-matrix ordered vector spaces.
Then 
\begin{quote}
$L(X,Y)$ is an ordered vector space under the adjoint as defined in \eqref{eqt:def-adj} and the cone $\CP(X,Y)$ as in Definition \ref{defn:semi-mat-ord-sp}(b). 
\end{quote}

%%\medskipI

In the case when $X$ is an unital operator system, it is well-known that one can determine whether a map $\varphi:X\to M_n$ belongs to $\CP(X,M_n)$ via the positivity of a linear functional associated with it. 
The next result tells us that it is also true in the general case.
This result is taken from  \cite[Lemma 5]{JN} (whose proof can be found in the ``arXiv version'' of \cite{JN}). 

%%\medskipI

\begin{lem}\label{lem:JN-Lemma5}
	Let $X$ be a semi-matrix ordered vector space and $n\in \BN$. 
	For any $\phi\in L(X,M_n)$, we set 
	\begin{equation}\label{eqt:rel-Theta-F}
		\Upsilon_\phi(u) :={\sum}_{k,l=1}^n \phi(u_{k,l})_{k,l}=\beta_e \phi^{(n)}(u)\beta_e^* \qquad (u\in M_n(X)). 
	\end{equation}	
	Then $\phi\mapsto \Upsilon_\phi$ is an order isomorphism from $L(X,M_n)$ onto $L(M_n(X),\BC)$, and its inverse is given by $F\mapsto (\Upsilon^{-1})_F$, where 
	\begin{equation*}\label{eqt:def-Delta}
		(\Upsilon^{-1})_F(x)_{k,l}=F(e^{k,l}\otimes x) \qquad (x\in X; k,l=1,\dots,n).
	\end{equation*}
	for any
	$F\in L(M_n(X),\BC)$. 
\end{lem}

%%\medskipI

In our main study, we also need to consider an arbitrary semi-matrix ordered vector space $X$ equipped with a compatible matrix norm. 
In this case, there is a canonical norm on $\CB(X,M_n)$. 
It is natural to ask if it is possible to describe the norm on $M_n(X)^*$ induced by the map $\Upsilon: L(X,M_n)\to L(M_n(X),\BC)$ in the above. 
In Theorem \ref{thm:bdd-pos-CCP}(b) below, we will see that one can at least describe the restriction of this induced norm on $M_n(X)^*_\sa$. 
In order to do this, we need a technical proposition.
Let us first give some preparation for it.

%%\medskipI

Consider $m,n\in\BN$. 
We recall from \cite[Lemma 2.1]{choi} that the map $\theta:M_m(M_n)\to L(M_m,M_n)$ defined by 
\begin{equation*}\label{eqt:choi}
	\theta_\tau(\alpha) := {\sum}_{k,l=1}^m\tau_{k,l}\alpha_{k,l}\qquad (\tau\in M_m(M_n);\alpha\in M_m)
\end{equation*}
is an order isomorphism, when $M_m(M_n)$ is identified with the $C^*$-algebra $M_{m\times n}$ in the canonical way, and the cone of $L(M_m,M_n)$ is $\CP(M_m,M_n)$.
This produces an order monomorphism $\theta:M_\infty(M_n)\to L(M_\infty,M_n)$, via the order monomorphism from $L(M_m,M_n)$  to $L(M_\infty,M_n)$ given by the truncation map from $M_\infty$ to $M_m$ (which is completely positive).  
One has 
\begin{equation*}\label{eqt:theta-tau-otimes}
	(\theta_\tau\otimes \id_X)(v) = {\sum}_{k,l=1}^\infty\tau_{k,l}\otimes v_{k,l}.
\end{equation*}
for $\tau=[\tau_{k,l}]_{k,l}\in M_\infty(M_n)$ and $v=[v_{k,l}]_{k,l}\in M_\infty(X)$.

%%\medskipI

Moreover, $\Tr_n:M_n\to \BC$ is the usual trace; i.e. $\Tr_n\big([\alpha_{k,l}]_{k,l}\big) := \sum_{i=1}^n \alpha_{i,i}$.
We consider $\Tr_\infty:M_\infty\to \BC$ to be the functional with $\Tr_\infty|_{M_n} = \Tr_n$ ($n\in \BN$),  
and we set (here, $\Tr_n^{(\infty)}:=(\Tr_n)^{(\infty)}$)
\begin{equation*}\label{eqt:def-norm-1}
	\|\tau\|_1:=\Tr_\infty\big(\Tr_n^{(\infty)}\big((\tau^*\tau)^{1/2}\big)\big) \qquad (\tau\in M_\infty(M_n)).
\end{equation*}
In the case when $\tau \in M_m(M_n) = M_{m\times n}$, 
one has $\Tr_m\big(\Tr_n^{(m)}(\tau)\big) = \Tr_{m\times n}(\tau)$ (where $\Tr_n^{(m)}$ sends $M_m(M_n)$ to $M_m$), and $\|\tau\|_1$ coincides with the usual trace-class norm of $\tau\in M_{m\times n}$.

%%\medskipI

For $\tau = [\tau_{k,l}]_{k,l}\in M_\infty(M_n)$, let $\tau^\tp \in M_\infty(M_n)$ be the element satisfying
$$(\tau^\tp)_{k,l} = (\tau_{l,k})^\tp \in M_n\qquad (k,l\in \BN),$$ 
where $(\tau_{l,k})^\tp$ is the usual transpose in  $M_n$. 
If $\tau \in M_m(M_n)$, then $\tau^\tp\in M_m(M_n)$ can be identified with the usual transpose of $\tau$ in $M_{m\times n}$.
It is easy to see that $\tau\mapsto \tau^\tp$ is an order isomorphism from $M_\infty(M_n)$ onto itself. 
Now, part (a) of the following lemma is clear (via the corresponding well-known fact about $M_{m\times n}^*$), while part (b) is taken from \cite[Lemma 6]{JN}. 

%%\medskipI

\begin{lem}\label{lem:dual}
	Let $m,n\in \BN$. 
	
	\smnoind
	(a) For $\tau\in M_m(M_n)$, we define a functional $\varphi_\tau: M_m(M_n)\to \BC$ by $$\varphi_\tau(\alpha) := \Tr_m\big(\Tr_n^{(m)}(\tau^\tp\alpha)\big)
	\qquad (\alpha\in M_m(M_n)).$$
	Then $\tau\mapsto \varphi_\tau$ is an isometric order isomorphism from $\big(M_m(M_n), \|\cdot\|_1\big)$ onto $M_m(M_n)^*$. 
	
	\smnoind
	(b) If $X$ is a semi-matrix ordered vector space and $\phi\in L(X,M_n)$, then 
	\begin{align}\label{eqt:Upsil-phi}
		\Upsilon_\phi\big((\theta_\tau\otimes \id_X)(v)\big) = \Tr_\infty\big(\Tr_n^{(\infty)}\big(\tau^\tp\phi^{(\infty)}(v)\big)\big).
	\end{align}
	when $v\in M_\infty(X)$ and $\tau\in M_\infty(M_n))$. 
\end{lem}

%%\medskipI

\begin{defn}\label{defn:fin-mat-gau}
	Let $X$ be a complex vector space. 
	
	\smnoind
	(a) A function $\lambda:X\to \RP$ is called a 
	\emph{finite gauge} if 
	$$\lambda(s x + ty)\leq s\lambda(x) +t\lambda(y) \qquad (x,y\in X;s,t\in \RP).$$ 
	
	\smnoind
	(b) A function $\rho:M_\infty(X)\to \RP$ is called a \emph{finite matrix gauge} if 
	$$\rho(\alpha^* x \alpha + \beta^* y\beta) \leq \|\alpha^*\alpha + \beta^*\beta\|\max \{\rho(x),\rho(y)\}$$
	for  $x,y\in M_\infty(X)$ and $\alpha,\beta\in M_\infty$.
	We denote $\rho_n:= \rho|_{M_n(X)}$ ($n\in \BN$). 
\end{defn}

%%\medskipI

Although the above definition of matrix gauge looks a bit different from the one defined in \cite[\S 6]{effros1}, one can easily show that they are actually the same (by using the argument of \cite[Proposition 2.3.6]{OPS} as well as the fact that $\alpha^* x \alpha + \beta^* y \beta
= (\alpha^*, \beta^*)
(x\oplus y)(\alpha^*, \beta^*)^*$, for $x,y\in X$ and $\alpha, \beta\in M_\infty$). 

%%\medskipI

Motivated by some ideas in \cite{effros1}, we have the following (as in the above, $I_\infty$ is the identity of $\CL(\ell^2)\otimes \CL(\ell^2)$). 

%%\medskipI

\begin{prop}\label{bijective}
	Suppose that $X$ is a semi-matrix ordered vector space and $\rho$ is a finite matrix gauge on $X$.  
	Fix $n\in \BN$.
	We define a finite gauge $\gamma$ on $M_n(X)$ such that for $u\in M_n(X)$, 
	\begin{equation*}\label{eqt:def-hat-rho}
		\gamma (u):=\inf\big\{\left\|\tau\right\|_1\rho(v):v\in M_\infty(X); \tau\in M_\infty(M_n)_+; u=(\theta_\tau \otimes \id_X)(v)\big\}.
	\end{equation*}
	
	\smnoind
	(a) For $\phi\in L(X,M_n)_\sa$, one has
	$\phi^{(\infty)}(v) \leq \rho(v) I_\infty$
	for any $v\in M_\infty(X)_+$ (respectively, $v\in M_\infty(X)_\sa$) if and only if $\Upsilon_\phi(u) \leq \gamma (u)$ for any $u\in M_n(X)_+$  (respectively, $u\in M_n(X)_\sa$). 
	
	\smnoind
	(b) Suppose that 
	\[\rho(y^*) = \rho(y) = \rho(e^{\mathrm{i}s} y)=\rho\Big(\begin{matrix}
		0 & y\\
		y^* & 0
	\end{matrix}\Big)\quad  (y\in M_\infty(X);s\in\BR). \]
	Then for each $\phi\in L(X,M_n)_\sa$, the following statements are equivalent:
	\begin{itemize}
		\item $\big\|\phi^{(\infty)}(y)\big\| \leq \rho(y)$ ($y\in M_\infty(X)$). 
		
		\item $\Upsilon_\phi(u) \leq \gamma (u)$ ($u\in M_n(X)_\sa$).		
		
		\item $|\Upsilon_\phi(z)| \leq \gamma (z)$ ($z\in M_n(X)$). 
	\end{itemize}
\end{prop}
\begin{proof}
	(a) $\Rightarrow$). 
	Pick any $v\in M_\infty(X)_+$ (respectively, $v\in M_\infty(X)_\sa$) and $\tau\in M_\infty(M_n)_+$.
	The assumption on $\phi^{(\infty)}$ and Lemma \ref{lem:dual}(a) 
	imply that 
	\begin{align*}
	\Tr_\infty\big(\Tr_n^{(\infty)}\big(\tau^\tp\phi^{(\infty)}(v)\big)\big) 
	& \leq \rho(v)\Tr_\infty\big(\Tr_n^{(\infty)}(\tau^\tp)\big)\\
	& = \rho(v)\Tr_\infty\big(\Tr_n^{(\infty)}(\tau)\big)
	= \left\|\tau\right\|_1\rho(v).
	\end{align*}
	Thus,  Relation \eqref{eqt:Upsil-phi} tells us that 
	\begin{equation}
		\label{F}
		\Upsilon_\phi((\theta_\tau\otimes \id_X)(v)) = \Tr_\infty\big(\Tr_n^{(\infty)}\big(\tau^\tp \phi^{(\infty)}(v)\big)\big) \leq \left\|\tau\right\|_1\rho(v).
	\end{equation}
	Suppose on the contrary that there is $u^0\in M_n(X)_+$ (respectively, $u^0\in M_n(X)_\sa$) with $\gamma (u^0)< \Upsilon_\phi(u^0)$. 
	By \cite[Lemma 6.8]{effros1}, one has
	\begin{align*}
	\gamma(u^0)=\inf\big\{\Tr_n(\beta^2)\rho(w): \beta\in (M_n)_+&\text{ is invertible and }\\
	& w\in M_n(X)  \text{ such that }u^0=\beta w\beta\big\}.
	\end{align*}
	Hence, we can find $w^0\in M_n(X)$ and an invertible matrix $\beta\in (M_n)_+$ with
	$$u^0=\beta w^0 \beta \quad \text{and} \quad \Tr_n(\beta^2)\rho(w^0)<\Upsilon_\phi(u^0).$$ 
	As $\beta^{-1}\in (M_n)_+$, one has $w^0=\beta^{-1}u^0\beta^{-1}\in M_n(X)_+$  (respectively, $w^0\in M_n(X)_\sa$).
	Let us write $\beta$ as $(\beta_1,\dots,\beta_n)$ with $\beta_i\in M_{n, 1}$, and put
	$\tilde{\beta}:=(\beta_1^*, \dots, \beta_n^*)\in M_{1,n}(M_{1,n})$. 
	If we set $\check\beta:=\tilde{\beta}^*\tilde{\beta}\in M_n(M_n)_+$, then 
	\begin{align*}
	u^0
	=\beta w^0 \beta
	= \beta w^0 \beta^*
	& =\Big[{\sum}_{k,l=1}^n\beta_{i,k}w^0_{k,l}\bar \beta_{j,l}\Big]_{i,j}\\
	& ={\sum}_{k,l=1}^n \check\beta_{k,l}\otimes w^0_{k,l} 
	= (\theta_{\check\beta}\otimes \id_X)(w^0).
	\end{align*}
	Thus,  $\left\|\check\beta\right\|_1\rho(w^0)=\Tr_n(\beta^2)\rho(w^0)< \Upsilon_\phi(u^0) = \Upsilon_\phi\big((\theta_{\check\beta}\otimes \id_X)(w^0)\big)$, which contradicts Relation (\ref{F}).
	
	\noindent
	$\Leftarrow$). 
	Consider $v\in M_\infty(X)_+$ (respectively, $v\in M_\infty(X)_\sa$).
	By the assumed inequality  and the definition of $\gamma$, we have
	\begin{equation}\label{G}
		\Upsilon_\phi\big((\theta_\tau\otimes \id_X)(v)\big)\leq \left\|\tau\right\|_1\rho(v) \qquad (\tau\in M_\infty(M_n)_+).
	\end{equation}
	It then follows from Relation \eqref{eqt:Upsil-phi} that for every $\tau\in M_\infty(M_n)_+$,  
	\[\Tr_\infty\big(\Tr_n^{(\infty)}\big(\tau^\tp \phi^{(\infty)}(v)\big)\big) 
	\leq \left\|\tau\right\|_1\rho(v) 
	= \rho(v)\Tr_\infty\big(\Tr_n^{(\infty)}(\tau^\tp)\big),\]
	which implies $\phi^{(\infty)}(v)\leq \rho(v) I_\infty$, because of Lemma \ref{lem:dual}(a). 
	
	\smnoind
	(b) As in Lemma \ref{lem:JN-Lemma5}, one has $\Upsilon_\phi\in L(M_n(X), \BC )_\sa$. 
	Moreover, the assumption on $\rho$ implies that 
	\begin{equation}\label{eqt:lambda-semi-norm}
		\gamma (z^*) = \gamma (z) = \gamma(e^{\mathrm{i}s}z)\qquad (z\in M_n(X);s\in \BR).
	\end{equation}                                                                             
	
	Assume  that $\|\phi^{(\infty)}(y)\| \leq \rho(y)$ ($y\in M_\infty(X)$). 
	Then 
	$$-\rho(v)I_\infty\leq \phi^{(\infty)}(v)\leq \rho(v)I_\infty\quad (v\in M_\infty(X)_\sa).$$ 
	It then follows from part (a) above that $\Upsilon_\phi(u) \leq \gamma (u)$ ($u\in M_n(X)_\sa$).
	
	Now, assume that $\Upsilon_\phi(u) \leq \gamma (u)$ ($u\in M_n(X)_\sa$). 
	Let $z\in M_n(X)$. 
	Then
	$$	\R \Upsilon_\phi(z)=\Upsilon_\phi(\R z)\leq\gamma (\R z)\leq (\gamma (z)+\gamma (z^*))/2 =\gamma (z).$$
	Since $\gamma (z) = \gamma (e^{\mathrm{i}s}z)$ for any $s\in \BR$, we then conclude that $|\Upsilon_\phi(z)| \leq \gamma (z)$. 
	
	Finally, assume that $|\Upsilon_\phi(z)| \leq \gamma (z)$ ($z\in M_n(X)$). 
	Then $\Upsilon_\phi(u)\leq \gamma (u)$, for each $u\in M_n(X)_\sa$.
	By part (a) above, for each $v\in M_\infty(X)_\sa$, 
	$$-\rho(v)I_\infty = -\rho(-v)I_\infty \leq \phi^{(\infty)}(v) \leq \rho(v)I_\infty;$$
	i.e. $\|\phi^{(\infty)}(v)\| \leq \rho(v)$. 
	Consequently, for $m\in \BN$ and $y\in M_m(X)$, the assumption on $\rho$ gives 
	$$\|\phi^{(m)}(y)\|
	= \Big\|\phi^{(2m)}\Big(\begin{matrix}
		0 & y\\
		y^* & 0
	\end{matrix}\Big)\Big\|
	\leq \rho\Big(\begin{matrix}
		0 & y\\
		y^* & 0
	\end{matrix}\Big) = \rho(y).$$ 
\end{proof}

%%\medskipI

Notice that Relation \eqref{F} does not imply $\Upsilon_\phi(u) \leq \gamma (u)$  directly, because in the definition of $\gamma$, the element $v$ is not assumed to be self-adjoint (respectively, positive) even when $u$ is self-adjoint (respectively, positive). 

%%\medskipI

\section{Some results concerning semi-matrix ordered operator spaces}

%%\medskipI

In order to study the dual space of an operator system (see Definition \ref{MOS}(b) below), which may not be an operator system any more, we need the notion of matrix ordered operator spaces from \cite{Wern-subsp}, and an adapted version of it (because the dual spaces may not even have proper matrix cones). 

%%\medskipI

\begin{defn}\label{MOS}
	(a) If $X$ is both a $^*$-operator space and a semi-matrix ordered vector space with the same involution, then it is called a \emph{semi-matrix ordered operator space (SMOS)} if the matrix cone $M_\infty(X)_+$ is norm closed. 
	Moreover, a SMOS $X$ is called a \emph{matrix ordered operator space (MOS)} if the matrix cone is \emph{proper}; i.e., $M_\infty(X)_+\cap-M_\infty(X)_+=\left\lbrace0\right\rbrace$.
	For SMOS $X$ and $Y$, we will denote by $\Morc(X,Y)$ the set of all completely positive complete contractions from $X$ to $Y$. 
	
	\smnoind
	(b) A MOS $V$ is called an (respectively, a \emph{unital}) \emph{operator system} if there is a completely order monomorphic  complete isometry $\Psi: V \to \CL(H)$ for a Hilbert space $H$ (respectively, such that $\Psi(V)$ contains the identity of $\CL(H)$). 
	On the other hand, $V$ is called a \emph{quasi-operator system} if $\Psi$ is only assumed to be a complete embedding instead of a complete isometry. 
\end{defn}

%%\medskipI

Operator systems as defined above is slightly different from the ones in \cite[Definition 4.13]{Wern-subsp}. 
In fact, ``operator systems'' in \cite{Wern-subsp} are precisely ``quasi-operator systems'' in our sense.

%%\medskipI

Note that our operator systems are self-adjoint subspace of some $\CL(H)$, equipped with the induced matrix norm and the induced matrix cone, and this definition coincides with the one in \cite{CvS} (see Remark \ref{rem:conn-with-lit} in the next section).

%%\medskipI

For basic materials on unital operator systems, the reader may consult, e.g. \cite{Paul}.

%%\medskipI

\begin{rem}\label{rem:dual-SMOS}
	Let $X$ and $Y$ be SMOS. 
	
	\smnoind
	(a) For an element 
	$$f\in M_\infty(X^*) = \CB(X,M_\infty),$$ 
	we always consider its norm $\|f\|$ to be the ``completely bounded norm'' as defined in \eqref{eqt:def-dual-mat-norm}, BUT NOT the usual operator norm $\|f\|_{\CL(X,M_\infty)}$. 
	
	\smnoind
	(b) The dual space $X^*$ is a SMOS, under the involution as in the second last paragraph of Section \ref{sec:Introd} and the following matrix cone: 
	$$M_\infty(X^*)_+:=\CB(X,M_\infty)\cap \CP(X,M_\infty).$$
	In other words, for $f\in M_\infty(X^*)$, one has $f\in M_\infty(X^*)_+$ if and only if 
	$$ f\in M_\infty(X^*)_\sa \qquad \text{and}\qquad  (f,x)\geq 0 \quad (x\in M_\infty(X)_+),$$ 
	where $(f,x)$ is as in Relation \eqref{eqt:def-inf-pair}. 
	Note that if $M_\infty(X)_+ = \{0\}$, then $M_\infty(X^*)_+=\CB(X,M_\infty)_\sa.$
	
	\smnoind
	(c) Since the involution on $X^*$ is weak$^*$-continuous, $M_\infty(X^*)_\sa$ is a weak$^*$-closed subspace of $M_\infty(X^*)$ (see Definition \ref{defn:weak-top}).

	\smnoind
	(d) One may identify $\Morc(X,M_n)$ with 
	$B_{M_n(X^*)}\cap M_n(X^*)_+$ $(n\in \BN)$. 
	
	\smnoind
	(e) If $\psi\in \Morc(X,Y)$, then $\psi^*$ is a weak$^*$-continuous map in $\Morc(Y^*,X^*)$.  
\end{rem}

%%\medskipI

%%\medskipI

Some people define the matrix cone $M_\infty(X^*)_+$ to be the one induced by the bijection $\Upsilon$ as in Lemma \ref{lem:JN-Lemma5}. 
The following result implies that this definition of matrix cones agrees with the one in Remark \ref{rem:dual-SMOS}(b). 
%Moreover, Theorem \ref{thm:bdd-pos-CCP}(a) below also gives a comparison between the norm on $M_n(X^*)_\sa$ (see Relation \eqref{eqt:def-dual-mat-norm}) and a norm defined on $M_n(X)^*_\sa$. 
Actually, the following theorems give a connection between the SMOS $X^*$ (respectively, $X^{**}$)  and the sequence $\{(M_n(X)^*, \gamma_X^*)\}_{n\in \BN}$ (respectively, $\{M_n(X)^{**}\}_{n\in \BN}$) of ordered Banach spaces. 

%%\medskipI

Notice also that the semi-norm $\gamma_X$ in Theorem \ref{thm:bdd-pos-CCP} coincides with the one introduced in \cite[Relation (25)]{effros1}, because of \cite[Lemma 6.8]{effros1}. 
As shown in the first paragraph of \cite[p.140]{effros1}, this semi-norm is equivalent to the original norm $\|\cdot\|$ on $M_n(X)$ given by the operator space structure on $X$. 
More precisely,  one has $\|\cdot\|\leq \gamma_X \leq n\|\cdot\|$. 
Hence, $\gamma_X$ is a norm. 

%%\medskipI

\begin{thm}\label{thm:bdd-pos-CCP}
	Let  $X$ be a SMOS and $n\in \BN$. 
	Let $e^{k,l}$ be as in \eqref{eqt:def-tau-kl}. 
	Define a norm $\gamma_X$ on $M_n(X)$ by 
	\begin{align*}
	\gamma_X(x):= \inf\big\{\Tr_n(\beta^2)\|w\|: &\beta\in (M_n)_+\text{ is invertible and }\\
	& w\in M_n(X)  \text{ such that }x=\beta w\beta\big\}\quad (x\in M_n(X)).
	\end{align*}
For $F\in M_n(X)^*$, let us also define $\Theta^X_F\in M_n(L(X;\BC))$ by 
$$(\Theta^X_F)_{k,l}(x) := F(e^{k,l}\otimes x)\qquad (x\in X; k,l=1,\dots,n).$$

	\smnoind
	(a) $F\mapsto \Theta^X_F$ gives a continuous order isomorphism $\Theta^X: M_n(X)^*\to M_n(X^*)$. 

\smnoind
(b) $\Theta^X|_{M_n(X)_\sa^*}: M_n(X)_\sa^* \to M_n(X^*)_\sa$ 
	is an isometry when $M_n(X)_\sa^*$ is equipped with the dual norm $\gamma_X^*$ of $\gamma_X$ and $M_n(X^*)_\sa$ is considered as a real normed subspace of $M_\infty(X^*)$.
	
	\smnoind
	(c) For $x \in M_n(X)$ and $F\in M_n(X)^*$, one has (see \eqref{eqt:def-inf-pair})
	\begin{equation}\label{eqt:F-Theta-X-F}
	F(x) = (x,  \Theta^X_F).
	\end{equation}
	Hence, $\Theta^X$ is a $\sigma\big(M_n(X)^*, M_n(X)\big)$-$\sigma\big(M_n(X^{*}), M_n(X)\big)$-homeomorphism. 
\end{thm}
\begin{proof}
	(a) Note that $\Theta^X_F = (\Upsilon^{-1})_F$, where  $\Upsilon^{-1}: L(M_n(X), \BC)\to L(X,M_n)$ is the order isomorphism in Lemma \ref{lem:JN-Lemma5}.
	Thus, in order to show that $\Theta^X$ is an order isomorphism from $M_n(X)^*$ to $M_n(X^*)$, it suffices to verify $\Upsilon^{-1}(M_n(X)^*) = \CB(X,M_n)$. 
	In fact, for $F\in M_n(X)^*$, one has $\|F^{(\infty)}\| = \|F\|$, which gives 
	\begin{equation}\label{eqt:norm-Delta-F}
\left\|(\Upsilon^{-1})_F(x)\right\|_{M_n}  = \big\|F^{(n)}\big([e^{k,l}]_{k,l}\otimes x\big)\big\| \leq n \|F\|\|x\| \qquad (x\in X), 
	\end{equation}
	because  $\big([e^{k,l}]_{k,l}\big)^* = [e^{k,l}]_{k,l}\in M_n(M_n)$ and $\big([e^{k,l}]_{k,l}\big)^2 = n [e^{k,l}]_{k,l}$. 
	Hence, 
	$$\Upsilon^{-1}(M_n(X)^*) \subseteq \CL(X,M_n) = \CB(X,M_n).$$ 
	Conversely, if $\phi\in \CB(X,M_n)$, then the functional $\Upsilon_\phi$ as defined in \eqref{eqt:rel-Theta-F} is clearly bounded; i.e., $\Upsilon_\phi\in M_n(X)^*$, and this ensures the surjectivity of $\Theta^X$.  
	
	Furthermore, \cite[Corollary 2.2.4]{OPS} implies 
	$$\|\Theta^X_F\|_{cb}\leq n\|\Theta_F^X\| = n\|(\Upsilon^{-1})_F\|.$$ 
	Hence, Relation \eqref{eqt:norm-Delta-F} gives $\|\Theta^X_F\|_{cb}\leq n^2\|F\|$. 
	This indicates that $\Theta^X$ is continuous. 
	
	\smnoind
	(b) Let $\gamma$ be the finite gauge on $M_n(X)$ as defined in Proposition \ref{bijective} when $\rho$ is the matrix norm on $X$.
	Then \cite[Lemma 6.8]{effros1} tells us that $\gamma = \gamma_X$. 
	It follows from Proposition \ref{bijective}(b) that for $F\in M_n(X)^*_\sa$, the condition
	$$\big\|(\Theta^X_F)^{(\infty)}(y)\big\|\leq \|y\| \ \  (y\in M_\infty(X))$$
	is the same as  $|F(z)| \leq \gamma_X(z)$ ($z\in M_n(X)$), 
	and this equivalence means that  $\gamma_X^*(F) = \|\Theta_F^X\|$ (see Remark \ref{rem:dual-SMOS}).

	\smnoind
	(c) This part follows from the definitions. 
\end{proof}

\begin{thm}\label{thm:bidual-MOS}
Let  $X$ be a SMOS and $n\in \BN$, and let $\Theta^X$ be as in Theorem \ref{thm:bdd-pos-CCP}. 
Define a map $\Lambda:M_n(X^{**})\to M_n(X)^{**}$ by 
	$$\Lambda(z)(F) := (z,  \Theta^X_F) \qquad (z\in M_n(X^{**}); F\in M_n(X)^*).$$
	Then $\Lambda$ is an isometric order isomorphism, when $M_n(X)^{**}$ is equipped with the bidual ordered Banach space structure and $M_n(X^{**})$ is considered as an ordered normed  subspace of $M_\infty(X^{**})$. 
	Moreover,  $\Lambda$ preserves the canonical images of $M_n(X)$ and is  a $\sigma\big(M_n(X^{**}), M_n(X^*)\big)$-$\sigma\big(M_n(X)^{**}, M_n(X)^*\big)$-homeomorphism. 
\end{thm}
\begin{proof}
By Theorem \ref{thm:bdd-pos-CCP}(a), the two maps $\Theta^X: M_n(X)^* \to M_n(X^*)$ and $\Theta^{X^*}: M_n(X^*)^*\to M_n(X^{**})$ are order isomorphisms.  
	Hence, the map 
	$$(\Theta^X)^*\circ (\Theta^{X^*})^{-1}: M_n(X^{**}) \to M_n(X)^{**}$$ 
	is an order isomorphism. 
	Let $y\in M_n(X)^{**}$, and $g:= (\Theta^{X^*})^{-1}(y)$. 
	Then 
	\begin{align*}
		\Lambda(y)(F) 
		& = (\Theta^{X^*}_g, \Theta^X_F) 
		\ = \ {\sum}_{k,l=1}^n (\Theta^{X^*}_g)_{k,l}\big( (\Theta^X_F)_{k,l} \big)\\
		& =  {\sum}_{k,l=1}^n g\big( e^{k,l}\otimes (\Theta^X_F)_{k,l} \big)
		 =  g(\Theta^X(F)) 
		 =  (\Theta^X)^*(g)(F)
	\end{align*}
	for every $F\in M_n(X)^*$. 
	This shows that $\Lambda$ coincides with the order isomorphism $(\Theta^X)^*\circ (\Theta^{X^*})^{-1}$.

	Furthermore, Theorem \ref{thm:bdd-pos-CCP}(c) tells us that $\Lambda$ sends the canonical image of $M_n(X)$ in $M_n(X^{**})$ onto the canonical image of $M_n(X)$ in $M_n(X)^{**}$. 
	Theorem \ref{thm:bdd-pos-CCP}(c) and the bijectivity of $\Theta^X$ also imply  that the map $\Lambda$ is a $\sigma\big(M_n(X^{**}), M_n(X^*)\big)$-$\sigma\big(M_n(X)^{**}, M_n(X)^*\big)$-homeomorphism. 
	Therefore,  by Theorem \ref{thm:bdd-pos-CCP}(c) again,  $\Lambda$ will send the subset  $\overline{B_{M_n(X)}}^{\sigma(M_n(X^{**}), M_n(X^*))}$ onto the subset $\overline{B_{M_n(X)}}^{\sigma(M_n(X)^{**}, M_n(X)^*)}$. 
	Since the closed unit ball of a normed space is weak$^*$-dense in the closed unit ball of its bidual, we have
	$$B_{M_n(X)^{**}} = \overline{B_{M_n(X)}}^{\sigma(M_n(X)^{**}, M_n(X)^*)}.$$ 
	Moreover, Corollary \ref{cor:dense-ball} tells us that $B_{M_n(X^{**})} = \overline{B_{M_n(X)}}^{\sigma(M_n(X^{**}), M_n(X^*))}$. 
	Consequently, $\Lambda$ is an isometry. 
\end{proof}

\begin{comment}
%%\medskipI

\begin{rem}\label{rem:bar-Theta}
\cite[Lemma 3.4(a)]{Ng-dual-OS} also associates an element $\varphi\in \Morc(X,M_n)$ with a positive contraction $F\in M_n(X)^*_+$. 
However, unlike $\Theta^X_F$, the map $\varphi$ as given by \cite[Lemma 3.4]{Ng-dual-OS} is not unique. 
On the other hand, the map $\Theta^X_F$ in Theorem \ref{thm:bdd-pos-CCP}(a) may not be completely contractive when $F$ is contractive (under the usual dual norm, instead of $\gamma_X$). 

Nevertheless, one can still use Theorem \ref{thm:bdd-pos-CCP}(a) to replace \cite[Lemma 3.4(a)]{Ng-dual-OS} to obtain the main result in \cite{Ng-dual-OS}, because we do not actually need an explicit control of the norm of $\varphi$ in the proof of \cite[Lemma 3.4(a)]{Ng-dual-OS} (see also the proof of Lemma \ref{lem:Ng-MOS}(a) below). 
\end{rem}

\end{comment}
%%\medskipI

For a SMOS $X$, let us denote
\begin{equation*}\label{eqt:def-B-Mn+}
	B_{M_n(X)}^+ := B_{M_n(X)}\cap M_n(X)_+ \qquad (n\in \BN).
\end{equation*}
We recall from \cite{Sch} that $X$ is said to be \emph{matrix regular} if for each $x\in M_\infty(X)_\sa$, the following conditions are satisfied:
\begin{itemize}
	\item the existence of $u\in B_{M_\infty(X)}^+$ with $-u\leq x\leq u$ will imply that $\|x\|\leq 1$;
	
	\item when $\|x\| < 1$, there exists $u\in B_{M_\infty(X)}^+$ with $-u\leq x\leq u$.
\end{itemize}

%%\medskipI

Our next lemma is easy to verify.

%%\medskipI

\begin{lem}\label{lem:matrix-reg}
	Let $X$ be a SMOS. 
	
	\smnoind
	(a) $X$ is matrix regular if and only if  $O_{M_\infty(X)_\sa}\subseteq \MS(B_{M_\infty(X)_\sa})\subseteq B_{M_\infty(X)_\sa}$, where $O_{M_\infty(X)_\sa}$ is the open unit ball of $M_\infty(X)_\sa$ and $\MS(B_{M_\infty(X)_\sa})$ is as defined in  Relation \eqref{eqt:def-solid}.
	
	\smnoind
	(b) If $X$ is matrix regular, then the norm on $M_\infty(X)_\sa$ is absolutely monotone (see Definition \ref{defn:Riesz}). 
	
	\smnoind
	(c) If $X$ is an operator system, the norm on $M_\infty(X)_\sa$ is absolutely monotone.  
	
	\smnoind
	(d) If $X$ is a $C^*$-algebra or a unital operator system, then it is matrix regular. 
\end{lem}

%%\medskipI

It follows from part (a) above and Lemma \ref{lem:abs-mono}(b) that  $X$ is matrix regular if and only if the norm on $M_\infty(X)_\sa$ is a regular norm (see Definition \ref{defn:Riesz}).
We will see in Proposition \ref{prop:iota-comp-embed} that an operator system $S$ is dualizable (see the Introduction) if and only if the norm on $M_\infty(S)_\sa$ is equivalent to a regular norm. 

%%\medskipI

\begin{lem}\label{lem:abs-cont-to-mat-reg}
	Let $X$ and $Y$ be SMOS such that the norm on $M_\infty(X)_\sa$ is absolutely monotone and $Y$ is matrix regular. 
	If $\varphi:X\to Y$ is a completely contractive complete order isomorphism satisfying $\|\varphi^{(\infty)}(u)\| = \|u\|$ ($u\in M_\infty(X)_+$),  then $\varphi$ is a complete isometry.
\end{lem}
\begin{proof}
	Since $\varphi$ is self-adjoint, by Relation \eqref{eqt:norm-sa}, it suffices to show that $\varphi^{(\infty)}$ restricts to an isometry on $M_\infty(X)_\sa$. 
	Consider $x\in M_\infty(X)_\sa$ with $\|\varphi^{(\infty)}(x)\|< 1$. 
	As $Y$ is matrix regular, there exists $w\in B_{M_\infty(Y)}^+$ with $-w \leq \varphi^{(\infty)}(x) \leq w$.
	The hypothesis gives a unique $v\in B_{M_\infty(X)}^+$ with $\varphi^{(\infty)}(v) = w$ and we have $-v\leq x\leq v$. 
	The absolute monotonicity of the norm on $M_\infty(X)_\sa$ then implies that $\|x\|\leq 1$, as required. 
\end{proof}

%%\medskipI

The proof of the last two corollaries in this section demonstrate that Theorem \ref{thm:bidual-MOS} provides an easy passage from some results in ordered normed space to their matrix analogues. 
In fact, the following corollary is a direct application of Theorem \ref{thm:bidual-MOS} and Proposition \ref{prop:pos-unit-ball}.

%%\medskipI

\begin{cor}\label{cor:emb-into-bidual}
	Let $X$ be a SMOS and $\kappa_X:X\to X^{**}$ be the canonical map. 
	
	\smnoind
	(a) $\kappa_X$ is a completely order monomorphic complete isometry. 
	
	\smnoind
	(b) For $n\in \BN$, the images of $M_{n}(X)_+$ and $B_{M_{n}(X)}^+$ under $\kappa_X^{(n)}$ are weak$^*$-dense in $M_{n}(X^{**})_+$ and  $B_{M_{n}(X^{**})}^+$, respectively.
\end{cor}

\emph{In the remainder of this paper, we may sometimes ignore the map $\kappa_X$ and identify $X$ as a sub-SMOS of $X^{**}$ directly. }

%%\medskipI

Note that the proof of the weak$^*$-density of $\kappa_X^{(n)}\big(B_{M_{n}(X)}^+\big)$ in $B_{M_{n}(X^{**})}^+$ given in \cite[Proposition 12]{JN} requires the matrix analogue of the Bonsall theorem, but it is not needed in our alternative proof above.

%%\medskipI

%%\medskipI

\begin{cor}\label{cor:bidual-abs-mono}
	Let $X$ be a SMOS. 
	
	\smnoind
	(a) The norm on $M_\infty(X^{**})_\sa$ is absolutely monotone if and only if the norm on $M_\infty(X)_\sa$ is absolutely monotone. 
	
	\smnoind
	(b) $X^{**}$ is matrix regular if and only if the norm on $M_\infty(X)_\sa$ is absolutely monotone, and for each $n\in \BN$ as well as  $x\in M_n(X)_\sa$ with $\|x\| < 1$, there is $u\in B_{M_n(X^{**})}^+$ with $-u\leq x\leq u$. 
	In particular, if $X$ is matrix regular, then so is $X^{**}$.
\end{cor}
\begin{proof}
	(a) Since the absolute monotonicity passes to subspaces, the forward implication follows from Corollary \ref{cor:emb-into-bidual}(a). 
	For the backward implication, let us fix $n\in \BN$. 
	Proposition \ref{prop:Jameson}(b) implies
	\begin{equation*}\label{eqt:weak-st-cl-of-S(B)}
		\MS(B_{M_n(X)^{**}_\sa})  = \MS\big(\overline{B_{M_n(X)_\sa}}^{\sigma(M_n(X)_\sa^{**}, M_n(X)_\sa^*)}\big)  = \overline{\MS(B_{M_n(X)_\sa})}^{\sigma(M_n(X)_\sa^{**}, M_n(X)_\sa^*)}.	
	\end{equation*}
	Moreover, by Lemma \ref{lem:abs-mono}(a), we have $\MS(B_{M_n(X)_\sa}) \subseteq B_{M_n(X)_\sa}$. 
	Hence, the displayed relation above gives $\MS(B_{M_n(X)^{**}_\sa}) \subseteq B_{M_n(X)^{**}_\sa}$. 
	It then follows from Theorem \ref{thm:bidual-MOS} that 
	$\MS(B_{M_n(X^{**})_\sa}) \subseteq B_{M_n(X^{**})_\sa}$, and  Lemma \ref{lem:abs-mono}(a) tells us that the norm on $M_n(X^{**})_\sa$ is absolutely monotone.
	
	\smnoind
	(b) We will only establish the first statement as the second statement follows directly from the first one. 
	Suppose that $X^{**}$ is matrix regular. 
	Then by Lemma \ref{lem:matrix-reg}(b) and part (a) above, we know that the norm on $M_n(X)_\sa$ is absolutely monotone. 
	Moreover, the second condition as in the statement clearly holds.  
	Conversely, assume that the said two conditions hold for $X$.
	If $n\in \BN$, then Lemma \ref{lem:abs-mono}(a) and part (a) above imply that 
	$$O_{M_n(X)_\sa}\subseteq \MS(B_{M_n(X^{**})_\sa}) \subseteq B_{M_n(X^{**})_\sa}.$$
	On the other hand, the displayed relation  in part (a) and Theorem \ref{thm:bidual-MOS} implies that $\MS(B_{M_n(X^{**})_\sa})$ is $\sigma(M_n(X^{**})_\sa, M_n(X^*)_\sa)$-closed. 
	Therefore, it follows from the inclusions above and Corollary \ref{cor:dense-ball} that $B_{M_n(X^{**})_\sa} = \MS(B_{M_n(X^{**})_\sa})$. 
	Now, Lemma \ref{lem:matrix-reg}(a) ensures $X^{**}$ to be matrix regular.
\end{proof}

%%\medskipI

\section{From dual SMOS to dual Operator systems}

%%\medskipI

To study the dual space of operator systems, we also need the dual version of SMOS defined below.

%%\medskipI

\begin{defn}\label{defn:dual-MOS}
	(a) Suppose that $X$ is a complete SMOS. 
	When $X^*$ is equipped with the SMOS structure as in Remark \ref{rem:dual-SMOS}(b) together with the weak$^*$-topology $\sigma(X^*,X)$, we call it a \emph{dual SMOS}. 
	In this case,  $X$ is called the \emph{fixed predual} of $X^*$. 
	If $X^*$ happens to be a MOS, then $X^*$ is called a \emph{dual MOS}. 
	For dual SMOS $E$ and $F$, we denote by $\Morc_w(E,F)$ the set of all weak$^*$-continuous completely positive complete contractions from $E$ to $F$.

	\smnoind
	(b) A dual MOS $V$ is called a \emph{dual operator system} 
	(respectively, \emph{dual quasi-operator system}) if there is a completely order monomorphic  complete isometry (respectively, complete embedding) $\Psi: V \to \CL(H)$ for some Hilbert space $H$ such that $\Psi$ is a weak$^*$-homeomorphism from $V$ to $\Psi(V)$. 
	In the case when $\Psi$ is completely isometric and $\Psi(V)$ contains the identity of $\CL(H)$, we call $V$ a \emph{unital dual operator system}. 
\end{defn}

%%\medskipI

In other words, we define a dual operator system as a self-adjoint weak$^*$-closed subspace of some $\CL(H)$,  equipped with the induced matrix norm, the induced matrix cone and the induced weak$^*$-topology.
This definition is slightly different from the one in some literature (where the unital assumption is usually imposed; see e.g., \cite{dual}).

%%\medskipI

In \cite{Ng-dual-OS}, a dual SMOS (respectively, a dual MOS) $E$ is defined to be a SMOS (respectively, MOS) with a Banach space predual $X$ such that the involution is weak$^*$-continuous and the matrix cone is weak$^*$-closed. 
This virtually weaker definition agrees with the definition as in the above because of \cite[Corollary 3.8]{Ng-dual-OS}.

%%\medskipI

For a complete SMOS $X$, one has 
$$\Morc_w(X^*,M_n)= B_{M_n(X)}\cap M_n(X)_+ = B_{M_n(X)}^+;$$ 
note that the norm on $M_n(X^{**})$ is the one induced from the completely bounded norm on $\Morc(X^*,M_n)$, and we have Corollary \ref{cor:emb-into-bidual}(a). 

%%\medskipI

Observe that it is possible to have two different complete SMOS $X$ and $Y$ such that there is a completely isometric complete order isomorphism from $X^*$ onto $Y^*$ (even in the case when $X^*$ is a unital dual operator system; see \cite[Proposition 2.3]{dual}). 
Thus, the predual of a dual SMOS $F$ may not be unique if one only considers the SMOS structure on $F$.
Of course,  the predual is unique if the weak$^*$-topology is taken into account.

%%\medskipI

In the following, we will extend a construction in \cite{Ng-MOS} and state some facts about it. 
More precisely, for a SMOS $X$, we denote by $j_X$ the evaluation map from $X$ to the $C^*$-algebra 
\begin{equation}\label{eqt:def-A(X)}
	A(X):= {\bigoplus}_{n\in \BN} C\big(B_{M_n(X^*)}^+;M_n\big);
\end{equation}
here, we identify $B_{M_n(X^*)}^+$ with $\Morc(X,M_n)$, and equip it with the compact Hausdorff topology induced from weak$^*$-topology on $M_n(X)^*$ (see Theorem \ref{thm:bdd-pos-CCP}(a)).
Parts (a) and (b) of the following result can be regarded as an extension of \cite[Lemma 2.4(c)]{Ng-MOS} to SMOS (in the original version, the matrix cone on $X$ is assumed to be proper).

%%\medskipI

\begin{lem}\label{lem:Ng-MOS}
	Let $X$ be a SMOS. 
	
	\smnoind
	(a) $j_X$ is a self-adjoint complete contraction satisfying
	$$\big(j_X^{(\infty)}\big)^{-1}\big(M_\infty(A(X))_+\big)\cap M_\infty(X)_\sa = M_\infty(X)_+.$$ 
	
	\smnoind
	(b) $X$ is a MOS if and only if $j_X$ is injective. 
	
	\smnoind
	(c) $X$ is an operator system  (respectively, a quasi-operator system) if and only if $j_X$ is a complete isometry (respectively, complete embedding). 
\end{lem}
\begin{proof}
	(a) Clearly, $j_X$ is a completely positive complete contraction. 
	In particular, one has 
	$$M_\infty(X)_+ \subseteq \big(j_X^{(\infty)}\big)^{-1}\big(M_\infty(A(X))_+\big)\cap M_\infty(X)_\sa.$$ 
	Suppose on the contrary that there exist $n\in \BN$ and $x\in M_n(X)_\sa\setminus M_n(X)_+$ with $j_X^{(n)}(x)\in M_n(A(X))_+$. 
	As $M_n(X)_+$ is a norm closed cone, we can find $f\in (M_n(X)_\sa)^*$ such that $f(y) \geq 0$ for all $y\in M_n(X)_+$ but $f(x) < 0$. 
	If $g$ is the complexification of $f$, then $g\in M_n(X)^*_+$ and $g(x) < 0$. 
	By Theorem \ref{thm:bdd-pos-CCP}(a), the map $\Theta_g^X$ belongs to $\CB(X,M_n)\cap \CP(X,M_n)$. 
	On the other hand, Lemma \ref{lem:JN-Lemma5} and the proof of Theorem \ref{thm:bdd-pos-CCP}(a) tell us that $\beta_e (\Theta_g^X)^{(n)}(x) \beta_e^* = g(x) = f(x) < 0$. 
	This shows that $(\Theta_g^X)^{(n)}(x)\not\not\geq 0$. 
	By rescaling $\Theta_g^X$, we may assume that $\Theta_g^X\in B_{M_n(X^*)}^+$. 
	This gives the contradiction that $j_X^{(n)}(x)\notin M_n(A(X))_+$. 
	
	\smnoind
	(b) The forward implication follows from \cite[Lemma 2.4(c)]{Ng-MOS}.
	The backward implication follows from part (a) above and the fact that $A(X)$ is a MOS. 
	
	\smnoind
	(c) This follows from part (b) above as well as Lemma 2.4(d) and Theorem 2.6 of \cite{Ng-MOS}.
\end{proof}

%%\medskipI

\begin{cor}\label{cor:smaller-cone-oper-sys}
	Let $S$ be an operator system (respectively, a quasi-operator system). 
	If $T$ is the $^*$-operator space $S$ equipped with a closed matrix cone $M_\infty(T)_+$ that satisfies $M_\infty(T)_+\subseteq M_\infty(S)_+$, then $T$ is again an operator system (respectively, a quasi-operator system). 
\end{cor}
\begin{proof}
	By Lemma \ref{lem:Ng-MOS}(c), it suffices to show that $j_T$ is a complete isometry (respectively, complete embedding). 
	Since $M_\infty(T)_+\subseteq M_\infty(S)_+$, we know that $\Morc(S;M_n) \subseteq \Morc(T;M_n)$ for every $n\in \BN$. 
	From this, and the fact that $j_T$ is a complete contraction (see Lemma \ref{lem:Ng-MOS}(a)), we see that $\|j_S^{(\infty)}(x)\|\leq \|j_T^{(\infty)}(x)\|\leq \|x\|$ for any $x\in M_\infty(S) = M_\infty(T)$. 
	As $j_S$ is a complete isometry (respectively, complete embedding), we know that $j_T$ is a complete isometry (respectively, complete embedding). 
\end{proof}

%%\medskipI

\begin{defn}\label{defn:unitization}
	Let $X$ be a SMOS. 
	Consider $V$ to be a unital operator system and $\iota\in \Morc(X,V)$. 
	Then $(V, \iota)$ is called a \emph{partial unitization} of $X$ if for every unital operator system $W$ and $\varphi\in \Morc(X,W)$, there is a unique unital complete contraction $\varphi^1:V\to W$ such that $\varphi = \varphi^1\circ \iota$. 	
\end{defn}

%%\medskipI

\begin{rem}\label{rem:unitization}
	(a) In the case when $X$ is a MOS, the requirement for the definition of partial unitization in \cite[Definition 4.1]{Wern-subsp}  is formally stronger than the one in Definition \ref{defn:unitization}, but these two definitions  are the same, because the pair $(V, \iota)$ satisfying the requirement in Definition \ref{defn:unitization} is clear unique up to bijective unital complete isometry (due to the universal property in the definition), and the object satisfying the requirement in \cite[Definition 4.1]{Wern-subsp} always exists (see \cite[Lemma 4.8(c)]{Wern-subsp}).  
	
	\smnoind
	(b) The pair in Definition \ref{defn:unitization} was called ``partial unitalization'' in \cite[Definition 3.2]{Ng-MOS}, but we change the name to ``partial unitization'' to match with the terminology in \cite{Wern-subsp}. 
\end{rem}

%%\medskipI

The following is a combination of \cite[Proposition 3.4]{Ng-MOS} and some results from \cite[\S 4]{Wern-subsp}.
Since we consider the more general situation of SMOS here (instead of MOS), some arguments are required. 

%%\medskipI

\begin{prop}\label{prop:reg}
	Let $X$ be a SMOS and $j_X:X\to A(X)$ be the map as above. 
	For each $x\in M_\infty(X)$, we define
	\begin{equation}\label{eqt:def-norm-r}
		\left\|x\right\|^\rn:= \big\|j_X^{(\infty)}(x)\big\|=\sup\big\{\|\varphi^{(\infty)}(x)\|: \varphi\in \Morc(X;M_m); m\in \BN \big\}. 
	\end{equation}
	Set $Z_0:=\ker j_X=\{x\in X:\|x\|^\rn = 0\}  $.
	Denote $\iota_X:X\to X/Z_0$ the quotient map, and $\|\cdot\|_0^\rn$ the quotient matrix norm on $X/Z_0$ induced by $\|\cdot\|^\rn$.  
	
	\smnoind
	(a) $X\check{\ }:= \big(X/Z_0, \|\cdot\|_0^\rn\big)$ is an operator system with its matrix cone $M_\infty(X\check{\ })_+$ being $\iota_X^{(\infty)}\big(M_\infty(X)_+\big)$. 
	Moreover, $X\check{\ }$ can be identified with $j_X(X)$ as operator systems; in other words, one may identify 
	$\iota_X = j_X$.
	Furthermore, 
	\begin{equation}\label{eqt:norm-jX}
		\|x\|^\rn
		= \sup\Big\lbrace\Big\|f\Big(\begin{matrix}
			0&x\\
			x^*&0
		\end{matrix}\Big)\Big\|:f\in B_{M_{2n}(X)^*}^+\Big\rbrace
		\quad (x\in M_n(X);n\in \BN).  
	\end{equation}
	
	\smnoind
	(b) $X$ is an operator system (respectively, a quasi-operator system) if and only if $\iota_X$ is a complete isometry (respectively, complete embedding). 
	
	\smnoind
	(c) If $T$ is an operator system, then for any $\varphi\in \Morc(X,T)$, there is a unique $\varphi\check{\ }\in \Morc( X\check{\ }, T)$ with $\varphi = \varphi\check{\ }\circ \iota_X$. 
	
	\smnoind
	(d) Let $T$ be an operator system.
	If $\varphi\in \Morc(X,T)$ satisfying: for each $n\in \BN$ and $\psi\in \Morc(X,M_n)$, there is $\bar\psi\in \Morc(T,M_n)$ with $\psi = \bar \psi\circ \varphi$, then the map $\varphi\check{\ }:X\check{\ }\to T$ as in part (c) is a completely order monomorphic complete isometry. 
\end{prop}
\begin{proof}
	(a) Let $X_0$ be the vector space $X/Z_0$ equipped with the quotient matrix norm $\|\cdot\|_0$ induced from the matrix norm on $X$ and the quotient matrix cone $\iota_X^{(\infty)}\big(M_\infty(X)_+\big)$. 
	Let $k_X:X_0\to A(X)$ be the injection satisfying 
	$$j_X = k_X\circ \iota_X.$$ 
	This, together with the surjectivity of $\iota_X$, the continuity of $k_X$ and Lemma \ref{lem:Ng-MOS}(a), implies that $\iota_X^{(\infty)}\big(M_\infty(X)_+\big)$ is norm closed. 
	Furthermore, the relation $M_\infty(X)_+ \cap - M_\infty(X)_+ \subseteq \ker j_X^{(\infty)} = M_\infty(Z_0)$ tells us  that $X_0$ is a MOS. 
	
	It is obvious that $\iota_X\in \Morc(X,X_0)$ produces a continuous injection 
	$$\tilde \iota_X: \Morc(X_0,M_k) \to \Morc(X,M_k) \qquad (k\in \BN).$$ 
	This map $\tilde \iota_X$ is actually bijective, because  of the fact that for every $\phi\in  \Morc(X,M_k)$, there is $\psi\in  \Morc(X_0,M_k)$ with $\phi = \psi\circ \iota_X.$
	From these, one has a $^*$-isomorphism $\iota_X^A: A(X) \to A(X_0)$ such that 
	$$\iota_X^A\circ k_X = j_{X_0}.$$
	It follows from Lemma \ref{lem:Ng-MOS}(a) that $j_{X_0}$ is a complete order monomorphism.
	Thus, $k_X:X\check{\ } \to A(X)$ is a complete order monomorphism (recall that $X\check{\ }$ and $X_0$ have the same matrix cone).
	On the other hand, 
	Relation \eqref{eqt:def-norm-r} tells us that the matrix norm $\|\cdot\|_0^\rn$ on $X/Z_0$ coincides with the one induced from the map $k_X$; in other words, $k_X:X\check{\ } \to A(X)$ is a complete isometry. 
	Consequently, $X\check{\ }$ can be identified with the operator system $k_X(X_0) = j_X(X)$.

	Concerning the last statement in part (a), let us first show that Equality \eqref{eqt:norm-jX} holds when $X$ is a MOS.
	Indeed, in this case, we know from \cite[Theorem 3.3]{Ng-MOS} that $\big(j_{X}(X) + \BC 1_{X}, j_{X}\big)$ is the partial unitization of the MOS $X$, where $1_{X}$ is the identity of $A(X)$. 
	Thus, $\|\cdot\|^\rn$ is the matrix norm induced by the partial unitization of $X$. 
	On the other hand, \cite[Lemma 4.8(c)]{Wern-subsp} tells us that the matrix norm on $X$ induced by the partial unitization of $X$ is the one given by the right hand side of Equality \eqref{eqt:norm-jX} (see Remark \ref{rem:unitization}(a)). 
	
	For the general case, we first note that the surjectivity of the map $\tilde \iota_X$ above and Theorem \ref{thm:bdd-pos-CCP}(a), together with the identification 
	$$M_{2n}(X^*)_+ = \CP(X,M_{2n})\cap \CB(X,M_{2n}),$$ 
	imply that for each $n\in \BN$ and $f\in M_{2n}(X)^*_+$, there is $g\in M_{2n}(X_0)^*_+$ with
	$$f = g\circ \iota_X^{(2n)}.$$
	As the matrix norm on $X_0$ is the one induced from $X$ via $\iota_X$, we see that if $f\in B_{M_{2n}(X)^*}^+$, then $g\in B_{M_{2n}(X_0)^*}^+$. 
	Consequently, by applying the MOS case above to $X_0$, we have (since $\iota_X^A$ is completely isometric)
	\begin{align*}
		\big\|j_{X}^{(n)}(x)\big\| 
		& = \big\|j_{X_0}^{(n)}\big(\iota_X^{(n)}(x)\big)\big\|
		= \sup\Big\lbrace\Big\|g\circ \iota_X^{(2n)}\Big(\begin{matrix}
			0&x\\
			x^*&0
		\end{matrix}\Big)\Big\|:g\in B_{M_{2n}(X_0)^*}^+\Big\rbrace\\
		& 	= \sup\Big\lbrace\Big\|f\Big(\begin{matrix}
			0&x\\
			x^*&0
		\end{matrix}\Big)\Big\|:f\in B_{M_{2n}(X)^*}^+\Big\rbrace
		\qquad  \quad (n\in \BN; x\in M_n(X)). 
	\end{align*}

	\smnoind
	(b)  It follows from part (a) that one may identify $\iota_X$ with $j_X$. 
	The conclusion then follows from Lemma \ref{lem:Ng-MOS}(c). 
	
	\smnoind
	(c) The uniqueness of $\varphi\check{\ }$ follows from the surjectivity of $\iota_X:X\to X\check{\ }$. 
	For the existence, we first note that compositions with $\varphi$ produces a continuous map $\tilde \varphi_n: \Morc(T,M_n) \to \Morc(X,M_n)$ $(n\in \BN)$, and hence a $^*$-homomorphism 
	$$\varphi^A: A(X)\to A(T)$$ 
	satisfying $\varphi^A \circ j_X = j_T\circ \varphi$. 
	By parts (a) and (c) of Lemma \ref{lem:Ng-MOS}, $j_T$ is a completely order monomorphic complete isometry and one may identify $T$ with $j_T(T)$. 
	Now, $\varphi^A|_{j_X(X)}$ induces a map $\varphi\check{\ }\in \Morc( X\check{\ }, T)$ satisfying $j_T\circ\varphi\check{\ } =  \varphi^A|_{j_X(X)}$. 
	
	\smnoind
	(d) The hypothesis means that the map $\tilde \varphi_n$ as in the proof of part (c) above is surjective, for every $n\in \BN$. 
	Therefore, the map $\varphi^A: A(X)\to A(T)$ is an injective $^*$-homomorphism (and hence is a completely order monomorphic complete isometry). 
	Since $j_T\circ\varphi\check{\ } =  \varphi^A|_{j_X(X)}$ (see the proof of part (c)) and $j_T$ is a completely order monomorphic complete isometry, the conclusion follows. 
\end{proof}

%%\medskipI

\begin{rem}\label{rem:conn-with-lit}
	Relation \eqref{eqt:norm-jX} means that the matrix semi-norm $\|\cdot\|^\rn$ coincides with the modified numerical radius $\nu_X(\cdot)$ as defined in \cite[\S 3]{Wern-subsp}. 
	This helps us to make connections between our results with those in \cite{Wern-subsp} and \cite{Werner}. 
	In particular, in \cite[Definition 2.15]{CvS}, Connes and van Suijlekom defined an operator system to be a MOS $X$ whose matrix norm $\|\cdot\|$ coincides with $\nu_X(\cdot)$.
	Hence, we learn from Propositon \ref{prop:reg}(b) that the meaning of operator systems in \cite[Definition 2.15]{CvS} is the same as the one in this article. 
\end{rem}

%%\medskipI

Due to Proposition \ref{prop:reg}(c), one may regard $X\check{\ }$ as the ``regularization'' of the SMOS $X$ and it is the reason behind the notation $\|\cdot\|^\rn$. 

%%\medskipI

The corresponding ``weak$^*$-regularization'' exists for any  dual SMOS $F$.  
In fact, for the case of dual MOS, this has already been established in \cite[Proposition 3.6]{Ng-dual-OS}. 
In the following result, we will extend this to dual SMOS. 
Moreover, we will show that the ``weak$^*$-regularization'' is actually the completion of the ``regularization'' of $F$ under a suitable locally convex topology.

Let $Y$ be a complete SMOS. 
Suppose that $((Y^*)\check{\ },\iota_F)$ is as in Proposition \ref{prop:reg}, and $Y_1$ is the norm closed linear span in $((Y^*)\check{\ })^*$ of the subset 
$$\{\chi\circ \psi\check{\ }:  \psi\in \Morc_w(Y^*,M_m); \chi\in M_m^*; m\in \BN \};$$ 
here $\psi\check{\ }\in \Morc((Y^*)\check{\ }, M_m)$ is the map given by Proposition \ref{prop:reg}(c). 

%%\medskipI

\begin{thm}\label{thm:reg-dual}
	Let $Y$ be a complete SMOS, and $Y_1$ be as in the above. 
	
	\smnoind
	(a) $Y_1$ will separate points of $(Y^*)\check{\ }$.

	\smnoind
	(b) Let $Y^\rd$ be the completion of $(Y^*)\check{\ }$ under $\sigma((Y^*)\check{\ }, Y_1)$. 
	Then there is a dual operator system structure on $Y^\rd$ with predual $Y_1$ such that $(Y^*)\check{\ }$ is an operator subsystem of $Y^\rd$. 
	Furthermore,  $\iota_{Y^*}:Y^*\to Y^\rd$ is weak$^*$-continuous, and $\big(\iota_{Y^*}^{(\infty)}\big)^{-1}\big(M_\infty(Y^\rd)_+\big)\cap M_\infty(Y^*)_\sa = M_\infty(Y^*)_+$.
	
	\smnoind
	(c) $Y^*$ is a dual operator system (respectively, a dual quasi-operator system) if and only if $\iota_{Y^*}$ is a complete isometry (respectively, complete embedding). 
	In this case, $Y^\rd = (Y^*)\check{\ }$, and $\iota_{Y^*}:Y^*\to Y^\rd$ is a weak$^*$-homeomorphism. 	
\end{thm}
\begin{proof}
%Set $F$ to be the dual SMOS $Y^*$.
	We denote by $\mu_{Y^*}$ the canonical map (given by evaluations) from $Y^*$ to the von Neumann algebra  
	\begin{equation}\label{eqt:def-N(F)}
		N(Y^*):={\bigoplus}_{n\in \BN}^{\ell^\infty} \ell^\infty\big(B_{M_n(Y)}^+;M_n\big)
	\end{equation}
	(here, $B_{M_n(Y)}^+$ is identified with $\Morc_w(Y^*,M_n)$). 
	Then 
	\begin{equation}\label{eqt:pre-dual}
	N(Y^*)_* ={\bigoplus}_{n\in \BN}^{\ell^1} \ell^1\big(B_{M_n(Y)}^+;M_n^*\big),	
	\end{equation}
	where $N(Y^*)_*$ is the predual of $N(Y^*)$. 
	Observe that for $y\in M_\infty(Y^*)_\sa$, 
	$$\mu_{Y^*}^{(\infty)}(y)\geq 0\quad \text{if and only if} \quad \omega^{(\infty)}(y)\geq 0\text{ for all }\omega\in B_{M_\infty(Y)}^+.$$ 
	Therefore, Corollary \ref{cor:emb-into-bidual}(b) tells us that for every $y\in M_\infty(Y^*)_\sa$, 
	\begin{equation}\label{eqt:j-pos-mu-pos}
		\mu_{Y^*}^{(\infty)}(y)\in M_\infty(N(Y^*))_+\quad \text{ if and only if }\quad j_{Y^*}^{(\infty)}(y) \in M_\infty(A(Y^*))_+
	\end{equation}
	(as $j_{Y^*}^{(\infty)}(y) \geq 0$ if and only if $\varphi^{(\infty)}(y)\geq 0$ for all $\varphi\in B_{M_\infty(Y^{**})}^+$).
	Moreover,  
	$$\|\mu_{Y^*}^{(\infty)}(x)\| = \sup\big\{\|\omega^{(\infty)}(x)\|: \omega\in B_{M_\infty(Y)}^+ \big\} \qquad (x\in M_\infty(Y^*)),$$
	Corollary \ref{cor:emb-into-bidual}(b) and Relation \eqref{eqt:def-norm-r} imply that for each $x\in M_\infty(Y^*)$, 
	\begin{equation}\label{eqt:norm-mu-F}
		\big\|\mu_{Y^*}^{(\infty)}(x)\big\| = \sup \big\{\|\phi^{(\infty)}(x)\|: \phi\in B_{M_\infty(Y^{**})}^+ \big\} = \big\|j_{Y^*}(x)\big\| = \|x\|^\rn.
	\end{equation}  
	Consequently, by Proposition \ref{prop:reg}(a), we may identify 
	\begin{equation}\label{eqt:iota=mu}
		j_{Y^*}(Y^*) =  (Y^*)\check{\ } = \mu_{Y^*}(Y^*) \quad \text{and} \quad j_{Y^*} = \iota_{Y^*} = \mu_{Y^*}.
	\end{equation}
	
	\smnoind
	(a) For $m\in \BN$, $\psi\in B_{M_m(Y)}^+$ and $\chi\in M_m^*$, we consider the element $\delta^\chi_\psi\in \ell^1\big(B_{M_m(Y)}^+;M_m^*\big)\subseteq N(Y^*)_*$ defined by
	$$\delta^\chi_\psi(\phi) := \begin{cases}
		\chi \ & \text{when } \phi = \psi\\
		0 &\text{otherwise.}
	\end{cases}$$
	It follows from $\chi\circ \psi = \delta_\psi^\chi\circ \mu_{Y^*}$ and Proposition \ref{prop:reg}(c) that $\chi\circ \psi\check{\ } = \delta_\psi^\chi|_{(Y^*)\check{\ }}$. 
	Since the complex linear span of 
	$\big\{\delta^\chi_\psi: \psi\in B_{M_m(Y)}^+; \chi\in M_m^*; m\in \BN\big\}$ 
	is norm dense in $N(Y^*)_*$, we know that 
	\begin{equation}\label{eqt:Y-1}
	Y_1 = \{\omega|_{(Y^*)\check{\ }}:\omega\in N(Y^*)_* \}, 	
	\end{equation}
and hence it will separate points of $(Y^*)\check{\ } = \mu_{Y^*}(Y^*)$.  
	
	\smnoind
	(b) We learn from \eqref{eqt:iota=mu} and \eqref{eqt:Y-1} that $\sigma((Y^*)\check{\ }, Y_1)$ can be seen as the restriction of $\sigma(N(Y^*), N(Y^*)_*)$ on $(Y^*)\check{\ } = \mu_{Y^*}(Y^*)$. 
	Thus, we may identify
	\begin{equation}\label{eqt:alt-F-reg}
		Y^\rd = \overline{\mu_{Y^*}(Y^*)}^{\sigma(N(Y^*),N(Y^*)_*)} \subseteq N(Y^*).
	\end{equation}
	From this, we know that $Y^\rd$ is a dual operator system containing $(Y^*)\check{\ }$ as a subsystem. 
	Moreover, since $Y$ is complete, using the argument of \cite[Propositon 3.6(a)]{Ng-dual-OS}, we know that $\mu_{Y^*}$ is weak$^*$-continuous (note that the argument for this fact does not depend on the properness of the matrix cone on $Y^*$). 
	This gives the weak$^*$-continuity of $\iota_{Y^*}:Y^*\to Y^\rd$. 
	Finally, the last equality in the statement follows from Lemma \ref{lem:Ng-MOS}(a) and Relation \eqref{eqt:iota=mu}.

	\smnoind
	(c) If $Y^*$ is a dual operator system (respectively, a dual quasi-operator system), then Proposition \ref{prop:reg}(b) implies that $\iota_{Y^*}$ is a complete isometry (respectively,  a complete embedding). 
	Conversely, suppose that $\iota_{Y^*} = j_{Y^*}$ is a complete isometry (respectively,  a complete embedding). 
	Then $Y^*$ is a MOS (see Lemma \ref{lem:Ng-MOS}(b)) and hence is a dual MOS. 
	Now, as $\mu_{Y^*} = \iota_{Y^*}$ is a complete isometry (respectively,  a complete embedding), one can use \cite[Proposition 3.6(b)]{Ng-dual-OS} to conclude that $Y^*$ is a dual operator system (respectively, a dual quasi-operator system). 
	The last statement in this part follows from \cite[Proposition 3.6(c)]{Ng-dual-OS}.
\end{proof}

\begin{thm}\label{thm:prop-os-dual}
	Let $Y$ be a complete SMOS and $Y^\rd$ be as in Theorem \ref{thm:reg-dual}. 
	Let $V$ be a dual operator system and $\varphi\in \Morc_w(Y^*,V)$.
	
	\smnoind
	(a) The map $\varphi\check{\ }:(Y^*)\check{\ }\to V$ as in Proposition \ref{prop:reg}(c) is $\sigma((Y^*)\check{\ },(Y^\rd)_*)$-$\sigma(V,V_*)$-continuous. 
	If we denote by $\overline{\varphi}: Y^\rd\to V$ the weak$^*$-continuous extension of $\varphi\check{\ }$, then $\overline{\varphi}\in \Morc_w(Y^\rd, V)$ and $\overline{\varphi}\circ \iota_{Y^*} = \varphi$.  
	
	\smnoind
	(b) Suppose that for each $n\in \BN$ and $\psi\in \Morc_w(Y^*,M_n)$, there exists $\hat\psi\in \Morc_w(V,M_n)$ with $\psi = \hat \psi\circ \varphi$. 
	Then $\overline{\varphi}:Y^\rd\to \overline{\varphi}(Y^\rd)$  is a completely isometric weak$^*$-homemorphic complete order isomorphism. 
	
	\smnoind
	(c) $\iota_{Y^*}^{(\infty)}(M_\infty(Y^*)_+)$ is weak$^*$-dense in $M_\infty(Y^\rd)_+$. 
\end{thm}
\begin{proof}
	(a) Similar to the argument of Proposition \ref{prop:reg}(c), the weak$^*$-continuous map $\varphi$ induces a normal $^*$-homomorphism 
	$$\varphi^N: N(Y^*)\to N(V)$$
	satisfying $\varphi^N\circ \mu_{Y^*} = \mu_V\circ \varphi$. 
	Theorem \ref{thm:reg-dual}(c) ensures that $\mu_V:V\to \mu_V(V)$ is a completely isometric weak$^*$-homemorphic complete order isomorphism. 
	Thus, one can find $\bar \varphi\in \Morc_w( Y^\rd, V)$ with  $\varphi^N|_{Y^\rd} = \mu_V\circ \bar \varphi$ (see \eqref{eqt:alt-F-reg}), and 
$$\mu_V\circ \bar \varphi\circ \iota_{Y^*} = \varphi^N \circ \iota_{Y^*} = \varphi^N \circ \mu_{Y^*} = \mu_V\circ \varphi,$$
(see \eqref{eqt:iota=mu}).
From this, the injectivity of $\mu_V$ and the uniqueness of $\varphi\check{\ }$,  we know that $\bar \varphi$ extends $\varphi\check{\ }$ (see also \eqref{eqt:iota=mu}).

	\smnoind
	(b) Similar to the argument of Proposition \ref{prop:reg}(d), the assumption for $\varphi$ implies that the normal $^*$-homomorphism $\varphi^N$ in the proof of part (a) above is a completely order monomorphic complete isometry. 
	As in the proof of part (b) above, we know that $\mu_V$ is a completely order monomorphic complete isometry and that  
	$$\varphi^N|_{Y^\rd} = \mu_V\circ \overline{\varphi}.$$ 
	Hence, $\overline{\varphi}$ is a completely order monomorphic complete isometry. 
	Since $\overline{\varphi}$ is weak$^*$-continuous and isometric, $\overline{\varphi}$ is also a weak$^*$-homeomorphism from $Y^\rd$ onto $\overline{\varphi}(Y^\rd)$ (see Lemma \ref{lem:image-weak-st-cont-bdd-below}). 
	
	\smnoind
	(c) Let $W$ be the dual $^*$-operator space $Y^\rd$, and we set $M_n(W)_+$ to be the weak$^*$-closure of $\iota_{Y^*}^{(n)}(M_n(Y^*)_+)$ in $M_n(W)_\sa$ ($n\in \BN$).  
	The last equality in Theorem \ref{thm:reg-dual}(b) implies that $M_\infty(Y^*)_+\subseteq M_\infty(Y^\rd)_+$, and so, 
	$$M_\infty(W)_+\subseteq M_\infty(Y^\rd)_+.$$ 
	Since $Y^\rd$ is a dual operator system, we know from \cite[Corollary 3.7(a)]{Ng-dual-OS} that $W$ is again a dual operator system. 
	Consider 
	$\varphi:Y^*\to W$ 
	to be  the map $\iota_{Y^*}:{Y^*}\to Y^\rd$. 
	Then $\varphi\in \Morc_w(Y^*,W)$.
	As $Y^*$ is weak$^*$-dense in $Y^\rd$, the induced completely positive map $\overline{\varphi}\in \Morc_w(Y^\rd,W)$ as in part (a) above is the identity map. 
	This produces $M_\infty(Y^\rd)_+\subseteq M_\infty(W)_+$.
\end{proof}

%%\medskipI

%%\medskipI

Theorem \ref{thm:reg-dual} tells us that one may use the quantity $$n^\os(Y^*):={\inf}_{n\in \BN}\ \! {\inf}_{z\in S_{M_n(Y^*)}}\ \! \|z\|^\rn,$$
(where $S_{M_n(Y^*)}$ is the unit sphere of $M_n(Y^*)$) to measure how far a dual SMOS $Y^*$ is from being a dual operator system. 
In particular, $n^\os(Y^*) = 1$ (respectively, $n^\os(Y^*) > 0$) if and only if $Y^*$ is a dual operator system (respectively, a dual quasi-operator system). 
%Notice that $n^\os(F) = 0$ if and only if for any $\epsilon > 0$, there exist $n\in \BN$ and $z\in S_{M_n(F)}$ such that for every $m\in \BN$ and $\varphi\in \Morc(X;M_m)$, one has  $\|\varphi^{(n)}(z)\| < \epsilon$. 

%%\medskipI

\begin{cor}\label{cor:dualizable}
	Let $Y$ be a complete SMOS.
	If $Y^*$ is an operator system under the dual SMOS structure, then it is a dual operator system (under the dual SMOS structure). 
	On the other hand, if $Y^*$ is a quasi-operator system,  then $Y^*$ can be turned into a dual operator system under the dual matrix cone, the topology $\sigma(Y^*,Y)$ and the matrix norm $\|\cdot\|^\rn$ as in Relation \eqref{eqt:norm-jX} (i.e. the modified numerical radius as in \cite{Wern-subsp}). 
\end{cor}

%%\medskipI

In fact, this follows directly from Proposition \ref{prop:reg}(b) and Theorem \ref{thm:reg-dual}(c) (since Theorem \ref{thm:reg-dual}(b) tells us that $Y^\rd$ is a dual operator system). 

%%\medskipI

Corollary \ref{cor:dualizable} applies, in particular, to the case when $Y$ is a $C^*$-algebra, because the dual space of a $C^*$-algebra is always a quasi-operator system (see \cite[Corollary 2.8]{Ng-MOS}). 

%%\medskipI

\section{The dual functor for general operator systems}

%%\medskipI

In this section, we will use Theorem \ref{thm:reg-dual} to extend the ``duality functor'' in \cite{Ng-dual-OS} (for dualizable operator systems; see the Introduction) to all operator systems, and we will call it the ``dual functor'' instead. 
We will then give some further study of this functor. 

%%\medskipI

Suppose that $S$ is an operator system. 
\begin{quote}
	We will consider the completion $\tilde S$ of $S$ (which is a $^*$-operator space) as an SMOS with its matrix cone $M_\infty(\tilde S)_+$ being the norm closure of $M_\infty(S)_+$ in $M_\infty(\tilde S)_\sa$.  
\end{quote}

\begin{rem}\label{rem:comp-OS}
Since $S$ can be regarded as an operator subsystem of some $\CL(H)$, we learn from Corollary \ref{cor:smaller-cone-oper-sys} that $\tilde S$ is an operator system under the above matrix cone. 
\end{rem}

Let us once again make clear the following convention:  	
\begin{quote}
	We equip $S^* = \tilde S^*$ with the weak$^*$-topology $\sigma(S^*, \tilde S)$ (instead of $\sigma(S^*, S)$). 
\end{quote}
Under  this weak$^*$-topology, $S^*$ is a dual SMOS. 
As in the proof of Theorem \ref{thm:reg-dual}(b), one may regard $\tilde S^\rd$ as the weak$^*$-closure of $\mu_{\tilde S^*}(\tilde S^*)$ in the von Neumann algebra $N(\tilde S^*)$ (see Relation \eqref{eqt:def-N(F)}), and we may identify $\mu_{\tilde S^*}$ with 
\begin{equation}\label{eqt:def-iota-incomp}
\iota_{S^*}: S^* \to (S^*)\check{\ }\subseteq S^\rd, \quad \text{where} \quad S^\rd:= \tilde S^\rd.
\end{equation}

%%\medskipI

For operator systems $S$ and $T$ as well as $\varphi\in \Morc(S,T)$, we define 
$\varphi^\rd$ to be the map $\overline{\iota_{S^*}\circ \varphi^*}\in \Morc_w(T^\rd, S^\rd)$ (see Theorem \ref{thm:prop-os-dual}(a)). 
In particular, $\varphi^\rd$ is the unique map in $\Morc_w(T^\rd, S^\rd)$ satisfying
\begin{equation}\label{eqt:phi-d}
	\varphi^\rd \circ \iota_{T^*} = \iota_{S^*}\circ \varphi^*.
\end{equation}
It is easy to see that $(S,T,\varphi)\mapsto (T^\rd, S^\rd, \varphi^\rd)$ is a contravariant functor from the category of operator systems (whose morphisms are  completely positive complete contractions) to the category of dual operator systems (whose morphisms are weak$^*$-continuous completely positive complete contractions). 

%%\medskipI

\begin{defn}\label{defn:dual-functor}
	The contravariant functor $(S,T,\varphi)\mapsto (T^\rd, S^\rd, \varphi^\rd)$ is called the \emph{dual functor} for operator systems. 
\end{defn}

%%\medskipI

There is an unpleasant feature of this functor that $S^\rd = \{0\}$ whenever $S_+ = \{0\}$, regardless of the size of $S$ (because in this case, we have $$M_m(S^*)_+ = \CB(S;M_m)\cap \CP(S,M_m) = M_m(S^*)_\sa,$$ which implies that $\Morc_w(S^*,M_m) = \{0\}$ for any $m\in \BN$,  and hence $\mu_F$ as in the proof of Theorem \ref{thm:reg-dual} is the zero map). 

%%\medskipI

\begin{eg}\label{eg:not-unique}
	Let $U$ (respectively, $V$) be the subspace of $M_2(\BC)$ (respectively, $M_3(\BC)$) consisting of matrices whose diagonal entries being zero. 
	Then $U^\rd = \{0\} = V^\rd$. 
\end{eg}

%%\medskipI

Nevertheless, we have the following direct application of \cite[Theorem 3.9(a)]{Ng-dual-OS} and Corollary \ref{cor:dualizable}, which implies that $S^\rd$ is big enough when $S$ is dualizable. 
Observe that any one of the conditions in the statement will imply $X^*$ to be a dual MOS (see e.g. \cite[Proposition 3.5]{Ng-dual-OS}), and hence \cite[Theorem 3.9(a)]{Ng-dual-OS} can be applied. 

%%\medskipI

\begin{lem}\label{lem:dualizable}
	Let $X$ be a complete SMOS. 
	The following are equivalent: 
	\begin{itemize}
		\item $\iota_{X^*}: X^*\to X^\rd$ is a complete embedding;
		
		\item $B_{M_\infty(X)}^+ - B_{M_\infty(X)}^+$ is a norm zero neighborhood of $M_\infty(X)_\sa$; i.e., there exists $r > 0$ with $rB_{M_\infty(X)_{sa}}\subseteq B_{M_\infty(X)}^+ - B_{M_\infty(X)}^+$;
		
		\item the norm closure of $B_{M_\infty(X)}^+ - B_{M_\infty(X)}^+$ is a norm zero neighborhood of $M_\infty(X)_\sa$;

		\item $X^*$ is a dual quasi-operator system;
		
		\item $X^*$ is a quasi-operator system. 
	\end{itemize}
\end{lem}

%%\medskipI

Let us extend the definition of ``dualizbility'' to all complete SMOS.

%%\medskipI

\begin{defn}\label{defn:dual}
	A complete SMOS $X$ is said to be \emph{dualizable} if and only if it
	satisfies any one of the equivalent conditions in Lemma \ref{lem:dualizable}.
\end{defn}

%%\medskipI

In the following, we will have a closer look at the case when $\iota_{S^*}:S^*\to S^\rd$ is injective, or is bounded below or is a complete embedding, when $S$ is an (not necessarily complete) operator system. 
%We first note that $\iota_{S^*} = \iota_{\tilde S^*}$.

%%\medskipI

\begin{prop}\label{prop:mu-S-st}
	Suppose that $S$ is an operator system. 
	
	\smnoind
	(a) $B_{M_\infty(S^*)}^+ - B_{M_\infty(S^*)}^+$ is a norm zero neighborhood of $M_\infty(S^*)_\sa$.

	\smnoind
	(b) $M_\infty(S^\rd)_+ - M_\infty(S^\rd)_+$ is weak$^*$-dense in $M_\infty(S^\rd)_\sa$.

	\smnoind
	(c) Suppose that $S$ is complete. 
Then $\iota_{S^*}$ is injective if and only if $M_\infty(S)_+ - M_\infty(S)_+$ is norm dense in $M_\infty(S)_\sa$. 
\end{prop}
\begin{proof}
As in Proposition \ref{prop:reg} and \eqref{eqt:iota=mu}, we identify \begin{equation}\label{eqt:iota-ti-S-*}
	j_{S^*} = \iota_{S^*} = \iota_{\tilde S^*} = \mu_{\tilde S^*}:S^* \to (\tilde S^*)\check{\ }\subseteq  S^\rd \subseteq N(\tilde S^*).
\end{equation}

\smnoind
	(a) If $\Phi: S\to \CL(H)$ is a completely isometric complete order monomorphism, then $\Phi^{**}: S^{**}\to \CL(H)^{**}$ is a completely isometric completely positive map. 
	Hence, $S^{**}$ is a dual operator system, by Corollaries \ref{cor:smaller-cone-oper-sys} and \ref{cor:dualizable}. 
	This part then follows from Lemma \ref{lem:dualizable} (when applying to the complete MOS $S^*$).

	\smnoind
	(b) This follows from part (a) above and the facts that $\iota_{S^*}^{(\infty)}(M_\infty(S^*)_\sa)$ is weak$^*$-dense in $M_\infty(S^\rd)_\sa$ (because $\iota_{S^*}(S^*)$ is weak$^*$-dense in $\tilde S^\rd = S^\rd$) and that $\iota_{S^*}$ is completely positive. 
	
	\smnoind
	(c) It follows from \cite[Proposition 3.5]{Ng-dual-OS} that $M_\infty(S)_+ - M_\infty(S)_+$ is norm dense in $M_\infty(S)_\sa$ if and only if $S^*$ is a MOS. 
	Hence, the conclusion follows from Lemma \ref{lem:Ng-MOS}(b). 
\end{proof}

\begin{prop}\label{prop:iota-bdd-below}
For an operator system $S$, the following are equivalent:
	\begin{enumerate}[label=\ \ C\arabic*).]
		\item $\iota_{S^*}$ is bounded below;		
		
		\item $\tilde S_\sa = \tilde S_+ - \tilde S_+$ (i.e., $\tilde S$ has a generating cone);
		
		\item $B_{\tilde S}^+ - B_{\tilde S}^+$ is a norm zero neighborhood of $\tilde S_\sa$;	
			
		\item the norm-closure $\overline{B_S^+ - B_S^+}$ is a norm zero neighborhood of $S_\sa$. 
	\end{enumerate}
	In this case, the map $\iota_{S^*}:S^* \to S^\rd$ is a weak$^*$-homeomorphic complete order isomorphism (which gives $(S^*)\check{\ } = S^\rd$) and $S^\rd$ has a generating cone.
\end{prop}
\begin{proof}
We continuous to use the identifications as in \eqref{eqt:iota-ti-S-*}.

	\smnoind
	(C1) $\Rightarrow$ (C2). 
	Consider $Q:N(\tilde S^*)\to \ell^\infty(B_{\tilde S}^+;\BC)$ to be the surjective normal $^*$-homomorphism given by truncation (see \eqref{eqt:def-N(F)}). 
	Define
	$$\Phi:= Q\circ \iota_{S^*}|_{S^*_\sa}: S^*_\sa  \to \ell^\infty(B_{\tilde S}^+;\BR).$$
	Let  $f\in \tilde S^*_\sa = S^*_\sa$ and $\epsilon > 0$.
	There are $n\in \BN$ and $\psi\in \Morc_w(\tilde S^*,M_n)$ with 
	$$\|\iota_{S^*}(f)\| \leq \|\psi(f)\| + \epsilon.$$
	Moreover, since $\psi(f)\in (M_n)_\sa$, there exists $\xi\in \BC^n$ with $\|\xi\| =1$ satisfying 
	$$\|\psi(f)\| = |\langle \psi(f)\xi, \xi\rangle|.$$ 
	The weak$^*$-continuity of $\psi$ ensures that $\psi_\xi: g\mapsto \langle \psi(g)\xi, \xi\rangle$ belongs to $B_{\tilde S}^+$.
	Since $\|\iota_{S^*}(f)\| \leq |\psi_\xi(f)| + \epsilon$, we have 
	$$\|\iota_{S^*}(f)\| \leq \|\Phi(f)\| + \epsilon.$$
	Conversely, one obviously has $\|\Phi(f)\|\leq \|\iota_{S^*}(f)\|$. 
	Hence, 
	$$\|\Phi(f)\| = \|\iota_{S^*}(f)\| \qquad (f\in \tilde S^*_\sa).$$
	This relation and Condition (C1) imply that  the weak$^*$-continuous positive map $\Phi$ is bounded below. 
	Consequently, $\Phi: \tilde S^* \to \Phi(\tilde S^*)$ is a weak$^*$-homeomorphism (see Lemma \ref{lem:image-weak-st-cont-bdd-below}). 
	The conclusion now follows from the equivalence of Conditions (2) and (4) in Proposition \ref{prop:dual-quasi-func-sys}.

	\smnoind
	(C2) $\Rightarrow$ (C1). 
Denote by $\tilde S_1$ the predual of $S^\rd=\tilde S^\rd$. 
As $\iota_{S^*}: \tilde S^* \to S^\rd$ is a weak$^*$-continuous completely positive complete contraction (see Theorem \ref{thm:reg-dual}(b)), there exists a completely positive complete contraction $\varphi:\tilde S_1\to \tilde S$ (by Corollary \ref{cor:emb-into-bidual}(a)) such that $\iota_{S^*} = \varphi^*$. 

	Consider $u\in B_{\tilde S}^+ = \Morc_w(\tilde S^*;\BC)$.
	Let $u\check{\ }: (\tilde S^*)\check{\ }\to \BC$ be the function given by Proposition \ref{prop:reg}(c), and let $\delta_u^1\in \ell^1(B_{\tilde S}^+,\BC)$ be the point-mass at $u$. 
	Then $\delta_u^1$ belongs to $N(\tilde S^* )_*$ (see \eqref{eqt:pre-dual}) and satisfies $u = \delta_u^1\circ \mu_{\tilde S^*}$, which implies $u\check{\ } = \delta_u^1|_{(\tilde S^*)\check{\ }}$. 
	Moreover, by \eqref{eqt:Y-1} and Theorem \ref{thm:reg-dual}(b), one has $u\check{\ }\in (\tilde S_1)_+$, and it is not hard to check that 
	$$\varphi(u\check{\ }) = \iota_{\tilde{S}^*}^*(u\check{\ }) = \mu_{\tilde S^*}^*(\delta_u^1) = u.$$ 
	From this and  Statement (C3), we know that $\varphi$ is surjective and hence is an open map. 
	Consequently, $\varphi^*$ is bounded below. 
	
	\smnoind
	(C2) $\Leftrightarrow$ (C3) $\Leftrightarrow$ (C4). 
	These follow directly from Proposition \ref{prop:dual-quasi-func-sys}.

	Finally, in order to verify the last statement, we note that if the weak$^*$-continuous map $\iota_{S^*}$ is bounded below, then by Lemma \ref{lem:image-weak-st-cont-bdd-below}, its range is weak$^*$-closed (which implies $\iota_{S^*}(S^*) =  S^\rd$) and $\iota_{S^*}: S^*\to S^\rd$ is a weak$^*$-homeomorphism. 
	Thus, it follows from Condition (C2) and Lemma \ref{lem:Ng-MOS}(a) that $\iota_{S^*}$ is a complete order isomorphism. 
	Proposition \ref{prop:mu-S-st}(a) then implies that $S^\rd$ has a generating cone.
\end{proof}

\begin{prop}\label{prop:iota-comp-embed}
For an operator system $S$, the following are equivalent: 
	\begin{enumerate}[label=\ \ D\arabic*).]
		\item $\iota_{S^*}$ is a complete embedding;	
		
		\item $\tilde S$ is dualizable;	
		
		\item there is a regular norm on $M_\infty(\tilde S)_\sa$ (see Definition \ref{defn:Riesz}) equivalent to the original norm on $M_\infty(\tilde S)_\sa$;		
		
		\item the norm-closure of $B_{M_\infty(S)}^+ - B_{M_\infty(S)}^+$ is a norm zero neighborhood of $M_\infty(S)_\sa$. 
	\end{enumerate}
	In this case, $S^\rd$ is dualizable.  
\end{prop}
\begin{proof}
As in \eqref{eqt:iota-ti-S-*}, we identify $\iota_{S^*} =  \iota_{\tilde S^*}$. 

\smnoind
(D1) $\Leftrightarrow$ (D2). 
	This follows from Definition \ref{defn:dual}. 
	
	\smnoind
	(D2) $\Leftrightarrow$ (D3). 
	By Definition \ref{defn:dual}, Statement (D2) is equivalent to $M_\infty(\tilde S)_\sa$ being locally decomposable, in the sense of Definition \ref{defn:loc-decomp}(c). 
	Moreover, since $\tilde S$ is an operator system (see Remark \ref{rem:comp-OS}), the norm on $M_\infty(\tilde S)$ is absolutely monotone (see Lemma \ref{lem:matrix-reg}(c)). 
	Thus, Proposition \ref{prop:rel-bdd-decomp-prop} gives the required equivalence. 
	
	\smnoind
	(D2) $\Leftrightarrow$ (D4). 
	Let us denote $\overline{B_{M_\infty(S)}^+ - B_{M_\infty(S)}^+}$ to be the norm closure of $B_{M_\infty(S)}^+ - B_{M_\infty(S)}^+$. 
	By Definition \ref{defn:dual}, Statement (D2) is equivalent to $\overline{B_{M_\infty(\tilde S)}^+ - B_{M_\infty(\tilde S)}^+}$ being a norm zero neighborhood of $M_\infty(\tilde S)_\sa$. 
	Since 
	$$\overline{B_{M_\infty(\tilde S)}^+ - B_{M_\infty(\tilde S)}^+}\cap M_\infty(S)_\sa \subseteq \overline{B_{M_\infty(S)}^+ - B_{M_\infty(S)}^+}$$ (because $M_\infty(\tilde S)_+$ is the norm closure of $M_\infty(S)_+$), we have (D2) $\Rightarrow$ (D4). 
	Conversely, if Statement (D4) holds, then since $\overline{B_{M_\infty(S)}^+ - B_{M_\infty(S)}^+}$ is contained in $\overline{B_{M_\infty(\tilde S)}^+ - B_{M_\infty(\tilde S)}^+}$, we know that Statement (D2) holds (observe that the closure of a norm neighborhood of $M_n(S)_\sa$ in its completion $M_n(\tilde S)_\sa$ is a norm neighborhood of $M_n(\tilde S)_\sa$).
	
	Finally, assume that  $\iota_{S^*}$ is a complete embedding. 
	Then $\iota_{S^*}$ is bounded below, and it follows from Proposition \ref{prop:iota-bdd-below} that $\iota_{S^*}^{(\infty)}$ is an order isomorphism from $M_\infty(S^*)$ onto $M_\infty(S^\rd)$. 
	Moreover, since $\iota_{S^*}^{(\infty)}$ is also a normed space isomorphism, Proposition \ref{prop:mu-S-st}(a) implies that $B_{M_\infty(S^\rd)}^+ - B_{M_\infty(S^\rd)}^+$ is a norm zero neighborhood of $M_\infty(S^\rd)_\sa$; i.e., $S^\rd$ is dualizable.
\end{proof}

%%\medskipI

%%\medskipI

\begin{cor}\label{cor:dual-funct}
	(a) The restriction of the dual functor $(S,T,\varphi)\mapsto (T^\rd, S^\rd, \varphi^\rd)$ to the category of complete operator systems having generating cones (which include all dualizable operator systems)	 is faithful.
	
	\smnoind
	(b) If a complete operator system $S$ has a generating cone, then so is $S^\rd$. 
	
	\smnoind
	(c) If a complete operator system $S$ is dualizable, then $S^\rd$ is dualizable. 
\end{cor}
\begin{proof}
	(a) Let $S$ and $T$ be complete operator systems having generating cones.
	By Proposition \ref{prop:iota-bdd-below}, the map $\iota_{S^*}:S^*\to S^\rd$ is bijective and so is $\iota_{T^*}$. 
	Suppose that $\phi, \psi\in \Morc(S,T)$ satisfying $\phi^\rd = \psi^\rd$.
	It then follows from Relation \eqref{eqt:phi-d} that $\phi^* = \psi^*$ and hence $\phi =\psi$.

	\smnoind
	(b) This part follows from Proposition \ref{prop:iota-bdd-below}.

	\smnoind
	(c) This part follows from Proposition \ref{prop:iota-comp-embed}. 
\end{proof}

%%\medskipI

We will see in the next section that the dual functor is also injective on objects and full, if we further restrict the functor to complete operator systems that admits a  ``good'' increasing net.

%%\medskipI

Suppose that $T$ is an operator system with $\tilde T_\sa = \tilde T_+ - \tilde T_+$.
It follows from Proposition \ref{prop:iota-bdd-below} that $\iota_{T^*}: T^*\to T^\rd$ is a Banach space isomorphism, and $T^\rd_\sa = T^\rd_+ - T^\rd_+$. 
Hence, $\iota_{(T^\rd)^*}:  (T^\rd)^* \to (T^\rd)^\rd$ 
is again a Banach space isomorphism (by Proposition \ref{prop:iota-bdd-below}). 
In Relation (3.10) of \cite{Ng-dual-OS}, the following bijection was considered:   
\begin{equation}\label{eqt:def-tau}
	\tau_T:=\iota_{(T^\rd)^*}\circ (\iota_{T^*}^*)^{-1}: T^{**}\to (T^\rd)^\rd, 
\end{equation}
and it was asked when this map will be a complete isometry. 

%%\medskipI

Let us look at the map $\tau_T$ from another angle. 
Suppose that $S$ is an operator system (but $\tilde S$ is not assumed to have a generating cone). 
Then $S^{**}$ is a dual operator system (see the proof of Proposition \ref{prop:mu-S-st}(a)). 
Hence, by Theorem \ref{thm:prop-os-dual}(a), the map $\iota_{S^*}^*\in \Morc_w((S^\rd)^*,S^{**})$ induces 
a map $\overline{\iota_{S^*}^*}\in \Morc_w((S^\rd)^\rd, S^{**})$ satisfying 
\begin{equation}\label{eqt:def-bar-iota-S*-*}
		\xymatrix{
	& (S^\rd)^* \ar[d]_{\iota_{(S^\rd)^*}} \ar[r]^{	\iota_{S^*}^*} & S^{**}\\
	& (S^\rd)^\rd \ar[ur]_{\overline{\iota_{S^*}^*}} 
}
\end{equation}
(see Theorem \ref{thm:prop-os-dual}(a)).
From this, we have the following. 

%%\medskipI

\begin{lem}\label{lem:inv-tau-S}
	Suppose that $S$ is an operator system with its completion $\tilde S$ having a generating cone. 
	Then $\tau_S^{-1} = \overline{\iota_{S^*}^*}$. 
	In this case, $\tau_S$ is a complete isometry if and only if $\tau_S$ is a complete contraction. 
\end{lem}

%%\medskipI

The good points about $\overline{\iota_{S^*}^*}$ are that its definition is more canonical, and that it can be defined for all operator systems.
The natural questions of when $\overline{\iota_{S^*}^*}$ is surjective, is injective, is weak$^*$-homeomorphic, is completely order monomorphic and is a complete isometry could also be asked for general operator systems.  

%%\medskipI

\begin{thm}\label{thm:comp-isom-tau}
	Let $S$ be an operator system such that $\tilde S$ have a generating cone. 
	
	\smnoind
	(a) $\overline{\iota_{S^*}^*}:(S^\rd)^\rd \to S^{**}$ is a weak$^*$-homeomorphic complete order isomorphism. 
	
	\smnoind
	(b) One has $\big\|(\iota_{S^*}^*)^{(\infty)}(v)\big\| 
	= \|v\|$ ($v\in M_\infty((S^\rd)^*)_+$) and $\big\|\overline{\iota_{S^*}^*}^{(\infty)}(u)\big\| 
	= \|u\|$ ($u\in M_\infty((S^\rd)^\rd)_+$).

	\smnoind
	(c) The following statements equivalent. 
	\begin{enumerate}[label=\ \ \arabic*).]
		\item The map $\tau_S$ as in \eqref{eqt:def-tau} is a complete isometry. 
		
		\item $\overline{\iota_{S^*}^*}:(S^\rd)^\rd \to S^{**}$ is a complete isometry. 
		
		\item $S^{**} \cong T^\rd$ completely isometrically as dual operator systems,  for a complete operator system $T$. 
		
		\item There is a complete contraction $\Phi: S\to (S^\rd)^\rd$ with $\kappa_S = \overline{\iota_{S^*}^*}\circ \Phi$. 
	\end{enumerate}
	
	\smnoind
	(d) If $S^{**}$ is matrix regular (in particular, if $S$ is matrix regular), then $\overline{\iota_{S^*}^*}:(S^\rd)^\rd \to S^{**}$ is a completely isometric dual operator system isomorphism.  
\end{thm}
\begin{proof}
	Note that the statements in parts (a), (b) and (d) as well as Statements (1) - (3) in part (c) remain unchanged when $S$ is replaced by $\tilde S$, and a complete contraction $\Phi$ as in Statement (4) will induce a complete contraction satisfying the corresponding statement for $\tilde S$. 
	Thus, we may assume that $S$ is complete. 
	
	Proposition \ref{prop:iota-bdd-below} implies that $\iota_{S^*}:S^*\to S^\rd$ is a weak$^*$-homeomorphic complete order isomorphism and $S^\rd_\sa = S^\rd_+ - S^\rd_+$. 
	Thus, the complete order isomorphism 
	$\iota_{S^*}^*: (S^\rd)^*\to S^{**}$ 
	is a weak$^*$-homeomorphism, and 
	$$\iota_{(S^d)^*}: (S^\rd)^* \to (S^\rd)^\rd$$ 
	is a weak$^*$-homeomorphic complete order isomorphism
	(again, by Proposition \ref{prop:iota-bdd-below}). 
	Therefore, 
	$\overline{\iota_{S^*}^*}: (S^\rd)^\rd \to S^{**}$ 
	is a weak$^*$-homeomorphism (recall from Relation \eqref{eqt:def-bar-iota-S*-*} that $\iota_{S^*}^*=\overline{\iota_{S^*}^*}\circ\iota_{(S^\rd)^*})$. 
	
	\smnoind
	(a) Since the complete order isomorphism $\iota_{S^*}:S^* \to S^\rd$ is norm homeomorphic (by the open mapping theorem), $(\iota_{S^*})^{-1}:S^\rd \to S^*$ is a bounded complete order isomorphism, and hence so is $((\iota_{S^*})^{-1})^*$.  
	From this, we see that $\iota_{S^*}^*$ is a complete order isomorphism. 
	Consequently, 
	$\overline{\iota_{S^*}^*}$ is also a complete order isomorphism (see Relation \eqref{eqt:def-bar-iota-S*-*}).
	
	\smnoind
	(b) Note that as $\overline{\iota_{S^*}^*}$ is a complete contraction, the second equality follows from the first one, Relation \eqref{eqt:def-bar-iota-S*-*}, 
	as well as the fact that $\iota_{(S^\rd)^*}$ is a completely contrative complete order isomorphism. 
	
	In order to show the first equality, consider $n\in \BN$ and $v\in M_n((S^\rd)^*)_+$ with $\big\|(\iota_{S^*}^*)^{(n)}(v)\big\| \leq 1$. 
	Pick any $m\in\BN$ and $f\in B_{M_m(S^\rd)}$. 
	Since $\iota_{S^*}(S^*)  =S^\rd$ (by Proposition \ref{prop:iota-bdd-below}), there is $\omega\in M_m(S^*)$ with $f = \iota_{S^*}^{(m)}(\omega)$. 
	As in Proposition \ref{prop:reg}(a), we identify $\iota_{S^*}$ with $j_{S^*}$, and obtain
	$$\sup \big\{\|x^{(m)}(\omega)\| : x \in B_{M_\infty(S^{**})}^+\big\} = \|\iota_{S^*}^{(m)}(\omega)\| = \|f\|\leq 1.$$
	Since $(\iota_{S^*}^*)^{(n)}(v)\in B_{M_n(S^{**})}^+$, the above displayed inequality implies that 
	$$\|f^{(n)}(v) \| = \big\|\big(\iota_{S^*}^{(m)}(\omega)\big)^{(n)}(v) \big\| = \big\|\big((\iota_{S^*}^*)^{(n)}(v)\big)^{(m)}(\omega)\big\|\leq 1.$$
	This means that $\|v\|\leq 1$.
	From this, and the complete contractivity of $\iota_{S^*}^*$, we see that  $\|\iota_{S^*}^*(v)\| = \|v\|$.

	\smnoind
	(c) (1)$\Leftrightarrow$(2). 
	This follows from Lemma \ref{lem:inv-tau-S}. 
	
	\noindent
	(2)$\Rightarrow$(3). 
	It follows from part (a) above that if Statement (2) holds, then $S^{**}\cong (S^\rd)^\rd$ completely isometrically as dual operator systems. 
	
	\noindent
	(3)$\Rightarrow$(2). 
	We identify the two dual operator systems $S^{**}$ and $T^\rd$ directly.  
	Set  
	$$\varphi:= \iota_{T^*}^*|_{S^*}\in \Morc(S^*,T),$$
	(observe that $\iota_{T^*}\in \Morc_w(T^*,S^{**})$). 
	Consider $\varphi\check{\ }\in \Morc\big(S^\rd,T\big)$ to be the map as in Proposition \ref{prop:reg}(c) (note that $(S^*)\check{\ } = S^\rd$, because of Proposition \ref{prop:iota-bdd-below}). 
	Then $\varphi\check{\ }\circ \iota_{S^*} = \varphi$, and $(\varphi\check{\ })^\rd\in \Morc_w\big(T^\rd, (S^\rd)^\rd \big)$  satisfies (see \eqref{eqt:phi-d})
	$$(\varphi\check{\ })^\rd\circ \iota_{T^*} = \iota_{(S^\rd)^*}\circ (\varphi\check{\ })^*.$$	
	
	We claim that $\overline{\iota_{S^*}^*}\circ (\varphi\check{\ })^\rd: T^\rd \to S^{**}$ is the identity map $\id_{S^{**}}$ on ${S^{**}}=T^\rd$.
	In fact, pick any $f\in T^*$ and $\omega\in S^*$. 
	By \eqref{eqt:def-bar-iota-S*-*}, one has 
	\begin{align*}
		\big(\overline{\iota_{S^*}^*}\big((\varphi\check{\ })^\rd\big(\iota_{T^*}(f)\big)\big)\big)(\omega)
		&  = \big(\overline{\iota_{S^*}^*}\big(\iota_{(S^\rd)^*}\big((\varphi\check{\ })^*(f)\big)\big)\big)(\omega)
		\ = \ \big(\iota_{S^*}^*\big(f\circ \varphi\check{\ }\big)\big)(\omega)\\
		& = (f\circ \varphi\check{\ })(\iota_{S^*}(\omega))
		\ = \ f(\varphi(\omega))\\
		& =  f\big(\iota_{T^*}^*(\omega)\big) 
		\ = \ \iota_{T^*}(f)(\omega). 
	\end{align*}
	This means  that $\overline{\iota_{S^*}^*}\circ (\varphi\check{\ })^\rd\circ \iota_{T^*} = \iota_{T^*}$. 
	As $\iota_{T^*}(T^*)$ is weak$^*$-dense in $T^\rd$ and $\overline{\iota_{S^*}^*}\circ (\varphi\check{\ })^\rd$ is weak$^*$-continuous, our claim is verified. 
	
	Now, Statement (2) follows from the above claim as well as the fact that the complete contraction $\overline{\iota_{S^*}^*}$ is bijective (see part (a)). 
	
	\noindent
	(2)$\Rightarrow$(4). 
	One may simply take $\Phi := \overline{\iota_{S^*}^*}^{-1}\circ \kappa_S$ (see part (a) above). 
	
	\noindent
	(4)$\Rightarrow$(2). 
	Consider $n \in \BN$ and $x\in M_n((S^\rd)^\rd)$ with $\|\overline{\iota_{S^*}^*}^{(n)}(x)\| \leq 1$. 
	By Corollary \ref{cor:dense-ball}, one can find a net $\{y_i\}_{i\in \KI}$ in $B_{M_n(S)}$ such that $\{\kappa_S^{(n)}(y_i)\}_{i\in \KI}$ weak$^*$-converges to $\overline{\iota_{S^*}^*}^{(n)}(x)$. 
	As $\Phi^{(n)}$ is contractive, one can find a weak$^*$-convergent subnet of $\{\Phi^{(n)}(y_i)\}_{i\in \KI}$ in $B_{M_n((S^\rd)^\rd)}$ (note that $(S^\rd)^\rd$ is a dual operator system). 
	Without loss of generality, we may assume that $\{\Phi^{(n)}(y_i)\}_{i\in \KI}$ weak$^*$-converges to some $z\in B_{M_n((S^\rd)^\rd)}$. 
	Since $\overline{\iota_{S^*}^*}$ is a weak$^*$-homeomorphism and $\kappa_S = \overline{\iota_{S^*}^*}\circ \Phi$, we see that $\{\kappa_S^{(n)}(y_i)\}_{i\in \KI}$ will weak$^*$-converge to $\overline{\iota_{S^*}^*}^{(n)}(z)$. 
	Now, the injectivity of $\overline{\iota_{S^*}^*}$ (see part (a)) tells us that $x = z\in B_{M_n((S^\rd)^\rd)}$. 
	Consequently, $\overline{\iota_{S^*}^*}^{(n)}$ is an isometry. 
	
	\smnoind
	(d) Notice that by Corollary \ref{cor:bidual-abs-mono}(b), if $S$ is matrix regular, then so is $S^{**}$. 
Thus, the statement concerning the matrix regularity of $S$ follows from the one concerning the matrix regularity of $S^{**}$. 

	As in  part (b) above, we have  
	$\big\|\overline{\iota_{S^*}^*}^{(\infty)}(w)\big\| = \|w\|$ ($w\in M_\infty((S^\rd)^\rd)_+$).
	Hence, if $S^{**}$ is matrix regular, then  
	 part (a) above as well as Lemma \ref{lem:matrix-reg}(c) and Lemma \ref{lem:abs-cont-to-mat-reg} imply that $\overline{\iota_{S^*}^*}$ is a complete isometry. 
	This, together with part (a), gives the required conclusion. 
\end{proof}

%%\medskipI

Part (d) above, together with Lemma \ref{lem:matrix-reg}(d), gives \cite[Theorem 3.17(b)]{Ng-dual-OS}, which states (via Lemma \ref{lem:inv-tau-S}) that $\overline{\iota_{S^*}^*}$ is a complete isometry when $S$ is unital or when $S$ is a $C^*$-algebra. 
It was conjectured in the paragraph following \cite[Theorem 3.17]{Ng-dual-OS} that the same is true when $S$ is ``approximately unital'' in a suitable sense. 
In the next section, we will give an explicit verification for it (again, via part (d) above). 
Another motivation for the next section comes from the work of Connes and van Suijlekom concerning tolerance relations.

%%\medskipI

\section{Approximately unital operator systems}

%%\medskipI

In \cite{CvS} and \cite{CvS2}, Connes and van Suijlekom associated an (not necessarily complete) operator system to every tolerance relation. 
In their study in \cite{CvS2}, they used the fact that such an operator system admits an approximate order unit that defines the matrix norm. 
The aim of this section is to consider duality of such operator systems. 
In fact, we will consider a slightly more general situation when an operator system contains a weakly matrix norm defining increasing net (see Theorem \ref{thm:unital-bidual}(c)). 
One good feature of this weaker alternative is that it is stable under completion (see Theorem \ref{thm:unital-bidual}(a)). 
On our way, we show that the bidual of a complete SMOS is a unital operator system if and only if this SMOS has a weakly matrix norm defining increasing net. 

%%\medskipI

Let $X$ be a SMOS.
If $a=\{a_i\}_{i\in \KI}$  is an increasing net in $X_+$, then we set, for each $n\in \BN$ and $x\in M_n(X)$,  
\begin{equation}\label{eqt:def-norm-e}
	\|x\|_a:= \inf \left\{t \in \RP: \begin{pmatrix}
		t (I_n\otimes a_i)&x\\
		x^*& t (I_n\otimes a_i)
	\end{pmatrix}\in M_{2n}(X)_+, \exists i\in \KI \right\},
\end{equation}
and 
\begin{align*}
	\|x\|_a^w := \inf \Bigg\{t \in \RP: \lim_i\, & F\begin{pmatrix}
		t (I_n\otimes a_i)&x\\
		x^*& t (I_n\otimes a_i)
	\end{pmatrix}\in \RP\cup \{+\infty\},\\
& \qquad \qquad \qquad \qquad \quad \text{ for every }F\in M_{2n}(X)^*_+\Bigg\}.
\end{align*}
Here, we use the convention that $\inf \emptyset := +\infty$.  
Notice that as the net $$\bigg\{F\begin{pmatrix}
	t (I_n\otimes a_i)&x\\
	x^*& t (I_n\otimes a_i)
\end{pmatrix}\bigg\}_{i\in\KI}$$ is increasing, its limit always exists in $(-\infty,+\infty]$. 

%%\medskipI

As in the above, we regard $x\in M_n(X)$ as an element in $M_{n+k}(X)$ when $k\in\BN$.
By rearranging the rows and the columns, we identify 
$\begin{pmatrix}
	t (I_{n+k}\otimes a_i)&x\\
	x^*& t (I_{n+k}\otimes a_i)
\end{pmatrix}$
with 
$\begin{pmatrix}
	t (I_{n}\otimes a_i)&x&0\\
	x^*& t (I_{n}\otimes a_i)&0\\
	0&0&t (I_{2k}\otimes a_i)
\end{pmatrix}$. 
This tells us  that the value $\|x\|_a$ is unchanged whether $x$ is regarded as an element in $M_n(X)$ or an element in $M_{n+k}(X)$.

On the other hand, using Theorem \ref{thm:bdd-pos-CCP}(c), we also see that the quantity $\|x\|_a^w$ is unchanged when we regard $x\in M_n(X)$ as an element in $M_{n+k}(X)$.

Moreover, if $\{a_i\}_{i\in \KI}$  is an approximate order unit of $X_\sa$  (see Definition \ref{defn:loc-decomp}(a)), then both $\|x\|_a$ and $\|x\|_a^w$ are finite (see the proof of \cite[Lemma 2.9]{CvS2}).

%%\medskipI

\begin{defn}\label{defn:weakly-matrix}
	Let $X$ be a SMOS with its matrix norm $\|\cdot\|$.  
	Let $a=\{a_i\}_{i\in \KI}$ be an increasing net in $X_+$. 
	
	\smnoind
	(a) $\{a_i\}_{i\in \KI}$ is said to be \emph{matrix norm defining}, or said to \emph{define the matrix norm} of $X$,  if $\|x\| = \|x\|_a$  for every  $x\in M_\infty(X)$. 
	In the case when $a=\{a_i\}_{i\in \KI}$ is the constant net, then we call it a \emph{matrix norm defining order unit}. 
	
	\smnoind
(b) $\{a_i\}_{i\in \KI}$ is said to be \emph{weakly matrix norm defining}, or said to \emph{weakly defines the matrix norm} of $X$, if $\|x \| = \|x\|_a^w$ for every  $x\in M_\infty(X)$.

\smnoind
(c) $X$ is said to be \emph{approximately unital} if $X$ admits a weakly matrix norm defining increasing net. 
\end{defn}

%%\medskipI

By \cite[Proposition 1.3.5]{OPS}, if an increasing net is matrix norm defining, then it is an approximate order unit of $X_\sa$. However, we will see in Example \ref{eg:weak-norm-def} below, an explicit example of a weakly matrix norm defining increasing net that is not an approximate order unit (and hence cannot be matrix norm defining).

Nevertheless, if we have a weakly matrix norm defining increasing net, there exists an approximate unit in the cone of the completion $\tilde X_\sa$ (but may not be the original increasing net) that weakly defines the matrix norm of $\tilde X$ (see Theorem \ref{thm:unital-bidual}(a) below). In order to
establish this, we need two lemmas. 

%%\medskipI

\begin{lem}\label{lem:approx-unit}
	Let $X$ be a SMOS and $a=\{a_i\}_{i\in \KI}$ be an increasing net in $X_+$. 
	
	\smnoind
	(a) $\|a_i\|_a\leq 1$ ($i\in \KI$). 
	
	\smnoind
	(b) $\|x\|_a^w \leq \|x\|_a$ for each $x\in M_\infty(X)$.

	\smnoind
	(c) For $u,v\in M_\infty(X)_\sa$ with $-v \leq u\leq v$, one has $\|u\|_a\leq \|v\|_a$. 
	
	\smnoind
	(d) If $\|\cdot\|_a$ is a norm, then $M_\infty(X)_+\cap - M_\infty(X)_+ = \{0\}$. 
	
	\smnoind
	(e) If $\{a_i\}_{i\in \KI}$ is matrix norm defining (respectively, weakly matrix norm defining), then any subnet of $\{a_i\}_{i\in \KI}$ is matrix norm defining (respectively, weakly matrix norm defining). 
\end{lem}
\begin{proof}
	For $n\in \BN$, $x\in M_n(X)$,  $s\in \RP$ and  
	$i\in \KI$, we denote $$x^{s,i}:= \begin{pmatrix}
		s (I_n\otimes a_{i})&x\\
		x^*& s (I_n\otimes a_{i})
	\end{pmatrix}.$$ 
	
	\smnoind
	(a) This part follows from the fact that $\begin{pmatrix}
		a_i& a_i\\
		a_i& a_i
	\end{pmatrix} = \begin{pmatrix}
		1\\
		1
	\end{pmatrix} a_i (1, 1)\in M_2(X)_+$.

	\smnoind
	(b) If $x^{s,i_0}\in M_{2n}(X)_+$ for some $i_0\in \KI$, then $F(x^{s,i_0})\in \RP$ ($F\in M_{2n}(X)^*_+$). 
	Since $\{x^{s,i}\}_{i\in \KI}$ is an increasing net  in $M_{2n}(X)_\sa$, we know that $\|x\|_a^w\leq \|x\|_a$. 
	
	\smnoind
	(c) This follows from the fact that for every $w\in M_n(X)_\sa$ and $t\in \RP$, one has  $w^{t,i}\in M_{2n}(X)_+$ if and only if $- t (I_n\otimes a_i) \leq w \leq t (I_n\otimes a_i)$ (see \cite[Proposition 1.3.5]{OPS}). 
	
	\smnoind
	(d) Suppose that $u\in M_\infty(X)_+\cap - M_\infty(X)_+$. 
	Then $0\leq u\leq 0$ and we know from part (c) that $\|u\|_a \leq \|0\|_a = 0$, which forces $u$ to be zero. 
	
	\smnoind
	(e) Suppose that $\{a_{i_j}\}_{j\in \KJ}$ is a subnet of $\{a_i\}_{i\in \KI}$ and $x\in M_n(X)$. 
	Set 
	$$\hat{x}^{t,j}:=\begin{pmatrix}
		t (I_n\otimes a_{i_j})&x\\
		x^*& t (I_n\otimes a_{i_j})
	\end{pmatrix}\qquad (j\in \KJ).$$ 
	The statement for the case when $\{a_{i}\}_{i\in \KI}$ is matrix norm defining follows from the fact that for each $t\in \RP$, the condition that $x^{t,i}\in M_{2n}(X)_+$ for some $i\in \KI$ is the same as  $\hat{x}^{t,j}\in M_{2n}(X)_+$ for some $j\in \KJ$. 
	
	Let now consider the case when $\{a_{i}\}_{i\in \KI}$ is weakly matrix norm defining. 
	For each $F\in M_{2n}(X)^*_+$, as
	$\big\{F(\hat{x}^{t,j})\big\}_{j\in \KJ}$
	is a subnet of the increasing net 
	$\big\{F(x^{t,i})\big\}_{i\in \KI}$, 
	they converge to the same limit. 
	Hence, 
	\begin{align*}
		\|x\| & = \inf \left\{t \in \RP: {\lim}_i  F(x^{t,i})\in \RP\cup \{+\infty\}, \text{ for every }F\in M_{2n}(X)^*_+\right\}\\
		& = \inf \big\{t \in \RP: {\lim}_j  F(\hat{x}^{t,j})\in \RP\cup \{+\infty\}, \text{ for every }F\in M_{2n}(X)^*_+\big\}, 
	\end{align*}
	as required. 
\end{proof}

%%\medskipI

Theorem \ref{thm:bdd-pos-CCP} allow us to extend $\|\cdot\|_a^w$ to a positive function $\|\cdot\|_a^{w^*}$ on $M_\infty(X^{**})$. 
More precisely, for $n\in \BN$ and $y\in M_n(X^{**})$,  we set 
\begin{eqnarray*}\|y\|_a^{w^*} :=\inf \Bigg\{t \in \RP: \lim_i  \left(\begin{pmatrix}
	t (I_n\otimes a_i)& y\\
	y^*& t (I_n\otimes a_i)
\end{pmatrix}, \Theta^X_F\right)\in \RP\cup \{+\infty\},\\
 \text{ for every }F\in M_{2n}(X)^*_+\Bigg\},
\end{eqnarray*}
where $\left(\begin{pmatrix}
	t (I_n\otimes a_i)& y\\
	y^*& t (I_n\otimes a_i)
\end{pmatrix}, \Theta^X_F\right)$ is the pairing as in  \eqref{eqt:def-inf-pair} and $\Theta^X_F$ is as in Theorem \ref{thm:bdd-pos-CCP}. 
By \eqref{eqt:F-Theta-X-F}, we have 
$$\|x\|_a^{w^*} = \|x\|_a^{w}\qquad (x\in M_\infty(X)).$$

%%\medskipI

\begin{lem}\label{lem:approx-unit-SMOS}
	Let $X$ be a SMOS, and $a=\{a_i\}_{i\in \KI}$ be an increasing net in $X_+$ that weak$^*$-converges to $p\in X^{**}_+$.
	Denote by $\|\cdot\|_p$ the function on $M_\infty(X^{**})$ as in \eqref{eqt:def-norm-e} for the net $\{p\}$, and by  $\|\cdot\|$ the bidual matrix norm on $X^{**}$. 
	
	\smnoind
	(a) $\|\cdot\|_a^{w^*} = \|\cdot\|_p$ on $M_\infty(X^{**})$.
	
	\smnoind
	(b) If $\|x\|_p \leq \|x\|$ ($x\in M_\infty(X)$), then $\|y \|_p \leq \|y\|$ ($y\in M_\infty(X^{**})$). 
	
	\smnoind
	(c) If $\|p\| \leq 1$ and the norm on $M_\infty(X^{**})_\sa$ is absolutely monotone (see Definition \ref{defn:Riesz}), then $\|y\| \leq \|y\|_p$ for any $y\in M_\infty(X^{**})$. 
\end{lem}
\begin{proof}
	Fix $n\in \BN$, $y\in M_n(X^{**})$ and $s\in \RP$. 
	We set $$y^{s,i}:= \begin{pmatrix}
		s (I_n\otimes a_i)& y\\
		y^*& s (I_n\otimes a_i)
	\end{pmatrix}\in M_{2n}(X^{**})_\sa\quad
	(i\in \KI).$$
	
	\smnoind
	(a) If $F\in M_{2n}(X)^*_+$, then \eqref{eqt:def-inf-pair} produces
	\begin{equation}\label{eqt:rel-weak-strong}
		{\lim}_i  (y^{s,i}, \Theta^X_F) 
		= \left(\begin{pmatrix}
			s (I_n\otimes p)& y\\
			y^*& s (I_n\otimes p) 
		\end{pmatrix}, \Theta^X_F \right).
	\end{equation}
	Therefore, if 
	\begin{equation}\label{eqt:x-e-gep-0}
		\begin{pmatrix}
			s (I_n\otimes p)& y\\
			y^* & s (I_n\otimes p)
		\end{pmatrix}\in M_{2n}(X^{**})_+, 
	\end{equation}
	then ${\lim}_i  (y^{s,i}, \Theta^X_F)\in \RP$ for every $F\in M_{2n}(X)^*_+$ (as $\Theta^X_F \geq 0$, by Theorem \ref{thm:bdd-pos-CCP}(a)).
	Conversely, if ${\lim}_i  (y^{s,i}, \Theta^X_F)\in \RP\cup \{+\infty\}$ for each $F\in M_{2n}(X)^*_+$, then Relation \eqref{eqt:x-e-gep-0} holds, 
	because of Relation \eqref{eqt:rel-weak-strong} as well as Theorem \ref{thm:bdd-pos-CCP}(a). 
	These show that $\|y\|_p = \|y\|_a^{w^*}$.
	
	\smnoind
	(b) Suppose that $n\in \BN$ and $y\in B_{M_n(X^{**})}$. 
	Using Corollary \ref{cor:dense-ball}, one can find a net $\{y_j\}_{j\in \KJ}$ in $M_n(X)$ with $\|y_j\| < 1$ ($j\in \KJ$) that $\sigma(M_n(X^{**}), M_n(X^*))$-converges to $y$. 
	For $j\in \KJ$, we learn from $\|y_j\|_p = \|y_j\| < 1$ that  
	$$\begin{pmatrix}
		I_n\otimes p& y_j\\
		y_j^*& I_n\otimes p
	\end{pmatrix}\in M_{2n}(X^{**})_+.$$
	The $\sigma(M_{2n}(X^{**}), M_{2n}(X^*))$-closedness of $M_{2n}(X^{**})_+$  implies that $\|y\|_p \leq 1$. 
	
	\smnoind
	(c) Consider $y\in M_\infty(X^{**})$ with $\|y\|_p < s$. 
	Then Relation \eqref{eqt:x-e-gep-0} holds, which gives 
	$\begin{pmatrix}
		s (I_n\otimes p)&- y\\
		-y^*& s (I_n\otimes p)
	\end{pmatrix}\in M_{2n}(X^{**})_+$. 
	Hence, we conclude from 
	\begin{align*}
		\begin{pmatrix}
			-s (I_n\otimes p)& 0 \\
			0 & -s (I_n\otimes p)
		\end{pmatrix}
		& \leq \begin{pmatrix}
			0 & y\\
			y^*& 0
		\end{pmatrix}
		\leq \begin{pmatrix}
			s (I_n\otimes p)& 0 \\
			0 & s (I_n\otimes p)
		\end{pmatrix},
	\end{align*}
	and the absolutely monotone assumption that $\|y\| 
	\leq s$ (see Relation \eqref{eqt:norm-sa}). 
	This gives $\|y\|\leq \|y\|_p$.
\end{proof}

\begin{cor}\label{cor:C-st-alg-approx-unit}
Let $A$ be a $C^*$-algebra and $\{a_i\}_{i\in \KI}$ is an approximate unit of $A$. 
Then $\{a_i\}_{i\in \KI}$ is weakly matrix norm defining. 
\end{cor}
\begin{proof}
Since $\{a_i\}_{i\in\KI}$ weak$^*$-converges to the identity of $A^{**}$, we learn from  Lemma \ref{lem:approx-unit-SMOS}(a) that $\|\cdot\|_a^{w^*}$ is the matrix norm on the von Neumann algebra $A^{**}$. 
This gives the required conclusion. 
\end{proof}

\begin{eg}\label{eg:weak-norm-def}
	For each $n\in \BN$, we define $a_n \in \mathrm{c}_0(\BN)$ by 
	\[a_n(k):=\begin{cases}
		1 \qquad &k\leq n\\
		0 \qquad &k >n
	\end{cases} \qquad (k\in \BN).\]
	Then $\{a_n\}_{n\in\BN}$ is an approximate unit of $\mathrm{c}_0(\BN)$ and Corollary \ref{cor:C-st-alg-approx-unit} tells us that 
	 $\{a_n\}_{n\in\BN}$ is weakly matrix norm defining. 
However, it is clear that $\{a_n\}_{n\in\BN}$ is not an approximate order unit (e.g., consider the element $b\in c_0(\BN)$ with $b(k):=1/k$), and so, it cannot be 
matrix norm defining.
\end{eg}

\begin{rem}
Let $S$ be an operator system, and $a\in S_+$. 
If the constant net $\{a\}$ is matrix norm defining, then $I_n\otimes a$ is an Archimedean order unit for $M_n(S)$ ($n\in \BN$). 
Thus, $S$ is an unital operator system if and only if it admits a matrix norm defining order unit. 
\end{rem}
%%\medskipI

Recall that the matrix cone $M_\infty(\tilde X)_+$ of the complete $\tilde X$  of a SMOS $X$ is  the norm closure of $M_\infty(X)_+$ in $M_\infty(\tilde X)_\sa$.

\begin{thm}\label{thm:unital-bidual}
	Let $X$ be a SMOS and $\tilde X$ be its completion.

	\smnoind
	(a) Consider the following statements:
	\begin{enumerate}[label=\ \ U\arabic*).]
		\item $X$ is approximately unital. 
		
		\item $X^{**}$ is a unital operator system. 
		
		\item There exists an approximate order unit $\{u_i\}_{i\in \KI}$ of $\tilde X_\sa$, which is a subnet of $O_{\tilde X}^+:=\{u\in \tilde X_+: \|u\| < 1 \}$,
		that weakly defines the matrix norm on $\tilde X$ (hence, $\tilde X$ is approximately unital). 
	\end{enumerate}
	Then (U1) $\Rightarrow$ (U2) $\Rightarrow$ (U3). 
	Consequently, when $X$ is complete, these three statements are equivalent. 
	
	\smnoind
	(b) If $X$ is approximately unital, then $X$ is an operator system.  
	
	\smnoind
	(c) If $\{a_i\}_{i\in \KI}$ is a matrix norm defining approximate order unit, then $\{a_i\}_{i\in \KI}$ is weakly matrix norm defining. 
\end{thm}
\begin{proof}
	We denote by $\|\cdot\|$ the bidual matrix norm on $X^{**}$. 
	
	\smnoind
	(a) (U1)$\Rightarrow$(U2). 
	Suppose that $a=\{a_i\}_{i\in \BN}$ is a weakly matrix norm defining increasing net in $X_+$. 
	We learn from parts (a) and (b) of Lemma \ref{lem:approx-unit} that $\|a_i\| = \|a_i\|_a^w \leq 1$ ($i\in \KI$). 
	Hence, there is a subnet $\{a_{i_j}\}_{j\in \KI}$ of $\{a_i\}_{i\in \KI}$ that weak$^*$-converges to an element $p\in B_{X^{**}}^+$. 
	By replacing $\{a_i\}_{i\in \KI}$ with this subnet if necessary (see Lemma \ref{lem:approx-unit}(e)), we may assume that $\{a_i\}_{i\in \KI}$ weak$^*$-converges to $p$.
	It then follows from Lemma \ref{lem:approx-unit-SMOS}(a) and the weakly matrix norm defining assumption that 
	$$\|x\| = \|x\|_p \qquad (x\in M_\infty(X)).$$ 
	
	Therefore, Lemma \ref{lem:approx-unit-SMOS}(b) tells us that $\|y\|_p\leq \|y\|$ for each $y\in M_\infty(X^{**})$. 
	On the other hand, we learn from Lemma \ref{lem:approx-unit}(c) that $\|u\|_p\leq \|v\|_p$ for $u,v\in M_\infty(X)_\sa$ with $-v \leq u\leq v$. 
	Hence, the norm $\|\cdot\|$ on $M_\infty(X)_\sa$ is absolutely monotone, and Corollary \ref{cor:bidual-abs-mono}(a) tells us that the norm on $M_\infty(X^{**})_\sa$ is absolutely monotone. 
	Now, Lemma \ref{lem:approx-unit-SMOS}(c) gives $\|\cdot\| \leq  \|\cdot\|_p$ on $M_\infty(X^{**})$.
	
	The equality $\|\cdot\| = \|\cdot\|_p$ on $M_\infty(X^{**})$ ensures that $X^{**}$ is a unital operator system (note that $X^{**}$ is a MOS because of Lemma \ref{lem:approx-unit}(d)) with $I_n\otimes p$ being an Archimedean order unit of $M_n(X^{**})$ for every $n\in \BN$.

	\smnoind
	(U2)$\Rightarrow$(U3). 
	Put $Y:=\tilde X$. 
	Condition (U2) states that $Y^{**} = X^{**}$ is a unital operator system. 
	This gives $p\in Y^{**}_+$ satisfying $\|y\| = \|y\|_p$ ($y\in M_\infty(Y^{**})$).
	For $u\in Y^{**}_\sa$ and $t\in \RP$, it follows from \cite[Proposition 1.3.5]{OPS} that 
	\begin{equation}\label{eqt:order-unit}
		-tp\leq u\leq tp \quad \text{if and only if}\quad \begin{pmatrix}
			t p& u\\
			u & tp
		\end{pmatrix}\in M_{2}(Y^{**})_+.
	\end{equation}
	This, together with the equality $\|\cdot\| = \|\cdot\|_p$ on $Y^{**}_\sa$,  implies that the hypothesis of Proposition \ref{prop:conv-to-unit} is satisfied for the real ordered Banach space $Y_\sa$. 
	Consequently, there exists an approximate order unit $a=\{a_i\}_{i\in \KI}$, which is a subnet of $O_Y^+$ that weak$^*$-converges to $p$.
	By Lemma \ref{lem:approx-unit-SMOS}(a), we know that $\|\cdot\|_a^{w^*} = \|\cdot\|_p$ on $M_\infty(Y^{**})$. 
	Hence, we have $\|\cdot\|_a^w = \|\cdot\|$ on $M_\infty(Y)$, as required.
	
	\smnoind
	Finally, if $X$ is complete, then Statement (U3) implies Statement (U1).

	\smnoind
	(b) We know from Statement (U2) and Corollary \ref{cor:emb-into-bidual}(a) that $X$ is an operator system. 
	
	\smnoind
	(c) Lemma \ref{lem:approx-unit}(b) tells us that $\|x\|_a^w\leq \|x\|_a = \|x\|$ for  $x\in M_\infty(X)$. 
	Conversely, by \cite[Proposition 2.10]{CvS2}, one knows that $X$ is matrix regular. 
	In particular, the norm on $M_\infty(X)_\sa$ is absolutely monotone (see Lemma \ref{lem:matrix-reg}(b)), and so is the norm on $M_\infty(X^{**})_\sa$ (by Corollary \ref{cor:bidual-abs-mono}(a)). 
	Moreover, Lemma \ref{lem:approx-unit}(a) implies $\|a_i\| = \|a_i\|_a\leq 1$ ($i\in \KI$). 
	By passing to a subnet if necessary  (see Lemma \ref{lem:approx-unit}(e)), we may assume that $\{a_i\}_{i\in \KI}$ weak$^*$-converges to an element  $p\in B^+_{X^{**}}$. 
	Now, parts (a) and (c) of Lemma \ref{lem:approx-unit-SMOS} give $\|x\| \leq \|x\|_p = \|x\|_a^{w^*} = \|x\|_a^w$ ($x\in M_\infty(X)$). 
\end{proof}

%%\medskipI

In \cite{huang}, a description for those complete operator systems whose biduals being unital was given. 
Our description in part (a) above is more general, in the sense that we describe those complete SMOS whose biduals being unital operator systems, and our characterization is different from the one in \cite{huang}. 
On the other hand, no characterization of complete operator systems whose biduals being unital was found in either \cite{Karn} or \cite{K}.

%%\medskipI

Because of Theorem \ref{thm:unital-bidual}(c), the implication $(U1)\Rightarrow (U2)$ of Theorem \ref{thm:unital-bidual}(a) can be regarded as an extension of \cite[Proposition 2.10]{CvS2}. 

%%\medskipI

Since all unital operator systems and all $C^*$-algebras are approximately unital operator systems (see Corollary \ref{cor:C-st-alg-approx-unit}), our next theorem  generalizes both Theorem 3.17(b) and Corollary 3.18 of \cite{Ng-dual-OS}. 

%%\medskipI

\begin{prop}\label{prop:approx-unit}
	Let $S$ be an approximately unital operator system.
	
	\smnoind
	(a) The complete operator system $\tilde S$ is dualizable. 
	
	\smnoind
	(b) One has $\|\iota_{S^*}^{(\infty)}(f)\| \leq \|f\|_{M_\infty(S^*)} \leq 4 \|\iota_{S^*}^{(\infty)}(f)\|$ ($f\in M_\infty(S^*)$) and 
	$\|g\|_{M_\infty(S^*)} = \|\iota_{S^*}^{(\infty)}(g)\|$ ($g\in M_\infty(S^*)_+$).
	
	\smnoind
	(c) The map $\overline{\iota_{S^*}^*}: (S^\rd)^\rd \to S^{**}$ is a completely isometric dual operator system isomorphism. 
\end{prop}
\begin{proof}
	(a) By Theorem \ref{thm:unital-bidual}(a), $S^{**}$ is a unital operator system.
	Hence, \cite[Theorem 3.17]{Ng-dual-OS} implies that $S^{**}$ is dualizable; i.e., $S^{***}$ is a quasi-operator system (see Definition \ref{defn:dual}). 
	From this, and  Corollary \ref{cor:emb-into-bidual}(a), we see that $\tilde S^* = S^{*}$ is a quasi-operator system. 
	Thus, $\tilde S$ is dualizable. 
	
	\smnoind
	(b) Consider $h\in M_\infty(S^*)$. 
	It follows from Proposition \ref{prop:reg}(a) that one may identify $\iota_{S^*}$ with $j_{S^*}$. 
	Thus, Relations \eqref{eqt:def-norm-r}, \eqref{eqt:norm-mu-F} and \eqref{eqt:iota=mu} imply that 
	$$\big\|\iota_{S^*}^{(\infty)}(h)\big\| = \big\|j_{S^*}^{(\infty)}(h)\big\| = \|h\|^\rn = \big\|\mu_{S^{***}}^{(\infty)}(h)\big\|= \big\|\iota_{S^{***}}^{(\infty)}(h)\big\|.$$
	Moreover, by Corollary \ref{cor:emb-into-bidual}(a), one has $\|h\|_{M_\infty(S^*)} = \|h\|_{M_\infty(S^{***})}$. 
	Therefore, the conclusion follows from \cite[Theorem 3.17(a)]{Ng-dual-OS} (when applied to the unital operator system $S^{**}$). 
	
	\smnoind
	(c) This follows from Theorem \ref{thm:unital-bidual}(a) and Theorem \ref{thm:comp-isom-tau}(d) (see also Lemma \ref{lem:matrix-reg}(d)).
\end{proof}

%%\medskipI

\begin{thm}\label{thm:dual-funct-approx-unit}
	The restriction of the dual functor to the category of approximately unital complete operator systems,  is injective on objects, full and faithful. 
\end{thm}
\begin{proof}
	By Proposition \ref{prop:approx-unit}(a) and Corollary \ref{cor:dual-funct}(a), we know that the restriction of the dual functor to the category of approximately unital complete operator systems is faithful. 
	
	Consider $S$ and $T$ to be approximately unital complete operator systems and $\Phi\in \Morc_w(T^\rd, S^\rd)$. 
	We learn  from Proposition \ref{prop:approx-unit}(a)   that both $S^*$  and $T^*$ are quasi-operator systems. 
	Therefore, Theorem \ref{thm:reg-dual}(c) implies that $\iota_{S^*}:S^*\to S^\rd$ and $\iota_{T^*}:T^*\to T^\rd$ are weak$^*$-homeomorphisms.  
	These, together with the weak$^*$-continuity of $\Phi$, produce a map $\phi\in \CL(S, T)$ with 
	\begin{equation}\label{eqt:phi-*}
		\iota_{S^*}	\circ \phi^* = \Phi\circ \iota_{T^*}.
	\end{equation}
	On the other hand, one has
	$	\Phi^\rd\in \Morc_w\big((S^\rd)^\rd, (T^\rd)^\rd\big).$
	It follows from Proposition \ref{prop:approx-unit}(c) that there exists $\hat \Phi\in \Morc_w(S^{**}, T^{**})$ with 
	\begin{equation}\label{eqt:hat-Phi-circ-bar-iota}
		\hat \Phi \circ \overline{\iota_{S^*}^*}= \overline{\iota_{T^*}^*}\circ \Phi^\rd.
	\end{equation}
	Therefore, 
	\begin{align*}
	\phi^{**}
	 = \iota_{T^*}^*\circ\Phi^*\circ(\iota_{S^*}^*)^{-1}
	&  =\overline{\iota_{T^*}^*}\circ\iota_{(T^\rd)^*}\circ\Phi^*\circ(\iota_{S^*}^*)^{-1}\\
	& = \overline{\iota_{T^*}^*}\circ\Phi^\rd\circ\iota_{(S^\rd)^*}\circ(\iota_{S^*}^*)^{-1}
	=\hat{\Phi};
	\end{align*}
here, the first equality follows from \eqref{eqt:phi-*} (and the bijectivity of  $\iota_{S^*}^*$; see Lemma \ref{lem:dualizable}), the second one from \eqref{eqt:def-bar-iota-S*-*}, the third one from \eqref{eqt:phi-d} and the last one from both \eqref{eqt:def-bar-iota-S*-*} and \eqref{eqt:hat-Phi-circ-bar-iota}.
	Now, one knows from Corollary \ref{cor:emb-into-bidual}(a) that $\phi\in \Morc(S,T)$.
	Since $\Phi = \phi^\rd$ (because of Equalities \eqref{eqt:phi-*} and \eqref{eqt:phi-d} as well as the bijectivity of $\iota_{T^*}$), we see that the functor is full. 
	
	In order to verify that this functor is injective on objects, we suppose that $\Phi\in \Morc_w(T^\rd, S^\rd)$ is a completely isometric dual operator system isomorphism. 
	By considering $(\Phi^{-1})^\rd$, we know that 
	$$\Phi^\rd:(S^\rd)^\rd \to (T^\rd)^\rd$$ 
	is a completely isometric operator system isomorphism. 
	Thus, if $\phi:S\to T$ and $\hat \Phi:S^{**}\to T^{**}$ are as in the above, then 
	$\phi^{**} = \hat \Phi$ is also a completely isometric operator system isomorphism (by Relation \eqref{eqt:hat-Phi-circ-bar-iota} and Proposition \ref{prop:approx-unit}(c)).
	Now, we conclude that $\phi$ is a completely isometric operator system isomorphism (because of Corollary \ref{cor:emb-into-bidual}(a)).
	This show that the functor is injective on objects.  
\end{proof}

%%\medskipI

The above tells us that one may regard  the category of approximately unital complete operator systems as a full subcategory of the category of dual operator systems (actually, that of dualizable dual operator systems).  

%%\medskipI

Let us also say a few words about ``sub-objects'' of approximately unital operator systems. 
Note that one usually considers unital operator subsystems as ``sub-objects'' of unital operator systems. 

%%\medskipI

\begin{defn}\label{defn:approx-unit-subsys}
	A self-adjoint subspace $T$ of an operator system $S$ is called an \emph{approximately unital operator subsystem} if there is an increasing net in $T_+$ that weakly defines the matrix norm of $S$. 
\end{defn}

%%\medskipI

Obviously, if $S$ has an approximately unital operator subsystem, then $S$ is itself an  approximately unital operator system.

%%\medskipI

\begin{prop}\label{prop:approx-unit-subsys}
	Let $S$ be an approximately unital operator system and $T\subseteq S$ be a self-adjoint subspace. 
	
	\smnoind
	(a) If $T$ is an approximately unital operator subsystem of $S$, then $T^{**}$ contains the matrix norm defining order unit of $S^{**}$. 
	The converse holds when $T$ is complete. 
	
	\smnoind
	(b) If $T$ is an approximately unital operator subsystem of $S$, then for every $n\in \BN$ and $\phi\in \Morc(T,M_n)$, one can find $\tilde \phi\in \Morc(S,M_n)$ satisfying $\phi = \tilde \phi|_T$ and $\|\tilde \phi\| = \|\phi\|$. 
\end{prop}
\begin{proof}
	(a) For the first statement, let $\{a_i\}_{i\in \KI}$ be an increasing net in $T_+$ that weakly defines the matrix norm of $S$. 
	By the proof of $(U1)\Rightarrow(U2)$ of Theorem \ref{thm:unital-bidual}(a), there is a subnet of $\{a_i\}_{i\in \KI}$ that $\sigma(S^{**},S^*)$-converges to the matrix norm defining order unit $p$ of $S^{**}$.
	Since $a_i\in T_+$ for every $i\in \KI$, one knows that $p\in T^{**}_+$. 
	
	For the second statement, suppose that $T$ is complete and that the matrix norm defining  order unit $p$ of $S^{**}$ belongs to $T^{**}$. 
	As $p$ is the order unit of $T^{**}$, the proof of $(U2)\Rightarrow(U3)$ of Theorem \ref{thm:unital-bidual}(a) produces an approximate order unit $\{a_i\}_{i\in \KI}$ of $T_\sa$ that $\sigma(T^{**}, T^*)$-converge to $p$. 
	Now, Lemma \ref{lem:approx-unit-SMOS}(a) tells us that $\|\cdot\|_a^{w}$ coincides with $\|\cdot\|_{p} = \|\cdot\|$ on $M_\infty(S)$. 
	This shows that $\{a_i\}_{i\in \KI}$ weakly defines the matrix norm on $S$. 
	
	\smnoind
	(b) Note that $\phi$ extends to a map $\phi^{**}\in \Morc_w(T^{**}, M_n)$ with $\|(\phi^{**})^{(\infty)}\| = \|\phi^{(\infty)}\|$. 
	By part (a) above, $T^{**}$ contains the order unit $p$ of $S^{**}$. 
	If we put $\alpha:= \phi^{**}(p)$, then $\|\alpha\| = \|(\phi^{**})^{(\infty)}\| \leq 1$. 
	It follows from \cite[Lemma 5.1.6]{OPS} and the Arveson extension theorem that $\phi^{**}$ extends to a completely positive map $\bar \phi:S^{**}\to M_n$ with $\|\bar \phi^{(\infty)}\| = \|\alpha\|$. 
	Now, $\tilde \phi := \bar \phi|_{S}$  will satisfy the requirement.
\end{proof}

%%\medskipI

Because of Theorem \ref{thm:unital-bidual}(c), one may regard part (b) of the above as an analogue of \cite[Proposition 2.13]{CvS2}.
Furthermore, by \cite[Lemma 4.2(a)]{Ng-MOS}, part (b) above implies that an approximately unital operator subsystem of an approximately unital operator system $S$ is a MOS-subspace of $S$, in the sense of \cite[Definition 2.12]{Werner}. 

%%\medskipI

It was shown in \cite[Proposition 4.10(3)]{CvS2} that the non-complete operator system associated with a tolerance relation (as in \cite{CvS} and \cite{CvS2}) on a metric admits a matrix norm defining approximate order unit. 
By Theorem \ref{thm:unital-bidual}(c), we know that such an operator system is approximately unital, and so is it completion (see Theorem \ref{thm:unital-bidual}(a)). 
This gives part (a) of the following corollary. 
Moreover, parts (b), (c) and (d) of this result follow from Propositions \ref{prop:iota-bdd-below} and \ref{prop:approx-unit} as well as Theorem \ref{thm:dual-funct-approx-unit}. 

%%\medskipI

\begin{cor}\label{cor:toler-rel}
	Let $\mathcal{E}$ be the non-complete operator system associated with a tolerance relation on a path metric measure space with a measure of full support. 
	
	\smnoind
	(a) Both $\mathcal{E}$ and $\tilde{\mathcal{E}}$ are approximately unital. 
	
	\smnoind
	(b) The map $\iota_{\mathcal{E}^*}:\mathcal{E}^*\to \mathcal{E}^\rd$ (see \eqref{eqt:def-iota-incomp}) is a weak$^*$-homeomorphic (under the weak$^*$-topology $\sigma(\mathcal{E}^*, \tilde{\mathcal{E}})$) complete order isomorphism as well as an operator space isomorphism (but not necessarily completely isometric). 
	
	\smnoind
	(c) The canonical map from $(\mathcal{E}^\rd)^\rd$ to $\mathcal{E}^{**}$ is a weak$^*$-homeomorphic completely isometric complete order isomorphism.  
	
	\smnoind
	(d) If $\mathcal{F}$ is the non-complete operator system associated with another tolerance relation on another metric space such that $\mathcal{E}^\rd\cong \mathcal{F}^\rd$ completely isometrically as dual operator systems, then $\tilde{\mathcal{E}}\cong \tilde{\mathcal{F}}$ completely isometrically as operator systems.   
\end{cor}

\appendix\section{Some facts on ordered vector spaces and ordered normed spaces}\label{sec:ord-vs}
\renewcommand{\thesection}{\Alph{section}}

%%\medskipI

In this appendix, we will give some well-known notations and facts on ordered vector spaces and ordered normed spaces. 
We will present proofs for those non-trivial results that we do not find their explicit references, although all the materials in this appendix are known. 

Let us first recall the following well-known fact from \cite[Lemma 2.5]{Ng-dual-OS}. 

\begin{lem}\label{lem:image-weak-st-cont-bdd-below}
	Let $E$ and $F$ be Banach spaces.
	Let $\Phi:E^*\to F^*$ be a bounded below weak$^*$-continuous linear map. 
	Then $\Phi(E^*)$ is a weak$^*$-closed subspace of $F^*$, and $\Phi$ is a weak$^*$-homeomorphism from $E^*$ onto $\Phi(E^*)$. 
\end{lem}

%%\medskipI

Let us recall that a real vector space $E$ is an \emph{ordered vector space}  if it is equipped with a convex subset $E_+\subseteq E$ satisfying $\RP \cdot E_+\subseteq E_+$ (called the \emph{cone} of $E$). 
The cone $E_+$ is said to be 
\begin{itemize}
	\item \emph{proper} if $E_+ \cap -E_+ = \{0\}$;
	
	\item \emph{generating} if $E = E_+ - E_+$. 
\end{itemize}
We do NOT assume the cone to be proper nor generating, unless it is stated explicitly. 

%%\medskipI

\begin{defn}\label{defn:pos}
	Let  $E$ and $F$ be real ordered vector spaces.
	A linear map $\psi:E\to F$ is called  
	\begin{itemize}
		\item a \emph{positive map} if $\psi(E_+)\subseteq F_+$;  
		
		\item an \emph{order monomorphism} if $\psi$ is injective and $\psi(E_+) = \psi(E)\cap F_+$; 
		
		\item an \emph{order isomorphism} if $\psi$ is a surjective order monomorphism. 
	\end{itemize}
\end{defn}

%%\medskipI

Notice that an injection $\psi$ is an order monomorphism if and only if the cone $E_+$ is the one induced by $\psi$; i.e. $E_+ = \psi^{-1}(F_+)$. 

%%\medskipI

For a convex subset $B\subseteq E$ satisfying $-B = B$, we put
\begin{equation}\label{eqt:def-B+}
	B^+:= B\cap E_+
\end{equation}
and denote, as in \cite[p.9]{wong}, 
\begin{equation}\label{eqt:def-solid}
	\MS(B):= \{x\in E: -v \leq x \leq v, \text{ for some }v\in B^+ \}.
\end{equation}
It is easy to see that $\MS(\MS(B)) = \MS(B)$. 
Let us also recall the following definitions from pages 20 and 24 of \cite{wong}. 

%%\medskipI

\begin{defn}\label{defn:Riesz}
	A semi-norm $p$ on an ordered vector space $E$ is said to be  
	\begin{itemize}
		\item \emph{absolutely monotone} if for every $x\in E$ and $u\in E_+$ with $-u \leq x \leq u$, one has $p(x) \leq p(u)$;
		
		\item \emph{regular} (or \emph{Riesz}) if 
		$p(x) = \inf \{p(u): u\in E_+; -u\leq x\leq u \}$ $(x\in E)$. 
	\end{itemize}
\end{defn}

%%\medskipI

The following fact is easy to verify. 

%%\medskipI

\begin{lem}\label{lem:abs-mono}
	Let $E$ be an ordered vector space and $p$ be a semi-norm on $E$.
	Denote 
	$$O_p:= \{x\in E: p(x) <1\}\quad \text{and} \quad B_p:= \{x\in E: p(x) \leq 1\}.$$
	
	\smnoind
	(a) $p$ is absolutely monotone if and only if $\MS(B_p)\subseteq B_p$.
	
	\smnoind
	(b) $p$ is regular if and only if $O_p = \MS(O_p)$, which is also equivalent to 
	$O_p \subseteq \MS(B_p)\subseteq B_p.$ 
\end{lem}

%%\medskipI

%%\medskipI

\begin{defn}\label{defn:loc-decomp}
	(a) An increasing net $\{a_i\}_{i\in \KI}$ in the cone $E_+$ is called an \emph{approximate order unit} of $E$ if for every $x\in E$, there exist $t>0$ and $i\in \KI$ such that $-t a_i \leq x \leq ta_i$. 
	
	\smnoind
	(b) A (real)  \emph{ordered normed space} (respectively, \emph{ordered Banach space}) $E$ is a real normed space (respectively, Banach space) equipped with a norm closed cone $E_+$ (again not assumed to be proper nor generating). 
	
	\smnoind
	(c) An ordered normed space $E$ is said to be 
	\emph{locally decomposable} if $B_E^+ - B_E^+$ is a norm zero neighborhood, where $B_E$ is the closed unit ball of $E$ (see \eqref{eqt:def-B+}). 
\end{defn}

%%\medskipI

Our notion of local decomposability is the same as that defined in \cite{wong}. 

%%\medskipI

\begin{prop}\label{prop:rel-bdd-decomp-prop}
	Let $E$ be a real ordered normed space with its norm being absolutely monotone.
	Then $E$ is locally decomposable if and only if there exists a regular norm on $E$ that is equivalent to the original norm. 
\end{prop}
\begin{proof}
	Clearly, $B_E^+ - B_E^+\subseteq 2\MS(B_E)$. 
	On the other hand, Lemma \ref{lem:abs-mono}(a) tells us that $\MS(B_E)\subseteq B_E$, which implies $\MS(B_E)\subseteq B_E^+ - B_E^+$ (if $-v\leq x\leq v$, then $x = (v+x)/2 - (v-x)/2$). 
	Hence, $E$ is locally decomposable if and only if $\MS(B_E)$ is a norm zero neighborhood. 
	
	Let us define a semi-norm $p$ on $E$ by 
	$$p(x):=\inf \{t>0: x\in t\MS(B_E) \}\qquad (x\in E)$$
	(note that $S(B_E)$ is a convex set containing $0$).
	We have (as $\MS(B_E)\subseteq B_E$) 
	\begin{equation}\label{eqt:Op-BE}
		O_p \subseteq \MS(B_E)\subseteq B_p  \subseteq B_E. 
	\end{equation}
	Thus, $O_p\subseteq \MS(B_E) = \MS(\MS(B_E))\subseteq \MS(B_p)$ and $\MS(B_p)\subseteq \MS(B_E)\subseteq B_p$. 
	Consequently, $p$ is a regular semi-norm (because of Lemma \ref{lem:abs-mono}(b)).
By Relation \eqref{eqt:Op-BE}, if $\MS(B_E)$ is a norm zero neighborhood of $E$, then $p$ is a norm and is equivalent to the original norm on $E$. 

Conversely, suppose that $q$ is a regular norm equivalent to the original norm. 
By rescaling, we assume that there is $\kappa > 0$ with (see Lemma \ref{lem:abs-mono}(b))
$$\kappa B_E\subseteq O_q\subseteq \MS(B_q)\subseteq B_q\subseteq B_E.$$
We then learn from $\kappa B_E\subseteq \MS(B_q) = \MS(\MS(B_q)) \subseteq \MS(B_E)\subseteq B_E$ that $\MS(B_E)$ is a norm zero neighborhood. 	
\end{proof}

%%\medskipI

\begin{rem}\label{rem:dual-cone}
	%Let $E$ be  a real normed space. 
	For a subset $A\subseteq E$, we define
	$$A^\bot:= \{f\in E^*: f(x) \leq 1, \text{ for every } x\in A\}.$$
	If we equip $E^*$ with the cone $E^*_+:= -E_+^\bot$, then $E^*$ becomes an ordered Banach space.  
	In the same way, we have an ordered Banach space structure on $E^{**}$. 
\end{rem}

%%\medskipI

The following is an application of the (strong) separation theorem as well as \cite[Theorem 1.1.1]{wong} (see also the argument of \cite[Lemma 1.1.5]{wong}).

%%\medskipI

\begin{prop}\label{prop:pos-unit-ball}
	Let $E$ be a real ordered normed space with a cone $E_+$. 
	
	\smnoind
	(a) $E_+$ is weak$^*$-dense in $E^{**}_+$. 
	
	\smnoind
	(b)	$B_E^+$ is weak$^*$-dense in $B_{E^{**}}^+$. 
\end{prop}

%%\medskipI

\begin{prop}\label{prop:conv-to-unit}
	Let $E$ be a real ordered Banach space.
	Suppose that there exists $u\in B_{E^{**}}^+$ satisfying 
	$$\|x\| = \inf \{t>0: -tu \leq x \leq tu\}\qquad (x\in E^{**}).$$
	Then there is an approximate order unit $\{a_i\}_{i\in \KI}$ of $E$, which is a subnet of $O_E^+:=\{x\in E_+: \|x\| < 1\}$, that weak$^*$-converges to $u$. 
\end{prop}
\begin{proof}
	Let us first show that $E^{**}_+$ is a proper cone (and hence so is $E_+$). 
	In fact, if  $x\in E^{**}_+\cap - E^{**}_+$, then $-tu \leq 0 \leq x \leq 0 \leq tu$ for all $t\in \RP$, which implies that $\|x\| = 0$; i.e. $x= 0$.  
	
	Moreover, since the norm on $E^{**}$ is an approximate-order-unit norm in the sense of \cite{Ng69} (note that the constant net $\{u\}$ is an approximate order unit of $E^{**}$), we know from \cite[Proposition 1]{Ng69} that the open unit ball $O_{E^{**}}$ is upward-directed. 
	Therefore, \cite[Theorem 3]{Ng69} implies that the norm on $E^*$ is additive on $E^*_+$.
	We then conclude from Proposition 3 and Theorem 4 of \cite{Ng69} that both the open unit ball $O_E$ and its positive part $O_E^+$ are upward directed. 
	Therefore, $O_E^+$ is an approximate order unit of $E$ (observe that for every $x\in O_E$, one can find $y\in O_E$ with $x, -x\leq y$, which implies $y\in O_E^+$ and $-y \leq x \leq y$), and hence is a net. 
	Now, there is a subnet $\{a_i\}_{i\in \KI}$ of $O_E^+$ that weak$^*$-converges to an element $v\in B_{E^{**}}^+$. 
Then $\{a_i\}_{i\in \KI}$ is increasing and the weak$^*$-closedness of $B_{E^{**}}^+$ ensures that $a_i\leq v$ ($i\in \KI$). 
	
	The hypothesis concerning $u$ tells us that $v\leq u$. 
	Conversely, as $w\leq v$ for all $w\in O_E^+$ (note that $\{a_i\}_{i\in \KI}$ is a subnet of $O_E^+$) and $O_E^+$ is weak$^*$-dense in $B_{E^{**}}^+$ (see Proposition \ref{prop:pos-unit-ball}(b)), we know that $u\leq v$ (as $B_{E^{**}}^+$ is weak$^*$-closed).
	Finally, the properness of the cone  $E^{**}_+$ ensures that $u=v$. 	
\end{proof}

%%\medskipI

\begin{prop}[Jameson]\label{prop:Jameson}
	Let $E$ be a real ordered normed space, and $V$ be a bounded convex norm zero neighborhood of $E$. 
	
	\smnoind
	(a) $\MS(V^\bot) =\MS(V)^\bot$. 
	
	\smnoind
	(b) $\MS\big(\overline{V}^{\sigma(E^{**},E^*)}\big) = \overline{\MS(V)}^{\sigma(E^{**},E^*)}$. 
\end{prop}

%%\medskipI

In fact, part (a) of the above is a particular case of \cite[Theorem 1.1.9]{wong} (note that the Mackey topology $\tau(E,E^*)$ coincides with the norm topology on $E$).
On the other hand, part (b) follows from part (a) as well as the facts that $V^\bot$ is a bounded convex norm zero neighborhood of $E^*$ (since $V$ is bounded convex norm zero neighborhood of $E$) and that $A^{\bot\bot} = \overline{A}^{\sigma(E^{**},E^*)}$ for any convex subset $A\subseteq E$ with $0\in A$. 

%%\medskipI

\begin{defn}
	Let $F$ be a real ordered Banach space. 
	Then $F^*$ is called a \emph{dual function system} (respectively, \emph{dual quasi-function system}) if there exist a set $I$ and an isometric (respectively, a continuous) order monomorphism $\Phi:F^*\to \ell^\infty(I;\BR)$ such that $\Phi:F^*\to \Phi(F^*)$ is a weak$^*$-homeomorphism. 
\end{defn}

%%\medskipI

If $F$ is a dual quasi-function system, then $\Phi(F^*)$ is norm-closed (as it is weak$^*$-closed, because $F^*$ is weak$^*$-complete) in $\ell^\infty(I;\BR)$, and the open mapping theorem implies that $\Phi$ is bounded below.

%%\medskipI

We will consider the completion $\tilde E$ of a real ordered normed space $E$ as a real ordered Banach space with its cone $\tilde E_+$ being the norm closure of $E_+$ in $\tilde E$ (we do not assume it to be proper).

%%\medskipI

\begin{prop}\label{prop:dual-quasi-func-sys}
	Let $E$ be a real ordered normed space. 
	The following statements are equivalent. 
	\begin{enumerate}[label=\ \ \arabic*).]
		\item $\tilde E^*$ is a dual quasi-function system.
		
		\item There exist a set $I$ and a continuous positive linear map $\Phi:\tilde E^*\to \ell^\infty(I;\BR)$ such that $\Phi: \tilde E^* \to \Phi(\tilde E^*)$ is a weak$^*$-homeomorphism.

		\item $\tilde E$ is locally decomposable (see Definition \ref{defn:loc-decomp}(c)). 
		
		\item $\tilde E  = \tilde E_+ - \tilde E_+$.
		
		\item The norm closure of $B_E^+ - B_E^+$ is a norm zero neighborhood of $E$. 
	\end{enumerate}
\end{prop}
\begin{proof}
	Consider $i_{\tilde E^*}: \tilde E^*\to \ell^\infty\big(B_{\tilde E}^+\big)$ to be the map given by evaluations. 
	Obviously, $i_{\tilde E^*}$ is a positive contraction. 
	As $\tilde E$ is complete, we know that $i_{\tilde E^*}$ is weak$^*$-continuous; i.e. $i_{\tilde E^*}^*\big(\ell^1\big(B_{\tilde E}^+\big)\big)\subseteq \tilde E$.
	Moreover, the positivity of $i_{\tilde E^*}$ gives $i_{\tilde E^*}^*\big(\ell^1\big(B_{\tilde E}^+\big)_+\big)\subseteq \tilde E_+$. 
	
	\smnoind
	(1) $\Rightarrow$ (2). 
	This implication is clear, as order monomorphisms are positive. 
	
	\smnoind
	(2) $\Rightarrow$ (3). 
	Let $I$ and $\Phi$ be as in Statement (2).
	Denote 
	$$Y:= \{\omega|_{\Phi(\tilde E^*)}: \omega\in \ell^1(I;\BR) \} \quad \text{and} \quad Y_+:= \{\omega|_{\Phi(\tilde E^*)}: \omega\in \ell^1(I;\BR)_+ \}.$$ 
	We equip $Y$ with the quotient norm induced from $\ell^1(I;\BR)$. 
	Since $\Phi: \tilde E^*\to \Phi(\tilde E^*)$ is a positive weak$^*$-homeomorphism, $\Phi^*$ induces a Banach space isomorphism $\Psi: Y\to \tilde E$ satisfying $\Psi(Y_+)\subseteq \tilde E_+$.
	On the other hand, since $B_{\ell^1(I;\BR)}^+ - B_{\ell^1(I;\BR)}^+$ contains $B_{\ell^1(I;\BR)}$, we know that $B_Y^+ - B_Y^+$ is a norm zero neighborhood of $Y$. 
	These show that $B_{\tilde E}^+ - B_{\tilde E}^+$ is a norm zero neighborhood of $\tilde E$, as is required.
	
	\smnoind
	(3) $\Rightarrow$ (4). 
	This implication is clear.

	\smnoind
	(4) $\Rightarrow$ (1). 
	If $f\in \ker i_{\tilde E^*}$, then $f(\tilde E_+) = \{0\}$, and Statement (4) implies that $f=0$. 
	This shows that $i_{\tilde E^*}$ is injective.
	If $f\in \tilde E^*\setminus \tilde E^*_+$, then we can find $y\in \tilde E_+$ with $f(y) < 0$, and by considering $y/\|y\|\in B_{\tilde E}^+$, we know that $i_{\tilde E^*}(f)\not\geq 0$. 
	Consequently, $i_{\tilde E^*}$ is an order monomorphism. 
	Pick any $u \in B_{\tilde E}^+\setminus \{0\}$.
	The point-mass $\delta_u^1$ at $u$ belongs to $\ell^1(B_{\tilde E}^+)_+$ and satisfies $i_{\tilde E^*}^*(\delta_u^1) = u$.
	This gives  $i_{\tilde E^*}^*\big(\ell^1(B_{\tilde E}^+)_+\big) = \tilde E_+$.
	Statement (4) then implies that $i_{\tilde E^*}^*$ is a continuous surjection from $\ell^1(B_{\tilde E}^+)$ onto $\tilde E$, and thus, it is an open map. 
	Therefore, $i_{\tilde E^*}: \tilde E^* \to i_{\tilde E^*}(\tilde E^*)$ is a Banach space isomorphism.
	It is a folklore result (see Lemma \ref{lem:image-weak-st-cont-bdd-below}) that the weak$^*$-continuous map $i_{\tilde E^*}$ will then be a weak$^*$-homeomorphism from $\tilde E^*$ to $i_{\tilde E^*}(\tilde E^*)$ as well.  
	
	\smnoind
	(3) $\Rightarrow$ (5). 
	Let $\overline{B_E^+ - B_E^+}$  be the norm closure of $B_E^+ - B_E^+$.
	It is clear that $\overline{B_{\tilde E}^+ - B_{\tilde E}^+}\cap E\subseteq \overline{B_E^+ - B_E^+}$, and the implication follows. 
	
	\smnoind
	(5) $\Rightarrow$ (2). 
	Statement (5) implies that there exists $s > 0$ with
	\begin{equation}\label{eqt:diff-pos-ball}
		s B_E\subseteq \overline{B_E^+ - B_E^+}.
	\end{equation}
	As $B_E^+\subseteq B_{\tilde{E}}^+$, the restriction map is a normal surjective $^*$-homomorphism $\Theta: \ell^\infty(B_{\tilde{E}}^+) \to \ell^\infty(B_E^+)$. 
	Let us set $k_{\tilde E^*} = \Theta\circ i_{\tilde E^*}$. 
	Then $k_{\tilde E^*}$ is a weak$^*$-continuous positive contraction. 
	Fix any $f\in \tilde E^*$ with $\|f\| = 1$, and denote $g:= f|_E$. 
	Then $g\in E^*$ with $\|g\| = 1$.
	Relation \eqref{eqt:diff-pos-ball} implies 
	$$\sup \{ |g(u-v)|: u,v\in B_E^+ \} \geq s .$$ 
	From this, we know that $\sup \{ |g(w)|: w\in B_E^+ \} \geq s/2 $. 
	Hence, $\|k_{\tilde E^*}(f)\| \geq s/2$, and  $k_{\tilde E^*}$ is bounded below. 
	Again, Lemma \ref{lem:image-weak-st-cont-bdd-below} implies that $k_{\tilde E^*}$ is a weak$^*$-homeomorphism from $\tilde E^*$ onto $k_{\tilde E^*}(\tilde E^*)$. 
\end{proof}

%%\medskipI

\section*{Acknowledgement}

%%\medskipI

The authors are supported by the National Natural Science Foundation of China (11871285) and  by the Nankai Zhide Foundation. 
We would like to thank Prof. van Suijlikom and Prof. D'Andrea for some comments on this work. 

%%\medskipI

\end{document}